\documentclass{tac}

\usepackage{etex}  

\usepackage[T1]{fontenc}
\usepackage[latin1]{inputenc}
\usepackage{newlfont, verbatim, indentfirst, enumerate}

\usepackage{amsmath, amssymb, amscd}
\usepackage{latexsym, amsfonts, amsbsy, mathrsfs}
\usepackage{array, pst-all, graphicx, rotating}
\usepackage{tikz} \usetikzlibrary{arrows}
  \usetikzlibrary{patterns} \usetikzlibrary{decorations.markings}

\usepackage[colorlinks=true]{hyperref}
\hypersetup{allcolors=[rgb]{0.1,0.1,0.4}}

\author{Matteo Tommasini}

\thanks{I would like to thank Dorette Pronk for several interesting discussions about her work on
bicategories of fractions, and for some useful suggestions about this paper. I am also very
thankful to the anonymous referee for carefully reading the first version of this paper, and for
lots of useful suggestions, in particular for pointing out that Lemmas~\ref{lem-04},
\ref{lem-02} and~\ref{lem-01} could be used as a foundational part of this work, instead of being only
described as a side result at the end of the paper (compare with the previous version of this work
at \url{http://arxiv.org/abs/1410.3990v2}). I am also grateful to Stefano Maggiolo for his program
\emph{Commu}: it provides a simple interface for shortening the construction of diagrams
with the package \emph{tikz}. If you are interested, you can download it for
free from \url{http://poisson.phc.unipi.it/~maggiolo/Software/commu/} (only for Linux distributions
and Mac). This research was mainly performed at the Mathematics Research Unit of the University
of Luxembourg, thanks to the grant 4773242 by Fonds National de la Recherche Luxembourg.}

\title[Some insights on bicategories of fractions]
{Some insights on bicategories of fractions: representations and compositions of $2$-morphisms}

\eaddress{matteo.tommasini2@gmail.com}

\newtheorem{thm}{Theorem}

\newtheorem{prop}{Proposition}
\newtheorem{lem}{Lemma}
\newtheorem{cor}{Corollary}
\newtheoremrm{rem}{Remark}

\newcommand {\functor} [1] {\mathcal{#1}}  
\newcommand {\emphatic} [1] {\emph{#1}}

\newcommand {\thetaa} [3] {\theta_{#1,#2,#3}}  
\newcommand {\thetab} [3] {\theta_{#1,#2,#3}^{-1}}  
\newcommand {\Thetaa} [3] {\Theta_{#1,#2,#3}}  

\def \CATC {\mathbf{\mathscr{C}}} 
\def \id {\operatorname{id}}  
\def \SETW {\mathbf{W}} 
\def \SETWinv {\mathbf{W}^{-1}}  

\def \CW      {\hyperref[C]{\operatorname{C(\SETW)}}}
\def \DW      {\hyperref[D]{\operatorname{D(\SETW)}}}
\def \bf      {\hyperref[bf]{$(\operatorname{BF})$}}
\def \bfOne   {\hyperref[bfOne]{$(\operatorname{BF1})$}}
\def \bfTwo   {\hyperref[bfTwo]{$(\operatorname{BF2})$}}
\def \bfThree {\hyperref[bfThree]{$(\operatorname{BF3})$}}
\def \bfFour  {\hyperref[bfFour]{$(\operatorname{BF4})$}}
\def \bfFourA {\hyperref[bfFourA]{$(\operatorname{BF4a})$}}
\def \bfFourB {\hyperref[bfFourB]{$(\operatorname{BF4b})$}}
\def \bfFourC {\hyperref[bfFourC]{$(\operatorname{BF4c})$}}
\def \bfFive  {\hyperref[bfFive]{$(\operatorname{BF5})$}}
\newcommand {\F} [1] {\hyperref[F#1]{$(\operatorname{F#1})$}}
\def \FtenPrime {\hyperref[F10prime]{$(\operatorname{F10})'$}}

\newcounter{desccount}
\newcommand{\descitem}[2]{\item[#1]\refstepcounter{desccount}\label{#2}}

\allowdisplaybreaks

\begin{document}
\maketitle

\begin{abstract}
In this paper we investigate the construction of bicategories of fractions originally described
by D.~Pronk: given any bicategory $\CATC$ together with a suitable class of morphisms $\SETW$, one
can construct a bicategory $\CATC[\SETWinv]$, where all the morphisms of $\SETW$ are
turned into internal equivalences, and that is universal with respect to this property. Most of the
descriptions leading to this construction were long and heavily based on the axiom of choice. In
this paper we considerably simplify the description of the equivalence relation on $2$-morphisms
and the constructions of associators, vertical and horizontal
compositions in $\CATC[\SETWinv]$, thus proving that the axiom of choice is not needed under
certain conditions. The simplified description of associators and compositions
will also play a crucial role in two forthcoming papers about pseudofunctors
and equivalences between bicategories of fractions.
\end{abstract}

\section{Introduction}
In~\cite{Pr} Dorette Pronk introduced the notion of (\emph{right}) \emph{bicalculus of fractions},
generalizing to the framework of bicategories the concept of (right) calculus of
fractions, originally described by Pierre Gabriel and Michel Zisman in~\cite{GZ}
in the framework of categories.
To be more precise, given any bicategory $\CATC$ and any class $\SETW$ of
$1$-morphisms in it, one considers the following set of axioms:

\begin{description}\label{bf}
 \descitem{$(\operatorname{BF1})$:}{bfOne}
  for every object $A$ of $\CATC$, the $1$-identity $\id_A$ belongs to $\SETW$;
 \descitem{$(\operatorname{BF2})$:}{bfTwo}
  $\SETW$ is closed under compositions;
 \descitem{$(\operatorname{BF3})$:}{bfThree}
  for every morphism $\operatorname{w}:A\rightarrow B$ in $\SETW$ and for every
  morphism $f:C\rightarrow B$, there are an object $D$, a morphism $\operatorname{w}':D\rightarrow
  C$ in $\SETW$, a morphism $f':D\rightarrow A$ and an invertible $2$-morphism $\alpha:f\circ
  \operatorname{w}'\Rightarrow\operatorname{w}\circ f' $;
 \descitem{$(\operatorname{BF4})$:}{bfFour}
 \begin{enumerate}[(a)]
  \item \label{bfFourA} given any morphism $\operatorname{w}:B\rightarrow A$ in $\SETW$, any pair of
   morphisms $f^1,f^2:C\rightarrow B$ and any $2$-morphism $\alpha:\operatorname{w}\circ f^1
   \Rightarrow\operatorname{w}\circ f^2$, there are an object $D$, a morphism $\operatorname{v}:D
   \rightarrow C$ in $\SETW$ and a $2$-morphism $\beta:f^1\circ\operatorname{v}\Rightarrow f^2\circ
   \operatorname{v}$, such that the following two compositions are equal:
   
   \[
   \begin{tikzpicture}[xscale=1.5,yscale=-0.7]
    \node (A0_2) at (2, 0) {$B$};
    \node (A1_0) at (0, 1) {$D$};
    \node (A1_1) at (1, 1) {$C$};
    \node (A1_3) at (3, 1) {$A$,};
    \node (A2_2) at (2, 2) {$B$};
    
    \node (A1_2) at (2, 1) {$\Downarrow\,\alpha$};

    \path (A1_1) edge [->]node [auto] {$\scriptstyle{f^1}$} (A0_2);
    \path (A1_1) edge [->]node [auto,swap] {$\scriptstyle{f^2}$} (A2_2);
    \path (A1_0) edge [->]node [auto] {$\scriptstyle{\operatorname{v}}$} (A1_1);
    \path (A0_2) edge [->]node [auto] {$\scriptstyle{\operatorname{w}}$} (A1_3);
    \path (A2_2) edge [->]node [auto,swap] {$\scriptstyle{\operatorname{w}}$} (A1_3);
    
    \def \z {-1}
    
    \node (A0_6) at (6+\z, 0) {$C$};
    \node (A1_5) at (5+\z, 1) {$D$};
    \node (A1_7) at (7+\z, 1) {$B$};
    \node (A1_8) at (8+\z, 1) {$A$;};
    \node (A2_6) at (6+\z, 2) {$C$};

    \node (A1_6) at (6+\z, 1) {$\Downarrow\,\beta$};

    \path (A2_6) edge [->]node [auto,swap] {$\scriptstyle{f^2}$} (A1_7);
    \path (A0_6) edge [->]node [auto] {$\scriptstyle{f^1}$} (A1_7);
    \path (A1_5) edge [->]node [auto,swap] {$\scriptstyle{\operatorname{v}}$} (A2_6);
    \path (A1_5) edge [->]node [auto] {$\scriptstyle{\operatorname{v}}$} (A0_6);
    \path (A1_7) edge [->]node [auto] {$\scriptstyle{\operatorname{w}}$} (A1_8);
   \end{tikzpicture}
   \]
   
  \item\label{bfFourB} if $\alpha$ in (a) is invertible, then so is $\beta$;
  \item\label{bfFourC} if $(D',\operatorname{v}':D'\rightarrow C,\beta':f^1\circ\operatorname{v}'
   \Rightarrow f^2\circ\operatorname{v}')$ is another triple with the same properties as $(D,
   \operatorname{v},\beta)$ in (a), then there are an object $E$, a pair of morphisms
   $\operatorname{u}:E\rightarrow D$ and $\operatorname{u}':E\rightarrow D'$, and an invertible
   $2$-morphism $\zeta:\operatorname{v}\circ\operatorname{u}\Rightarrow\operatorname{v}'\circ
   \operatorname{u}'$, such that $\operatorname{v}\circ\operatorname{u}$ belongs to $\SETW$ and
   the following two compositions are equal:
   
   \[
   \begin{tikzpicture}[xscale=1.5,yscale=-0.9]
    \node (A0_1) at (1, 0) {$D$};
    \node (A1_0) at (0, 1) {$E$};
    \node (A1_2) at (2, 1) {$C$};
    \node (A1_3) at (3, 2) {$B$,};
    \node (A2_1) at (1, 2) {$D'$};
    \node (A3_2) at (2, 3) {$C$};

    \node (A1_1) at (0.9, 1) {$\Downarrow\,\zeta$};
    \node (A2_2) at (2, 2) {$\Downarrow\,\beta'$};
    
    \path (A1_2) edge [->]node [auto] {$\scriptstyle{f^1}$} (A1_3);
    \path (A3_2) edge [->]node [auto,swap] {$\scriptstyle{f^2}$} (A1_3);
    \path (A1_0) edge [->]node [auto] {$\scriptstyle{\operatorname{u}}$} (A0_1);
    \path (A1_0) edge [->]node [auto,swap] {$\scriptstyle{\operatorname{u}'}$} (A2_1);
    \path (A0_1) edge [->]node [auto] {$\scriptstyle{\operatorname{v}}$} (A1_2);
    \path (A2_1) edge [->]node [auto] {$\scriptstyle{\operatorname{v}'}$} (A1_2);
    \path (A2_1) edge [->]node [auto,swap] {$\scriptstyle{\operatorname{v}'}$} (A3_2);

    \def \z {-1}

    \node (A0_7) at (7+\z, 0) {$C$};
    \node (A1_6) at (6+\z, 1) {$D$};
    \node (A1_8) at (8+\z, 1) {$B$;};
    \node (A2_5) at (5+\z, 2) {$E$};
    \node (A2_7) at (7+\z, 2) {$C$};
    \node (A3_6) at (6+\z, 3) {$D'$};

    \node (A1_7) at (7.1+\z, 1) {$\Downarrow\,\beta$};
    \node (A2_6) at (6+\z, 2) {$\Downarrow\,\zeta$};

    \path (A0_7) edge [->]node [auto] {$\scriptstyle{f^1}$} (A1_8);
    \path (A2_7) edge [->]node [auto,swap] {$\scriptstyle{f^2}$} (A1_8);
    \path (A2_5) edge [->]node [auto] {$\scriptstyle{\operatorname{u}}$} (A1_6);
    \path (A1_6) edge [->]node [auto] {$\scriptstyle{\operatorname{v}}$} (A2_7);
    \path (A2_5) edge [->]node [auto,swap] {$\scriptstyle{\operatorname{u}'}$} (A3_6);
    \path (A1_6) edge [->]node [auto] {$\scriptstyle{\operatorname{v}}$} (A0_7);
    \path (A3_6) edge [->]node [auto,swap] {$\scriptstyle{\operatorname{v}'}$} (A2_7);
   \end{tikzpicture}
   \]
  \end{enumerate}
 \descitem{$(\operatorname{BF5})$:}{bfFive}
  if $\operatorname{w}:A\rightarrow B$ is a morphism in $\SETW$, $\operatorname{v}:A
  \rightarrow B$ is any morphism and if there is an invertible $2$-morphism $\operatorname{v}
  \Rightarrow\operatorname{w}$, then also $\operatorname{v}$ belongs to $\SETW$.
\end{description}

For simplicity of exposition, in axioms \bfFourA{} and \bfFourC{}
we omitted the associators of $\CATC$. Also in the rest of this paper we will omit all the associators
of $\CATC$, as well as the right and left unitors (except for the few cases
where we cannot ignore them
discussing coherence results); the interested reader can easily complete the proofs with the
missing associators and unitors (by coherence, any two ways of filling a diagram
of $\CATC$ with such $2$-morphisms will give the same result). In particular,
\emph{each statement in the rest of this paper 
\emphatic{(}except Corollaries~\ref{cor-01} and~\ref{cor-02}\emphatic{)}
is given without mentioning the associators of $\CATC$,
\emphatic{(}as if $\CATC$ is a $2$-category\emphatic{)}, but holds also when $\CATC$ is simply
a bicategory}.\\

A pair $(\CATC,\SETW)$ is said to \emph{admit a} (\emph{right}) \emph{bicalculus of fractions} if
all the axioms \bf{} are satisfied (actually in~\cite[\S~2.1]{Pr} condition
\bfOne{} is slightly more restrictive than the version stated above,
but it is not necessary for any of the constructions in that paper). Under these conditions,
Pronk proved that there are a
bicategory $\CATC\left[\SETWinv\right]$ (called (\emph{right}) \emph{bicategory of fractions}) and a
pseudofunctor $\functor{U}_{\SETW}:\CATC\rightarrow\CATC\left[\SETWinv\right]$, satisfying a
universal property (see~\cite[Theorem~21]{Pr}). In order to describe the main results of this
paper, we need to recall briefly the construction of $\CATC\left[\SETWinv\right]$ as described
in~\cite[\S~2]{Pr}.\\

The objects of $\CATC\left[\SETWinv\right]$ are the same as
those of $\CATC$. A morphism from $A$ to $B$ in $\CATC\left[\SETWinv\right]$ is any triple $(A',
\operatorname{w},f)$, where $A'$ is an object of $\CATC$, $\operatorname{w}:A'\rightarrow A$ is an
element of $\SETW$ and $f:A'\rightarrow B$ is a morphism of $\CATC$. In order to compose such
triples, we need to fix the following set of choices:

\begin{description}
 \descitem{$\operatorname{C}(\SETW)$:}{C}
  for every set of data in $\CATC$ as follows

  \begin{equation}\label{eq-30}

  \end{equation}
  with $\operatorname{v}$ in $\SETW$, using axiom \bfThree{} we \emph{choose} an object
  $A''$, a pair of morphisms $\operatorname{v}'$ in $\SETW$ and $f'$, and an invertible $2$-morphism
  $\rho$ in $\CATC$, as follows:

  \begin{equation}\label{eq-32}
  %
  \end{equation}
\end{description}

Then, given any pair of morphisms from $A$ to $B$ and from $B$ to $C$ in
$\CATC\left[\SETWinv\right]$ as follows

\begin{equation}\label{eq-82}
%
\end{equation}
(with both $\operatorname{w}$ and $\operatorname{v}$ in $\SETW$), one has to use the choice 
for the pair $(f,\operatorname{v})$ in the set
$\CW$, in order to get data as in
\eqref{eq-32}; after having done that, one sets $\underline{g}\circ\underline{f}:=(A'',
\operatorname{w}\circ\operatorname{v}',g\circ f')$.\\

Since axiom \bfThree{} does not ensure uniqueness in general, the set of
choices $\CW$ in general is not unique; therefore different sets of choices give rise to
different compositions of morphisms, hence to different bicategories of fractions.
Such different bicategories are equivalent by~\cite[Theorem~21]{Pr} (using the axiom of choice).\\

Given any pair of objects $A,B$ and any pair of morphisms $(A^m,\operatorname{w}^m,f^m):A\rightarrow
B$ for $m=1,2$ in $\CATC\left[\SETWinv\right]$, a $2$-morphism from $(A^1,\operatorname{w}^1,
f^1)$ to $(A^2,\operatorname{w}^2,f^2)$ is an equivalence class of data $(A^3,\operatorname{v}^1,
\operatorname{v}^2,\alpha,\beta)$ in $\CATC$ as follows

\begin{equation}\label{eq-16}

\end{equation}
such that $\operatorname{w}^1\circ\operatorname{v}^1$ belongs to $\SETW$ and such that $\alpha$ is
invertible in $\CATC$ (in~\cite[\S~2.3]{Pr} it is also required that $\operatorname{w}^2\circ
\operatorname{v}^2$ belongs to $\SETW$, but this follows from \bfFive). Any other set
of data

\begin{equation}\label{eq-77}
%
\end{equation}
(such that $\operatorname{w}^1\circ\operatorname{v}^{\prime 1}$ belongs to $\SETW$ and
$\alpha'$ is invertible) represents the same $2$-morphism in $\CATC\left[\SETWinv\right]$
if and only if there is a set of data $(A^4,\operatorname{z},\operatorname{z}',\sigma^1,
\sigma^2)$ in $\CATC$ as in the following diagram

\begin{equation}\label{eq-109}
%
\end{equation}
such that:

\begin{itemize}
 \item $\operatorname{w}^1\circ\operatorname{v}^1\circ\operatorname{z}$ belongs to $\SETW$;
 \item $\sigma^1$ and $\sigma^2$ are both invertible;
 \item the compositions of the following two diagrams are equal:
 
 \begin{equation}\label{eq-135}
  %
  \end{equation}

 \item the compositions of the following two diagrams are equal:
 
 \begin{equation}\label{eq-136}
  %
  \end{equation}
\end{itemize}

For symmetry reasons, in~\cite[\S~2.3]{Pr} it is also required that $\operatorname{w}^1\circ
\operatorname{v}^{\prime 1}\circ\operatorname{z}'$ belongs to $\SETW$, but this follows from
\bfFive, using the invertible $2$-morphism $\operatorname{w}^1
\circ\,\sigma^1:\operatorname{w}^1\circ\operatorname{v}^{\prime 1}\circ\operatorname{z}'
\Rightarrow\operatorname{w}^1\circ\operatorname{v}^1\circ\operatorname{z}$,
so we will always omit this unnecessary technical condition. The previous relation is an equivalence
relation. We denote by

\begin{equation}\label{eq-71}
\Big[A^3,\operatorname{v}^1,\operatorname{v}^2,\alpha,\beta\Big]:\Big(A^1,\operatorname{w}^1,
f^1\Big)\Longrightarrow\Big(A^2,\operatorname{w}^2,f^2\Big)
\end{equation}
the class of any data as in \eqref{eq-16}; these classes are the $2$-morphisms of the bicategory
of fractions. With an abuse of notation, we will say that ``diagram \eqref{eq-16} is equivalent to
\eqref{eq-77}'' meaning that ``the data $(A^3,\operatorname{v}^1,\operatorname{v}^2,\alpha,\beta)$
are equivalent to the data $(A^{\prime 3},\operatorname{v}^{\prime 1},
\operatorname{v}^{\prime 2},\alpha',\beta')$''. Analogously, we will say that ``\eqref{eq-16} represents
\eqref{eq-71}'' meaning that ``the data $(A^3,\operatorname{v}^1,\operatorname{v}^2,\alpha,\beta)$
represent \eqref{eq-71}''. We denote the morphisms of $\CATC\left[\SETWinv\right]$
by $\underline{f},\underline{g},\cdots$ and the $2$-morphisms by $\Gamma,\Delta,\cdots$. Even
if $\CATC$ is a $2$-category, in general $\CATC\left[\SETWinv\right]$ will only be
a bicategory (with trivial right and left unitors, but non-trivial associators if the set of choices
$\CW$ is non-unique).
From time to time we will have to take care explicitly of the associators
of $\CATC\left[\SETWinv\right]$;
whenever we will need to write them we will use the notation $\Theta^{\CATC,\SETW}_{\bullet}$.\\

Until now we have not described how associators and vertical/horizontal compositions
of $2$-morphisms are constructed in $\CATC\left[\SETWinv\right]$. For such constructions,
in~\cite[\S~2.2, \S~2.3 and Appendix]{Pr} it is required to fix an additional set of choices:

\begin{description}
 \descitem{$\operatorname{D}(\SETW)$:}{D}
  for any morphism $\operatorname{w}:B\rightarrow A$ in $\SETW$, any pair of morphisms $f^1,f^2:C\rightarrow
  B$ and any $2$-morphism $\alpha:\operatorname{w}\circ f^1\Rightarrow\operatorname{w}\circ f^2$,
  using axiom \bfFourA{} we \emph{choose} a morphism $\operatorname{v}:D\rightarrow
  C$ in $\SETW$ and a $2$-morphism $\beta:f^1\circ\operatorname{v}\Rightarrow f^2\circ
  \operatorname{v}$, such that $\alpha\circ\operatorname{v}=\operatorname{w}\circ\beta$.
\end{description}

Having fixed also this additional set of choices, the descriptions of associators
and vertical/horizontal compositions in~\cite{Pr} are very long and they do not allow
much freedom on some additional choices that are done
at each step of the construction (see the explicit descriptions in the next pages). Different
sets of choices $\DW$ (with a fixed set of choices $\CW$)
might appear to lead to different bicategories of fractions, having the same objects, morphisms,
$2$-morphisms and unitors, but (a priori) different associators and vertical/horizontal
compositions. The bicategories obtained with different sets of choices
$\DW$ (with a fixed set of choices $\CW$) are obviously 
\emph{equivalent} (using the already mentioned~\cite[Theorem~21]{Pr}), but one cannot get
better results using only the statements of~\cite{Pr}.\\

In this paper we will prove that \emph{the choices} $\DW$
\emph{are actually not necessary} because different sets of
choices $\DW$ (with a fixed set of choices $\CW$)
give \emph{the same} bicategory of fractions instead of simply equivalent ones.
In the process of proving this fact, we will also considerably simplify the descriptions
of associators and vertical/horizontal compositions in $\CATC\left[\SETWinv\right]$,
thus providing a useful set of tools for explicit computations in any bicategory of fractions.\\

In order to prove these results, we will first need a simple way of comparing
(representatives of)
associators and vertical/horizontal compositions induced by different sets of
choices $\DW$. More generally, since the current way of comparing
representatives of $2$-morphisms (based on the existence of a set of data
$(A^4,\operatorname{z},\operatorname{z}',\sigma^1,\sigma^2)$) is too long,
we will prove (and then use) the following comparison result for any pair of $2$-morphisms
with the same source and target in $\CATC\left[\SETWinv\right]$.

\begin{prop}\label{prop-05}\emph{\textbf{(comparison of $2$-morphisms in $\CATC\left[\SETWinv
\right]$)}}
Let us fix any pair $(\CATC,\SETW)$ satisfying conditions \bf, any pair
of objects $A,B$, any pair of morphisms $\underline{f}^m:=(A^m,\operatorname{w}^m,f^m):A
\rightarrow B$ for $m=1,2$ and any pair of $2$-morphisms $\Gamma,\Gamma':
\underline{f}^1\Rightarrow\underline{f}^2$ in $\CATC\left[\SETWinv\right]$.
Then there are an object $A^3$, a morphism $\operatorname{v}^1:A^3\rightarrow A^1$
in $\SETW$, a morphism $\operatorname{v}^2:A^3\rightarrow A^2$,
an invertible $2$-morphism $\alpha:
\operatorname{w}^1\circ\operatorname{v}^1\Rightarrow\operatorname{w}^2\circ\operatorname{v}^2$ and
a pair of $2$-morphisms $\gamma,\gamma':f^1\circ\operatorname{v}^1\Rightarrow f^2\circ
\operatorname{v}^2$, such that:

\begin{equation}\label{eq-04}
\Gamma=\Big[A^3,\operatorname{v}^1,\operatorname{v}^2,\alpha,\gamma\Big]\quad\textrm{and}\quad
\Gamma'=\Big[A^3,\operatorname{v}^1,\operatorname{v}^2,\alpha,\gamma'\Big].
\end{equation}

In other words, given any pair of $2$-morphisms with the same source and target in
$\CATC\left[\SETWinv\right]$, they admit representatives which only differ
\emph{(}possibly\emph{)} in the last variable.
Moreover, given any pair of $2$-morphisms $\Gamma,\Gamma':\underline{f}^1\Rightarrow
\underline{f}^2$ with representatives satisfying 
\eqref{eq-04}, the following facts are equivalent:

\begin{enumerate}[\emphatic{(}i\emphatic{)}]
 \item $\Gamma=\Gamma'$;
 \item there are an object $A^4$ and a morphism $\operatorname{z}:A^4\rightarrow A^3$
  in $\SETW$, such that $\gamma\circ\operatorname{z}=\gamma'\circ\operatorname{z}$;
 \item there are an object $A^4$ and a morphism $\operatorname{z}:A^4\rightarrow A^3$,
  such that $\operatorname{v}^1\circ\operatorname{z}$ belongs to $\SETW$ and such that
  $\gamma\circ\operatorname{z}=\gamma'\circ\operatorname{z}$.
\end{enumerate}
\end{prop}

This result implies that the universal pseudofunctor $\mathcal{U}_{\SETW}:\CATC\rightarrow
\CATC\left[\SETWinv\right]$ mentioned above in general is not $2$-full,
neither $2$-faithful, but only $2$-full and $2$-faithful ``modulo morphisms of $\SETW$''
(for more details, see Remark~\ref{rem-03}).\\

Using Proposition~\ref{prop-05} we will show some simple procedures (still long, but shorter
than the original ones in~\cite{Pr}) in order to compute:

\begin{enumerate}[(i)]
 \item \textbf{the associators of} $\CATC\left[\SETWinv\right]$: in
  Proposition~\ref{prop-01} we show how to compute associators in a way that
  does not depend on the set of choices $\DW$;
 \item \textbf{vertical composition of $2$-morphisms}:
  in Proposition~\ref{prop-02} we will prove that this construction
  does not depend on the set of choices $\CW$ nor on the set of choices
  $\DW$. In other words, vertical compositions are the same in \emph{any}
  bicategory of fractions for $(\CATC,\SETW)$ (a priori we only knew that
  objects, morphisms, $2$-morphisms and unitors are the same in any bicategory of fractions
  for $(\CATC,\SETW)$);
 \item \textbf{horizontal composition of $2$-morphisms with $1$-morphisms} on the left
  (Proposition \ref{prop-03}) and on the right (Proposition~\ref{prop-04}):
  in both cases we will prove that the composition does not depend on the set of choices
  $\DW$.
\end{enumerate}

We omit here the precise statements of the four propositions mentioned above
since they are rather long and technical.
Note that the explicit construction of \emph{representatives} for associators
and vertical/horizontal compositions in a bicategory of fractions will still depend:

\begin{itemize}
 \item on a choice of representatives for the $2$-morphisms that we want to
  compose (in Propositions~\ref{prop-02}, \ref{prop-03} and~\ref{prop-04});
 \item on some additional choices of data of $\CATC$ (these are the choices called
 \F{1} -- \F{10} in the next sections).
\end{itemize}

However, any two representatives (for an associator or a vertical/horizontal composition), even if
they are constructed in different ways and hence appear different, are actually representatives
for \emph{the same} $2$-morphism. If you need to use the constructions described in
Propositions~\ref{prop-01}, \ref{prop-02}, \ref{prop-03} or~\ref{prop-04} for associators and
vertical/horizontal compositions but you don't remember the precise statements, the
general philosophy behind all such results is the following:

\begin{enumerate}[(a)]
 \item firstly, you have to use the set of choices $\CW$ in order to construct
  the compositions on the level of $1$-morphisms of $\CATC\left[\SETWinv\right]$
  (see \eqref{eq-82});
 \item now you need to construct a representative of a $2$-morphism between the morphisms
  constructed in (a). If you need to compute associators, jump to point (c) below;
  if you need to compute compositions of the form (ii) and (iii), then the diagram that you want to
  construct can be partially filled with $2$-morphisms of $\CATC$,
  coming from any chosen representatives for the $2$-morphisms that we are composing;
 \item the missing ``holes'' in the diagram that you are constructing can
  be filled using axioms \bfThree{} and \bfFour{} finitely many times
  (an example of such an iterative procedure can be found in the remarks following each proposition
  mentioned above);
 \item in general, axioms \bfThree{} and \bfFour{} do not ensure uniqueness:
  in (c) you might get different ``fillings'' for the holes, and hence different diagrams.
  The results in Propositions~\ref{prop-01}, \ref{prop-02}, \ref{prop-03} and~\ref{prop-04}
  ensure that different choices in (b) and (c) induce diagrams representing
  \emph{the same $2$-morphism} in  $\CATC\left[\SETWinv\right]$. So you do not have to worry about
  the choices made above: you simply take the equivalence class of the diagram that you constructed, and
  you are done.
\end{enumerate}

Each horizontal composition in any bicategory can be obtained as a suitable combination
of compositions of the form (ii) and (iii).
Thus Propositions~\ref{prop-02}, \ref{prop-03} and~\ref{prop-04}
prove immediately that \emph{horizontal compositions in $\CATC\left[\SETWinv\right]$ do not depend
on the set of choices} $\DW$
(actually, by the mentioned $3$ propositions, one gets easily that each
horizontal composition depends on at most $3$ choices of type $\CW$).
This, together with (i) and (ii), implies at once the main
result of this paper:

\begin{thm}\label{thm-01}\emph{\textbf{(the structure of $\CATC\left[\SETWinv\right]$)}}
Let us fix any pair $(\CATC,\SETW)$ satisfying conditions \emphatic{\bf}. Then the
construction of $\CATC\left[\SETWinv\right]$ depends only on  the set of choices
$\CW$, i.e., different sets of choices $\CW$
\emphatic{(}for any fixed set of choices $\CW$\emphatic{)}
give the \emph{same} bicategory of fractions, instead of only equivalent ones.
\end{thm}

In particular, we get immediately:

\begin{cor}\label{cor-03}
Let us suppose that for each pair $(f,\operatorname{v})$ with $\operatorname{v}$ in $\SETW$ as in
\eqref{eq-30} there is a \emph{unique choice}
of $(A'',\operatorname{v}',f',\rho)$ as in $\CW$. Then the
construction of $\CATC\left[\SETWinv\right]$ does not depend on the axiom of choice. The same
result holds if ``unique choice'' above is replaced by ``canonical choice'' \emphatic{(}i.e., when
there is a canonical choice of pullback diagrams or iso comma squares for axiom
\emphatic{\bfThree}\emphatic{)}.
\end{cor}

In Section~\ref{sec-03} we will show some explicit applications of the results mentioned
so far. In particular, we will describe a simple procedure for
checking the invertibility of a $2$-morphism in any bicategory of fractions:

\begin{prop}\label{prop-06}\emph{\textbf{(invertibility of $2$-morphisms
in $\CATC\left[\SETWinv\right]$)}}
Let us fix any pair $(\CATC,\SETW)$ satisfying conditions \emphatic{\bf},
any pair of morphisms $\underline{f}^m:=(A^m,\operatorname{w}^m,f^m):A\rightarrow B$ in
$\CATC\left[\SETWinv\right]$ for $m=1,2$ and any $2$-morphism $\Gamma:\underline{f}^1\Rightarrow
\underline{f}^2$. Then the following facts are equivalent:

\begin{enumerate}[\emphatic{(}i\emphatic{)}]
 \item $\Gamma$ is invertible in $\CATC\left[\SETWinv\right]$;
 \item $\Gamma$ has a representative as in \eqref{eq-16},
  such that $\beta$ is invertible in $\CATC$;
 \item given any representative \eqref{eq-16} for $\Gamma$, there are an object
  $A^4$ and a morphism $\operatorname{u}:A^4\rightarrow A^3$ in $\SETW$, such that
  $\beta\circ\operatorname{u}$ is invertible in $\CATC$;
 \item given any representative \eqref{eq-16} for $\Gamma$, there are an object
  $A^4$ and a morphism $\operatorname{u}:A^4\rightarrow A^3$, such that
  $\operatorname{v}^1\circ\operatorname{u}$ belongs to $\SETW$
  and such that $\beta\circ\operatorname{u}$ is invertible in $\CATC$.
\end{enumerate}
\end{prop}

Note that in \eqref{eq-16} $\alpha$ is always invertible by the definition of $2$-morphism in
$\CATC\left[\SETWinv\right]$.\\

As another application, we obtain a nice description of some associators in
$\CATC\left[\SETWinv\right]$ in terms of the associators in the original bicategory
$\CATC$. For the precise statement we refer directly to Corollary~\ref{cor-02}.\\

We are going to apply all the results mentioned so far in the next two papers~\cite{T4}
and~\cite{T5}, where we will investigate the problem of constructing
pseudofunctors and equivalences between bicategories of fractions.\\

In the Appendix of this paper we will give an alternative description of $2$-morphisms
in any bicategory of fraction, inspired by the results of Proposition~\ref{prop-05}. This
alternative description is not further exploited in this paper since
it does not interact well with vertical and horizontal compositions;
it does seem interesting for other
purposes as it is considerably simpler than the description given above:
instead of having classes of equivalence where each representative is a collection of
$5$ data $(A^3,\operatorname{v}^1,\operatorname{v}^2,\alpha,\beta)$ as in \eqref{eq-16},
we have classes of equivalence where each representative is a collection
of $3$ data of $\CATC$. In addition, the equivalence relation in this alternative
description is much simpler than the one recalled above.

\section{Notations and basic lemmas}\label{sec-02}
We mainly refer to~\cite{L} and~\cite[\S~1]{PW} for a general overview on bicategories,
pseudofunctors (i.e., homomorphisms of bicategories), lax natural transformations and modifications.
For simplicity of exposition, each composition
of $1$-morphisms and $2$-morphisms will be
denoted by $\circ$, both in $\CATC$ and in $\CATC\left[\SETWinv\right]$.\\

In the rest of this paper we will often use the following four easy lemmas. Even if each of
them is not difficult for experts in this area, they may be harder for inexperienced readers; so
we give a detailed proof for each of them.

\begin{lem}\label{lem-03}
Let us fix any pair $(\CATC,\SETW)$ satisfying conditions \emphatic{\bf}.
Let us fix any morphism $\operatorname{w}:B\rightarrow A$ in $\SETW$, any pair of morphisms
$f^1,f^2:C\rightarrow B$ and any pair of $2$-morphisms $\gamma,\gamma':f^1\Rightarrow f^2$, such
that $\operatorname{w}\circ\,\gamma=\operatorname{w}\circ\,\gamma'$. Then there are an object $E$
and a morphism $\operatorname{u}:E\rightarrow C$ in $\SETW$, such that $\gamma\circ
\operatorname{u}=\gamma'\circ\operatorname{u}$.
\end{lem}

\begin{proof}
We set:

\[\alpha:=\operatorname{w}\circ\,\gamma:\,\,\operatorname{w}\circ f^1\Longrightarrow
\operatorname{w}\circ f^2.\]

Then condition \bfFourA{} is obviously satisfied by the set of data:

\[D:=C,\quad\operatorname{v}:=\id_C,\quad\beta:=\gamma:\,
f^1\circ\operatorname{v}\Longrightarrow f^2\circ\operatorname{v}.\]

Since we have also $\alpha=\operatorname{w}\circ\,\gamma'$ by hypothesis,
\bfFourA{} is also satisfied by the data:

\[D':=C,\quad\operatorname{v}':=\id_C,\quad\beta':=\gamma':\,
f^1\circ\operatorname{v}'\Longrightarrow f^2\circ\operatorname{v}'.\]

Then by \bfFourC{} there are an object $E$, a pair of morphisms
$\operatorname{u},\operatorname{u}':E\rightarrow C$ (with $\operatorname{u}$ in $\SETW$)
and an invertible $2$-morphism $\zeta:\operatorname{u}\Rightarrow\operatorname{u}'$, such that:

\[\Big(\gamma'\circ\operatorname{u}'\Big)\circ\Big(f^1\circ\zeta\Big)=
\Big(f^2\circ\zeta\Big)\circ\Big(\gamma\circ\operatorname{u}\Big).\]

Using the coherence axioms on the bicategory $\CATC$ and the fact that $\zeta$ is
invertible, this implies that $\gamma\circ
\operatorname{u}=\gamma'\circ\operatorname{u}$.
\end{proof}

The next two lemmas prove that if conditions \bf{} hold, then conditions
\bfThree, \bfFourA{} and \bfFourB{} hold under less
restrictive conditions on the morphism $\operatorname{w}$. To be more precise, instead of
imposing that $\operatorname{w}$ belongs to $\SETW$, it is sufficient
to impose that $\operatorname{z}\circ\operatorname{w}$ belongs to $\SETW$ for some
morphism $\operatorname{z}$ in $\SETW$ (as a special case, one gets back again 
\bfThree, \bfFourA{} and \bfFourB{} when we choose
$\operatorname{z}$ as a $1$-identity).

\begin{lem}\label{lem-05}
Let us fix any pair $(\CATC,\SETW)$ satisfying conditions \emphatic{\bf}.
Let us choose any quadruple of objects $A,B,B',C$ and any triple of morphisms $\operatorname{w}:
A\rightarrow B$, $\operatorname{z}:B\rightarrow B'$ and $f:C\rightarrow B$, such that both
$\operatorname{z}$ and $\operatorname{z}\circ\operatorname{w}$ belong to $\SETW$. Then there
are an object $D$, a morphism $\operatorname{w}'$ in $\SETW$, a morphism $f'$
and an invertible $2$-morphism $\alpha$ as follows:

\[
\begin{tikzpicture}[xscale=1.2,yscale=-0.8]
    \node (A0_0) at (0, 0) {$D$};
    \node (A0_2) at (2, 0) {$A$};
    \node (A2_0) at (0, 2) {$C$};
    \node (A2_2) at (2, 2) {$B$};
    \node (A2_4) at (3.8, 2) {$B'$.};

    \node (A1_1) at (1.2, 1) {$\alpha$};
    \node (B1_1) [rotate=315] at (0.9, 1) {$\Uparrow$};

    \path (A2_2) edge [->,dashed]node [auto,swap] {$\scriptstyle{\operatorname{z}}$} (A2_4);
    \path (A2_0) edge [->]node [auto,swap] {$\scriptstyle{f}$} (A2_2);
    \path (A0_0) edge [->]node [auto,swap] {$\scriptstyle{\operatorname{w}'}$} (A2_0);
    \path (A0_2) edge [->]node [auto] {$\scriptstyle{\operatorname{w}}$} (A2_2);
    \path (A0_0) edge [->]node [auto] {$\scriptstyle{f'}$} (A0_2);
\end{tikzpicture}
\]
\end{lem}

\begin{proof}
Since $\operatorname{z}\circ\operatorname{w}$ belongs to $\SETW$, we can
apply axiom \bfThree{} to the pair of morphisms $(\operatorname{z}\circ f,
\operatorname{z}\circ\operatorname{w})$, so we get an object $E$, a morphism
$\operatorname{t}:E\rightarrow C$ in $\SETW$, a morphism $g:E\rightarrow A$ and an invertible
$2$-morphism $\gamma:\operatorname{z}\circ f\circ\operatorname{t}\Rightarrow\operatorname{z}
\circ\operatorname{w}\circ\,g$. Since $\operatorname{z}$ belongs to $\SETW$, we can
apply \bfFourA{} and \bfFourB{}
to the invertible $2$-morphism $\gamma$.
So there are an object $D$, a morphism $\operatorname{r}:D\rightarrow E$ in $\SETW$ and an
invertible $2$-morphism $\alpha:f\circ\operatorname{t}\circ\operatorname{r}\Rightarrow
\operatorname{w}\circ\,g\circ\operatorname{r}$, such that
$\gamma\circ\operatorname{r}=\operatorname{z}\circ\,\alpha$.
Then we set $f':=g\circ\operatorname{r}:D\rightarrow A$ and
$\operatorname{w}':=\operatorname{t}\circ\operatorname{r}:D\rightarrow C$; this last morphism
belongs to $\SETW$ by construction and \bfTwo. This suffices to conclude.
\end{proof}

\begin{lem}\label{lem-06}
Let us fix any pair $(\CATC,\SETW)$ satisfying conditions \emphatic{\bf}.
Let us choose any quadruple of objects $A,A',B,C$ and any quadruple of morphisms
$\operatorname{w}:B\rightarrow A$, $\operatorname{z}:A\rightarrow A'$ and $f^1,f^2:C\rightarrow
B$, such that both $\operatorname{z}$ and $\operatorname{z}\circ\operatorname{w}$ belong to
$\SETW$. Moreover, let us fix any $2$-morphism $\alpha:\operatorname{w}\circ f^1\Rightarrow
\operatorname{w}\circ f^2$. Then there are an object $D$, a morphism $\operatorname{v}:
D\rightarrow C$ in $\SETW$ and a $2$-morphism $\beta:f^1\circ\operatorname{v}\Rightarrow
f^2\circ\operatorname{v}$, such that $\alpha\circ\operatorname{v}=\operatorname{w}\circ\beta$.
Moreover, if $\alpha$ is invertible, so is $\beta$.
\end{lem}

\[
\begin{tikzpicture}[xscale=1.5,yscale=-1.0]
    \node (A0_2) at (2, 0) {$B$};
    \node (A1_0) at (0, 1) {$D$};
    \node (A1_1) at (1, 1) {$C$};
    \node (A1_3) at (3, 1) {$A$};
    \node (A1_4) at (4, 1) {$A'$,};
    \node (A2_2) at (2, 2) {$B$};

    \node (A1_2) at (2, 1) {$\Downarrow\,\alpha$};
    
    \path (A1_1) edge [->]node [auto] {$\scriptstyle{f^1}$} (A0_2);
    \path (A1_0) edge [->]node [auto] {$\scriptstyle{\operatorname{v}}$} (A1_1);
    \path (A2_2) edge [->]node [auto,swap] {$\scriptstyle{\operatorname{w}}$} (A1_3);
    \path (A1_1) edge [->]node [auto,swap] {$\scriptstyle{f^2}$} (A2_2);
    \path (A0_2) edge [->]node [auto] {$\scriptstyle{\operatorname{w}}$} (A1_3);
    \path (A1_3) edge [->,dashed]node [auto] {$\scriptstyle{\operatorname{z}}$} (A1_4);
    
    \def \z {5}
    
    \node (B0_1) at (1+\z, 0) {$C$};
    \node (B1_0) at (0+\z, 1) {$D$};
    \node (B1_2) at (2+\z, 1) {$B$};
    \node (B2_1) at (1+\z, 2) {$C$};
    \node (B1_3) at (3+\z, 1) {$A$};
    \node (B1_4) at (4+\z, 1) {$A'$.};

    \node (B1_1) at (1+\z, 1) {$\Downarrow\,\beta$};
    
    \path (B1_2) edge [->]node [auto] {$\scriptstyle{\operatorname{w}}$} (B1_3);
    \path (B1_3) edge [->,dashed]node [auto] {$\scriptstyle{\operatorname{z}}$} (B1_4);
    \path (B0_1) edge [->]node [auto] {$\scriptstyle{f^1}$} (B1_2);
    \path (B1_0) edge [->]node [auto] {$\scriptstyle{\operatorname{v}}$} (B0_1);
    \path (B1_0) edge [->]node [auto,swap] {$\scriptstyle{\operatorname{v}}$} (B2_1);
    \path (B2_1) edge [->]node [auto,swap] {$\scriptstyle{f^2}$} (B1_2);
\end{tikzpicture}
\]

\begin{proof}
Since $\operatorname{z}\circ\operatorname{w}$ belongs to $\SETW$,
we can apply \bfFourA{} to the $2$-morphism

\[\operatorname{z}\circ\,\alpha:\,(\operatorname{z}\circ\operatorname{w})
\circ f^1\Longrightarrow(\operatorname{z}\circ\operatorname{w})\circ f^2.\]

So there are an object $E$, a morphism $\operatorname{t}:E\rightarrow C$ in $\SETW$ and a
$2$-morphism $\gamma:f^1\circ\operatorname{t}\Rightarrow f^2\circ\operatorname{t}$, such that:

\[\operatorname{z}\circ\,\alpha\circ\operatorname{t}=\operatorname{z}
\circ\operatorname{w}\circ\,\gamma.\]

Since $\operatorname{z}$ belongs to $\SETW$, by Lemma~\ref{lem-03}
there are an object $D$ and a morphism $\operatorname{r}:
D\rightarrow E$ in $\SETW$, such that:

\begin{equation}\label{eq-01}
\alpha\circ\operatorname{t}\circ\operatorname{r}=\operatorname{w}\circ\,\gamma\circ\operatorname{r}.
\end{equation}

Then we define $\operatorname{v}:=\operatorname{t}\circ\operatorname{r}:D\rightarrow C$; this
morphism belongs to $\SETW$ by construction and \bfTwo. Moreover, we set
$\beta:=\gamma\circ\operatorname{r}:f^1\circ\operatorname{v}\Rightarrow f^2\circ\operatorname{v}$.
Then from \eqref{eq-01} we get that $\alpha\circ\operatorname{v}=\operatorname{w}\circ\,\beta$.
Moreover, if $\alpha$ is invertible, then by \bfFourB{} so is $\gamma$, hence so is
$\beta$.
\end{proof}

\begin{lem}\label{lem-11}
Let us fix any pair $(\CATC,\SETW)$ satisfying conditions \emphatic{\bf},
any triple of objects $A,B,C$ and any pair of morphisms $\operatorname{w}:
C\rightarrow B$ and $\operatorname{z}:B\rightarrow A$, such that both
$\operatorname{z}$ and $\operatorname{z}\circ\operatorname{w}$ belong to $\SETW$.
Then there are an object $D$ and a morphism $\operatorname{v}$ as below,
such that $\operatorname{w}\circ\operatorname{v}$ belongs to $\SETW$:
\end{lem}

\[
\begin{tikzpicture}[xscale=1.8,yscale=-1.2]
    \node (A0_0) at (0, 0) {$D$};
    \node (A0_1) at (1, 0) {$C$};
    \node (A0_2) at (2, 0) {$B$};
    \node (A0_3) at (3, 0) {$A$.};

    \path (A0_0) edge [->]node [auto] {$\scriptstyle{\operatorname{v}}$} (A0_1);
    \path (A0_1) edge [->]node [auto] {$\scriptstyle{\operatorname{w}}$} (A0_2);
    \path (A0_2) edge [->]node [auto] {$\scriptstyle{\operatorname{z}}$} (A0_3);
\end{tikzpicture}
\]

\begin{proof}
First of all, we apply axiom \bfThree{} on the pair $(\operatorname{z},
\operatorname{z}\circ\operatorname{w})$. Since $\operatorname{z}\circ\operatorname{w}$
belongs to $\SETW$ by hypothesis, there are an object $R$, a morphism
$\operatorname{p}$ in $\SETW$, a morphism $\operatorname{q}$ and an invertible
$2$-morphism $\alpha$ as below:

\[
\begin{tikzpicture}[xscale=1.5,yscale=-0.8]
    \node (A0_1) at (1, 0) {$R$};
    \node (A1_0) at (0, 2) {$B$};
    \node (A1_2) at (2, 2) {$C$.};
    \node (A2_1) at (1, 2) {$A$};

    \node (A1_1) at (1, 1) {$\alpha$};
    \node (B1_1) at (1, 1.4) {$\Rightarrow$};
    
    \path (A1_2) edge [->]node [auto] {$\scriptstyle{\operatorname{z}
      \circ\operatorname{w}}$} (A2_1);
    \path (A0_1) edge [->]node [auto] {$\scriptstyle{\operatorname{q}}$} (A1_2);
    \path (A1_0) edge [->]node [auto,swap] {$\scriptstyle{\operatorname{z}}$} (A2_1);
    \path (A0_1) edge [->]node [auto,swap] {$\scriptstyle{\operatorname{p}}$} (A1_0);
\end{tikzpicture}
\]

Now let us apply axioms \bfFourA{} and \bfFourB{}
to the invertible $2$-morphism $\alpha:\operatorname{z}\circ\operatorname{p}
\Rightarrow\operatorname{z}\circ(\operatorname{w}\circ\operatorname{q})$. Since
$\operatorname{z}$ belongs to $\SETW$ by hypothesis,
there are an object $D$, a morphism $\operatorname{r}:D\rightarrow R$
in $\SETW$ and an invertible $2$-morphism $\beta:\operatorname{p}\circ\operatorname{r}
\Rightarrow\operatorname{w}\circ\operatorname{q}\circ\operatorname{r}$, such
that $\alpha\circ\operatorname{r}=\operatorname{z}\circ\beta$.
By construction both $\operatorname{p}$ and $\operatorname{r}$ belong to $\SETW$, hence
also $\operatorname{p}\circ\operatorname{r}$ belongs to $\SETW$ by axiom
\bfTwo. Since $\beta$ is invertible, then by 
\bfFive{} we conclude that also $\operatorname{w}\circ\operatorname{q}
\circ\operatorname{r}$ belongs to $\SETW$. Then in order
to conclude it suffices to define
$\operatorname{v}:=\operatorname{q}\circ\operatorname{r}:D\rightarrow C$.
\end{proof}

\section{Comparison of \texorpdfstring{$2$}{2}-morphisms in a bicategory of fractions}\label{sec-04}
In this section we are going to prove Proposition~\ref{prop-05}.
First of all, we need the following lemma. In most of this paper this result will
be applied with $A=A'$ and $\operatorname{z}=\id_{A}$, but we need to state
it in this more general form since the presence of a non-trivial $\operatorname{z}$
will be crucial in a couple of points in the paper.

\begin{lem}\label{lem-04}
Let us fix any pair $(\CATC,\SETW)$ satisfying conditions \emphatic{\bf}, and
any set of data as follows in $\CATC$, such that $\operatorname{z}$,
$\operatorname{z}\circ\operatorname{w}^1$, $\operatorname{z}\circ
\operatorname{w}^2$, $\operatorname{z}\circ
\operatorname{w}^1\circ\operatorname{p}$ and 
$\operatorname{z}\circ\operatorname{w}^1\circ\operatorname{r}$ all belong to $\SETW$,
and such that $\varsigma$ and $\eta$ are both invertible:

\[

\]

Then there are:

\begin{enumerate}[\emphatic{(}1\emphatic{)}]
 \item a pair of objects $G$ and $A^3$;
 \item a morphism $\operatorname{t}^1:G\rightarrow E$, such that $\operatorname{p}\circ
  \operatorname{t}^1$ belongs to $\SETW$;
 \item a morphism $\operatorname{t}^2:G\rightarrow F$;
 \item a morphism $\operatorname{t}^3:A^3\rightarrow G$ in $\SETW$;
 \item a pair of invertible $2$-morphisms $\varepsilon$ and $\kappa$ in $\CATC$ as follows:
\end{enumerate}

\[
%
\]
such that $\varsigma\circ\operatorname{t}^1\circ\operatorname{t}^3$ is equal to the
composition of the following diagram:

\begin{equation}\label{eq-104}
%
\end{equation}
\end{lem}

The choice of symbols in Lemma~\ref{lem-04} may seem curious at first. It makes
sense once we combine this lemma with Lemma~\ref{lem-02} below.

\begin{proof}
\textbf{Step 1.} First of all, let us prove the special case when $A=A'$ and
$\operatorname{z}=\id_{A}$. In this case the hypotheses imply that
$\operatorname{w}^1$, $\operatorname{w}^2$, $\operatorname{w}^1\circ\operatorname{p}$ and 
$\operatorname{w}^1\circ\operatorname{r}$ all belong to $\SETW$. Since both
$\operatorname{w}^1$ and $\operatorname{w}^1\circ\operatorname{p}$ belong to
$\SETW$, by Lemma~\ref{lem-11} there are an object $E'$ and a morphism
$\operatorname{v}:E'\rightarrow E$, such that $\operatorname{p}\circ\operatorname{v}$
belongs to $\SETW$.\\

Since both $\operatorname{w}^1$ and $\operatorname{w}^1\circ\operatorname{r}$ belong
to $\SETW$, by Lemma~\ref{lem-05} there are an object $G$,
a morphism $\widetilde{\operatorname{t}}^{\,1}:G\rightarrow E'$ in $\SETW$, a morphism
$\operatorname{t}^2:G\rightarrow F$ and an invertible $2$-morphism $\varepsilon$ as follows:

\[
\begin{tikzpicture}[xscale=1.2,yscale=-0.8]
    \node (A0_1) at (1, 0) {$G$};
    \node (A1_0) at (0, 2) {$E$};
    \node (K1_0) at (-1, 2) {$E'$};
    \node (A1_2) at (3, 2) {$F$.};
    \node (A2_1) at (1, 2) {$A^1$};

    \node (A1_1) at (1, 1) {$\varepsilon$};
    \node (B1_1) at (1, 1.4) {$\Rightarrow$};
    
    \path (K1_0) edge [->]node [auto,swap] {$\scriptstyle{\operatorname{v}}$} (A1_0);
    \path (A1_2) edge [->]node [auto] {$\scriptstyle{\operatorname{r}}$} (A2_1);
    \path (A0_1) edge [->]node [auto] {$\scriptstyle{\operatorname{t}^2}$} (A1_2);
    \path (A1_0) edge [->]node [auto,swap] {$\scriptstyle{\operatorname{p}}$} (A2_1);
    \path (A0_1) edge [->]node [auto,swap]
      {$\scriptstyle{\widetilde{\operatorname{t}}^1}$} (K1_0);
\end{tikzpicture}
\]

We set $\operatorname{t}^1:=\operatorname{v}\circ\,\widetilde{\operatorname{t}}^{\,1}:G
\rightarrow E$. By construction $\operatorname{p}\circ\operatorname{v}$
and $\widetilde{\operatorname{t}}^{\,1}$ both belong to $\SETW$; so by \bfTwo{}
the morphism $\operatorname{p}\circ\operatorname{t}^1=
\operatorname{p}\circ\operatorname{v}\circ\,\widetilde{\operatorname{t}}^{\,1}$
belongs to $\SETW$, so condition (2) is verified.\\

By hypothesis and construction, both $\operatorname{w}^1$ and $\operatorname{p}
\circ\operatorname{t}^1$ belong to $\SETW$, hence $\operatorname{w}^1
\circ\operatorname{p}\circ\operatorname{t}^1$ also belongs to $\SETW$.
So using axiom \bfFive{} on the invertible $2$-morphism

\[\varsigma^{-1}\circ\operatorname{t}^1:\,\operatorname{w}^2\circ\operatorname{q}
\circ\operatorname{t}^1\Longrightarrow\operatorname{w}^1\circ\operatorname{p}
\circ\operatorname{t}^1,\]
we get that $\operatorname{w}^2\circ\operatorname{q}\circ
\operatorname{t}^1$ also belongs to $\SETW$.
Since $\operatorname{w}^2$ belongs to $\SETW$ by hypothesis,
using Lemma~\ref{lem-05} we conclude that there are an object $H$, a
morphism $\operatorname{z}^1:H\rightarrow G$ in $\SETW$,
a morphism $\operatorname{z}^2:H\rightarrow G$ and an invertible $2$-morphism
$\phi$ as follows:

\[

\]

Now let us denote by $\alpha:(\operatorname{w}^2\circ\operatorname{q}\circ
\operatorname{t}^1)\circ\operatorname{z}^2\Rightarrow(\operatorname{w}^2\circ\operatorname{q}\circ
\operatorname{t}^1)\circ\operatorname{z}^1$ the following composition:

\begin{equation}\label{eq-108}
%
\end{equation}

Since all the $2$-morphisms above are invertible, so is $\alpha$; moreover we already said
that $\operatorname{w}^2\circ\operatorname{q}\circ\operatorname{t}^1$ belongs to $\SETW$.
Therefore using axioms
\bfFourA{} and \bfFourB{} on $\alpha$, we get an object $A^3$,
a morphism $\operatorname{z}^3:A^3\rightarrow H$ in $\SETW$ and an invertible $2$-morphism
$\beta:\operatorname{z}^2\circ\operatorname{z}^3\Rightarrow
\operatorname{z}^1\circ\operatorname{z}^3$, such that $\alpha\circ\operatorname{z}^3=
(\operatorname{w}^2\circ\operatorname{q}\circ\operatorname{t}^1)\circ\,\beta$.
The previous identity together with \eqref{eq-108} implies that the following composition

\begin{equation}\label{eq-110}

\end{equation}
is equal to the following one:

\[
%
\]

Since both $\operatorname{z}^1$ and $\operatorname{z}^3$ belong to $\SETW$ by construction,
so does $\operatorname{t}^3:=\operatorname{z}^1\circ\operatorname{z}^3:
A^3\rightarrow G$ by \bfTwo.
Then in order to conclude it suffices to define $\kappa$ as the composition 
of $\phi$ and $\beta$ as on the left hand side of \eqref{eq-110}. So we have proved the lemma
in the special case when $A=A'$ and $\operatorname{z}=\id_A$.\\

\textbf{Step 2.} Let us prove the general case: for that, we consider
the diagrams:

\[
%
\]

Using Step 1, there are:

\begin{enumerate}[(1)]
 \item a pair of objects $G$ and $\widetilde{A}^3$;
 \item a morphism $\operatorname{t}^1:G\rightarrow E$, such that $\operatorname{p}\circ
   \operatorname{t}^1$ belongs to $\SETW$;
 \item a morphism $\operatorname{t}^2:G\rightarrow F$;
 \item a morphism $\widetilde{\operatorname{t}}^{\,3}:\widetilde{A}^3\rightarrow G$ in $\SETW$;
 \item a pair of invertible $2$-morphisms $\varepsilon$ and $\widetilde{\kappa}$
   in $\CATC$ as follows:
\end{enumerate}

\[
%
\]
such that $\operatorname{z}\circ\,\varsigma\circ\operatorname{t}^1
\circ\,\widetilde{\operatorname{t}}^{\,3}$ is equal to the
composition of the following diagram:

\[
%
\]

Since $\operatorname{z}$ belongs to $\SETW$ by hypothesis,
we can apply Lemma~\ref{lem-03} to the previous identity. So there are an object $A^3$
and a morphism $\operatorname{u}:A^3\rightarrow\widetilde{A}^3$ in $\SETW$, 
such that $\varsigma\circ\operatorname{t}^1
\circ\,\widetilde{\operatorname{t}}^{\,3}\circ\operatorname{u}$
coincides with the following composition:

\[
%
\]

Then it suffices to set $\operatorname{t}^3:=\widetilde{\operatorname{t}}^{\,3}\circ
\operatorname{u}$ and $\kappa:=\widetilde{\kappa}\circ\operatorname{u}$ in order to prove
the general case.
\end{proof}

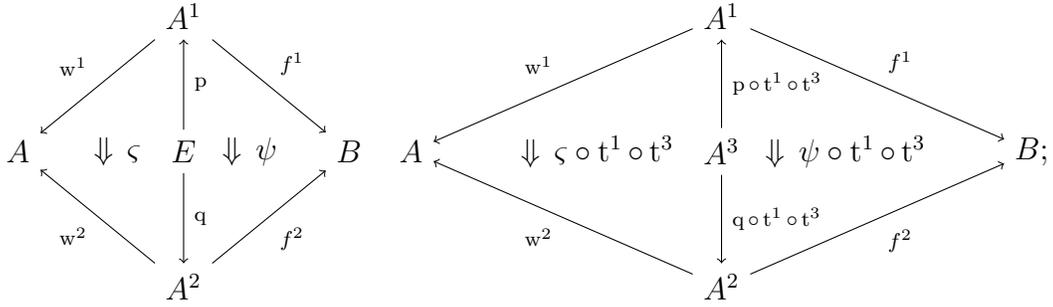
\begin{lem}\label{lem-02}
Let us fix any pair $(\CATC,\SETW)$ satisfying conditions \emphatic{\bf}, any pair
of objects $A,B$, any pair of morphisms $\underline{f}^m:=(A^m,\operatorname{w}^m,f^m):A
\rightarrow B$ for $m=1,2$, and any pair of $2$-morphisms $\Gamma,\Gamma':
\underline{f}^1\Rightarrow\underline{f}^2$ in $\CATC\left[\SETWinv\right]$.
Let us fix \emph{any} representative $(E,\operatorname{p},\operatorname{q},
\varsigma,\psi)$ for $\Gamma$ and $(F,\operatorname{r},\operatorname{s},\eta,\mu)$ for
$\Gamma'$ as follows:

\begin{equation}\label{eq-10}

\end{equation}

Let us fix any set of data \emphatic{(}1\emphatic{)} -- \emphatic{(}5\emphatic{)} as in
Lemma~\ref{lem-04}. Then $\Gamma$, respectively $\Gamma'$, has a representative as
follows:

\begin{equation}\label{eq-111}
%
\end{equation}
where $\varphi$ is the following composition:

\begin{equation}\label{eq-106}
%
\end{equation}
\end{lem}

\begin{proof}
Since $\underline{f}^1$ and $\underline{f}^2$ are morphisms in $\CATC\left[\SETWinv\right]$,
$\operatorname{w}^1$ and $\operatorname{w}^2$ belong to $\SETW$. Moreover, since
$\Gamma$ and $\Gamma'$ are $2$-morphisms in such a bicategory, $\operatorname{w}^1
\circ\operatorname{p}$ and $\operatorname{w}^1\circ\operatorname{r}$ also
belong to $\SETW$. So a set
of data (1) -- (5) as in Lemma~\ref{lem-04} exists. Now we claim that:

\begin{enumerate}[(a)]
 \item the following two diagrams are equivalent, i.e., they represent the same $2$-morphism of
  $\CATC\left[\SETWinv\right]$:

    \[

    \]
 
 \item the following two diagrams are equivalent:
 
    \[
    %
    \]
\end{enumerate}

In order to prove such claims, let us consider the following set of data:

\begin{equation}\label{eq-79}
%
\end{equation}

Then claim (a) follows at once using the set of data on the left hand side of
\eqref{eq-79} and from the description of $2$-cells of $\CATC\left[\SETWinv\right]$.
Claim (b) follows easily using the set of data on the right hand side of \eqref{eq-79}
together with the identities of \eqref{eq-104} and of \eqref{eq-106}.
\end{proof}

\begin{cor}\label{cor-04}
Let us fix any pair of morphisms $\underline{f}^m:=(A^m,\operatorname{w}^m,f^m):A\rightarrow B$
for $m=1,2$ and any $2$-morphism $\Phi:\underline{f}^1\Rightarrow\underline{f}^2$
in $\CATC\left[\SETWinv\right]$. Moreover, let us fix any set of data in $\CATC$ as follows

\[

\]
such that $\operatorname{w}^1\circ\operatorname{p}$ belongs to $\SETW$ and $\varsigma$ is
invertible. Then there are an object $A^3$, a morphism $\operatorname{t}:A^3\rightarrow E$
such that $\operatorname{p}\circ\operatorname{t}$ belongs to $\SETW$,
and a $2$-morphism $\varphi:f^1\circ\operatorname{p}\circ\operatorname{t}
\Rightarrow f^2\circ\operatorname{q}\circ\operatorname{t}$ in $\CATC$, such that $\Phi$
is represented by the following diagram:

\begin{equation}\label{eq-145}
%
\end{equation}
\end{cor}

\begin{proof}
The proof follows the same ideas mentioned in the proof of Lemma~\ref{lem-02} for $\Gamma':=\Phi$;
it suffices to set $\operatorname{t}:=\operatorname{t}^1\circ\operatorname{t}^3$.
The fact that $\operatorname{p}\circ\operatorname{t}$ belongs to $\SETW$
follows from \bfTwo{} together with conditions (2) and (4) in Lemma~\ref{lem-04}.
\end{proof}

\begin{lem}\label{lem-01}
Let us fix any pair $(\CATC,\SETW)$ satisfying conditions \emphatic{\bf},
any pair of objects $A,B$, any pair of morphism $\underline{f}^m:=(A^m,\operatorname{w}^m,
f^m):A\rightarrow B$ for $m=1,2$ and any pair of $2$-morphisms $\Gamma,\Gamma':
\underline{f}^1\Rightarrow\underline{f}^2$ in $\CATC\left[\SETWinv\right]$. Let us suppose
that there are an object $A^3$, a pair of morphisms $\operatorname{v}^1$ and
$\operatorname{v}^2$ \emphatic{(}such that $\operatorname{w}^1\circ\operatorname{v}^1$
belongs to $\SETW$\emphatic{)}
and $2$-morphisms $\alpha$, $\gamma$ and $\gamma'$ in $\CATC$
\emphatic{(}with $\alpha$ invertible\emphatic{)} as below, such that
the following diagrams represent $\Gamma$, respectively $\Gamma'$:

\[
\begin{tikzpicture}[xscale=1.4,yscale=-0.6]
    \node (A0_2) at (2, 0) {$A^1$};
    \node (A2_2) at (2, 2) {$A^3$};
    \node (A2_0) at (0, 2) {$A$};
    \node (A2_4) at (4, 2) {$B$,};
    \node (A4_2) at (2, 4) {$A^2$};

    \node (A2_3) at (2.8, 2) {$\Downarrow\,\gamma$};
    \node (A2_1) at (1.2, 2) {$\Downarrow\,\alpha$};

    \path (A4_2) edge [->]node [auto,swap] {$\scriptstyle{f^2}$} (A2_4);
    \path (A0_2) edge [->]node [auto] {$\scriptstyle{f^1}$} (A2_4);
    \path (A2_2) edge [->]node [auto,swap] {$\scriptstyle{\operatorname{v}^1}$} (A0_2);
    \path (A2_2) edge [->]node [auto] {$\scriptstyle{\operatorname{v}^2}$} (A4_2);
    \path (A4_2) edge [->]node [auto] {$\scriptstyle{\operatorname{w}^2}$} (A2_0);
    \path (A0_2) edge [->]node [auto,swap] {$\scriptstyle{\operatorname{w}^1}$} (A2_0);
    
    \def \z {5}
    
    \node (B0_2) at (2+\z, 0) {$A^1$};
    \node (B2_2) at (2+\z, 2) {$A^3$};
    \node (B2_0) at (0+\z, 2) {$A$};
    \node (B2_4) at (4+\z, 2) {$B$.};
    \node (B4_2) at (2+\z, 4) {$A^2$};

    \node (B2_3) at (2.8+\z, 2) {$\Downarrow\,\gamma'$};
    \node (B2_1) at (1.2+\z, 2) {$\Downarrow\,\alpha$};

    \path (B4_2) edge [->]node [auto,swap] {$\scriptstyle{f^2}$} (B2_4);
    \path (B0_2) edge [->]node [auto] {$\scriptstyle{f^1}$} (B2_4);
    \path (B2_2) edge [->]node [auto,swap] {$\scriptstyle{\operatorname{v}^1}$} (B0_2);
    \path (B2_2) edge [->]node [auto] {$\scriptstyle{\operatorname{v}^2}$} (B4_2);
    \path (B4_2) edge [->]node [auto] {$\scriptstyle{\operatorname{w}^2}$} (B2_0);
    \path (B0_2) edge [->]node [auto,swap] {$\scriptstyle{\operatorname{w}^1}$} (B2_0);
\end{tikzpicture}
\]

Then the following facts are equivalent:

\begin{enumerate}[\emphatic{(}i\emphatic{)}]
 \item $\Gamma=\Gamma'$;
 \item there are an object $A^4$ and a morphism $\operatorname{z}:A^4\rightarrow A^3$
  in $\SETW$, such that $\gamma\circ\operatorname{z}=\gamma'\circ\operatorname{z}$;
 \item there are an object $A^4$ and a morphism $\operatorname{z}:A^4\rightarrow A^3$,
  such that $\operatorname{v}^1\circ\operatorname{z}$ belongs to $\SETW$ and such that
  $\gamma\circ\operatorname{z}=\gamma'\circ\operatorname{z}$.
\end{enumerate}
\end{lem}

\begin{proof}
Let us suppose that (i) holds and let us prove (iii). Using the description of $2$-cells of
$\CATC\left[\SETWinv\right]$ that we recalled
in the Introduction, there is a set of data $(E^1,\operatorname{r}^1,\operatorname{r}^2,
\sigma,\rho)$ in $\CATC$ as in the internal part of
the following diagram

\[

\]
such that:

\begin{itemize}
 \item $\operatorname{w}^1\circ\operatorname{v}^1\circ\operatorname{r}^1$ belongs to $\SETW$;
 \item $\sigma$ and $\rho$ are both invertible in $\CATC$;
 \item $\alpha\circ\operatorname{r}^2$ is equal to the following composition:
 
  \begin{equation}\label{eq-19}
  %
  \end{equation}

 \item $\gamma'\circ\operatorname{r}^2$ is equal to the following composition:

  \begin{equation}\label{eq-26}
  %
  \end{equation}
\end{itemize}

Since $\Gamma$ is a $2$-morphism in $\CATC\left[\SETWinv\right]$,
$\operatorname{w}^1$ and $\operatorname{w}^1\circ\operatorname{v}^1$ both belong to $\SETW$;
so by Lemma~\ref{lem-06} applied to $\sigma$
there are an object $E^2$, a morphism $\operatorname{r}^3:
E^2\rightarrow E^1$ in $\SETW$ and an invertible $2$-morphism
$\widetilde{\sigma}:\operatorname{r}^2\circ\operatorname{r}^3\Rightarrow\operatorname{r}^1
\circ\operatorname{r}^3$, such that $\operatorname{v}^1\circ\,\,
\widetilde{\sigma}=\sigma\circ\operatorname{r}^3$.\\

Since $\operatorname{w}^1\circ\operatorname{v}^1$ belongs to $\SETW$, by
\bfFive{} applied $\alpha^{-1}$ we get that also $\operatorname{w}^2\circ
\operatorname{v}^2$ belongs to $\SETW$; since also $\operatorname{w}^2$ belongs to $\SETW$,
by Lemma~\ref{lem-06} applied to $\rho\circ\operatorname{r}^3$
there are an object $E^3$, a morphism
$\operatorname{r}^4:E^3\rightarrow E^2$ in $\SETW$ and an invertible $2$-morphism
$\widetilde{\rho}:\operatorname{r}^1\circ\operatorname{r}^3\circ\operatorname{r}^4
\Rightarrow\operatorname{r}^2\circ\operatorname{r}^3\circ\operatorname{r}^4$, such that
$\operatorname{v}^2\circ\,\widetilde{\rho}=\rho\,\circ\operatorname{r}^3\circ
\operatorname{r}^4$. Using \eqref{eq-19}, this implies that $\alpha\circ\operatorname{r}^2
\circ\operatorname{r}^3\circ\operatorname{r}^4$ is equal to the following composition:

\[
\begin{tikzpicture}[xscale=1.8,yscale=-1.3]
    \node (A0_2) at (2, 0) {$E^1$};
    \node (A0_4) at (4, 0) {$A^1$};
    \node (A1_0) at (0, 2) {$E^3$};
    \node (A1_1) at (1, 1) {$E^2$};
    \node (A1_2) at (2, 1) {$E^1$};
    \node (A1_3) at (3, 1) {$A^3$};
    \node (A1_5) at (5, 1) {$A$.};
    \node (A2_1) at (1, 2) {$E^2$};
    \node (A2_2) at (2, 2) {$E^1$};
    \node (A2_4) at (4, 2) {$A^2$};
    
    \node (A0_1) at (1.5, 1.5) {$\Downarrow\,\widetilde{\rho}$};
    \node (A0_3) at (2, 0.5) {$\Downarrow\,\widetilde{\sigma}$};
    \node (A1_4) at (4, 1) {$\Downarrow\,\alpha$};

    \path (A2_4) edge [->]node [auto,swap] {$\scriptstyle{\operatorname{w}^2}$} (A1_5);
    \path (A1_0) edge [->]node [auto,swap] {$\scriptstyle{\operatorname{r}^4}$} (A2_1);
    \path (A1_1) edge [->]node [auto] {$\scriptstyle{\operatorname{r}^3}$} (A0_2);
    \path (A0_4) edge [->]node [auto] {$\scriptstyle{\operatorname{w}^1}$} (A1_5);
    \path (A1_3) edge [->]node [auto] {$\scriptstyle{\operatorname{v}^1}$} (A0_4);
    \path (A1_0) edge [->]node [auto] {$\scriptstyle{\operatorname{r}^4}$} (A1_1);
    \path (A1_3) edge [->]node [auto,swap] {$\scriptstyle{\operatorname{v}^2}$} (A2_4);
    \path (A1_1) edge [->]node [auto] {$\scriptstyle{\operatorname{r}^3}$} (A1_2);
    \path (A2_2) edge [->]node [auto,swap] {$\scriptstyle{\operatorname{r}^2}$} (A1_3);
    \path (A0_2) edge [->]node [auto] {$\scriptstyle{\operatorname{r}^2}$} (A1_3);
    \path (A1_2) edge [->]node [auto] {$\scriptstyle{\operatorname{r}^1}$} (A1_3);
    \path (A2_1) edge [->]node [auto,swap] {$\scriptstyle{\operatorname{r}^3}$} (A2_2);
\end{tikzpicture}
\]

Since $\alpha$ is invertible by hypothesis, the previous identity
implies that $\operatorname{w}^1\circ\operatorname{v}^1\circ\,
\widetilde{\sigma}\circ\operatorname{r}^4=\operatorname{w}^1\circ\operatorname{v}^1\circ\,
\widetilde{\rho}^{\,-1}$. Since $\operatorname{w}^1\circ\operatorname{v}^1$ belongs to
$\SETW$, by Lemma~\ref{lem-03} there are an object $E^4$ and a morphism $\operatorname{r}^5:
E^4\rightarrow E^3$ in $\SETW$, such that

\begin{equation}\label{eq-112}
\widetilde{\sigma}\circ\operatorname{r}^4\circ\operatorname{r}^5=
\widetilde{\rho}^{\,-1}\circ\operatorname{r}^5.
\end{equation}

From \eqref{eq-26} we get that
$\gamma'\circ\operatorname{r}^2\circ\operatorname{r}^3\circ\operatorname{r}^4
\circ\operatorname{r}^5$ is equal to the following composition:

\begin{equation}\label{eq-113}
\begin{tikzpicture}[xscale=1.8,yscale=-1.3]
    \node (A0_2) at (2, 0) {$E^1$};
    \node (A0_4) at (4, 0) {$A^1$};
    \node (B0_0) at (-1, 2) {$E^4$};
    \node (A1_0) at (0, 2) {$E^3$};
    \node (A1_1) at (1, 1) {$E^2$};
    \node (A1_2) at (2, 1) {$E^1$};
    \node (A1_3) at (3, 1) {$A^3$};
    \node (A1_5) at (5, 1) {$B$.};
    \node (A2_1) at (1, 2) {$E^2$};
    \node (A2_2) at (2, 2) {$E^1$};
    \node (A2_4) at (4, 2) {$A^2$};
    
    \node (A0_1) at (1.5, 1.5) {$\Downarrow\,\widetilde{\rho}$};
    \node (A0_3) at (2, 0.5) {$\Downarrow\,\widetilde{\sigma}$};
    \node (A1_4) at (4, 1) {$\Downarrow\,\gamma$};

    \path (B0_0) edge [->]node [auto] {$\scriptstyle{\operatorname{r}^5}$} (A1_0);
    \path (A2_4) edge [->]node [auto,swap] {$\scriptstyle{f^2}$} (A1_5);
    \path (A1_0) edge [->]node [auto,swap] {$\scriptstyle{\operatorname{r}^4}$} (A2_1);
    \path (A1_1) edge [->]node [auto] {$\scriptstyle{\operatorname{r}^3}$} (A0_2);
    \path (A0_4) edge [->]node [auto] {$\scriptstyle{f^1}$} (A1_5);
    \path (A1_3) edge [->]node [auto] {$\scriptstyle{\operatorname{v}^1}$} (A0_4);
    \path (A1_0) edge [->]node [auto] {$\scriptstyle{\operatorname{r}^4}$} (A1_1);
    \path (A1_3) edge [->]node [auto,swap] {$\scriptstyle{\operatorname{v}^2}$} (A2_4);
    \path (A1_1) edge [->]node [auto] {$\scriptstyle{\operatorname{r}^3}$} (A1_2);
    \path (A2_2) edge [->]node [auto,swap] {$\scriptstyle{\operatorname{r}^2}$} (A1_3);
    \path (A0_2) edge [->]node [auto] {$\scriptstyle{\operatorname{r}^2}$} (A1_3);
    \path (A1_2) edge [->]node [auto] {$\scriptstyle{\operatorname{r}^1}$} (A1_3);
    \path (A2_1) edge [->]node [auto,swap] {$\scriptstyle{\operatorname{r}^3}$} (A2_2);
\end{tikzpicture}
\end{equation}

From \eqref{eq-112} and \eqref{eq-113} we conclude that

\begin{equation}\label{eq-140}
\gamma'\circ\operatorname{r}^2\circ\operatorname{r}^3\circ\operatorname{r}^4
\circ\operatorname{r}^5=\gamma\circ\operatorname{r}^2\circ\operatorname{r}^3
\circ\operatorname{r}^4\circ\operatorname{r}^5.
\end{equation}

By construction we have that $\operatorname{w}^1\circ
\operatorname{v}^1\circ\operatorname{r}^1$ belongs to $\SETW$, so also $\operatorname{w}^1\circ
\operatorname{v}^1\circ\operatorname{r}^2$ belongs to $\SETW$ (using \bfFive{}
on $\operatorname{w}^1\circ\,\sigma$). By construction $\operatorname{r}^3$,
$\operatorname{r}^4$ and $\operatorname{r}^5$ also belong to $\SETW$.
Using \bfTwo, we conclude that also
the morphism $\operatorname{w}^1\circ\,(\operatorname{v}^1\circ\operatorname{r}^2\circ
\operatorname{r}^3\circ\operatorname{r}^4\circ\operatorname{r}^5)$ belongs to $\SETW$.\\

By hypothesis $\operatorname{w}^1$ also belongs to $\SETW$.
So by Lemma~\ref{lem-11} there are an object $A^4$ and a morphism $\operatorname{r}^6:
A^4\rightarrow E^4$, such that the morphism
$\operatorname{v}^1\circ\operatorname{r}^2\circ\operatorname{r}^3\circ\operatorname{r}^4
\circ\operatorname{r}^5\circ\operatorname{r}^6$ also belongs to $\SETW$. We set
$\operatorname{z}:=\operatorname{r}^2\circ\operatorname{r}^3\circ\operatorname{r}^4
\circ\operatorname{r}^5\circ\operatorname{r}^6:A^4\rightarrow A^3$ (so that
$\operatorname{v}^1\circ\operatorname{z}$ belongs to $\SETW$). Then
\eqref{eq-140} implies that $\gamma'\circ\operatorname{z}=\gamma\circ\operatorname{z}$,
so (iii) holds.\\

Now let us assume (iii) and let us prove (ii). By hypothesis $\operatorname{v}^1
\circ\operatorname{z}$ belongs to $\SETW$; moreover $\operatorname{w}^1$ belongs
to $\SETW$ because $\underline{f}^1$ is a morphism in $\CATC\left[\SETWinv\right]$.
Then by axiom \bfTwo{} also $\operatorname{w}^1\circ\operatorname{v}^1
\circ\operatorname{z}$ belongs to $\SETW$. Again by hypothesis $\operatorname{w}^1
\circ\operatorname{v}^1$ belongs to $\SETW$. So by Lemma~\ref{lem-11}
there are an object $A^{\prime 4}$ and a morphism $\operatorname{z}':A^{\prime 4}
\rightarrow A^4$,
such that $\operatorname{z}\circ\operatorname{z}'$ belongs to $\SETW$. By (iii) we have
$\gamma\circ\operatorname{z}=\gamma'\circ\operatorname{z}$, hence
$\gamma\circ(\operatorname{z}\circ\operatorname{z}')=
\gamma'\circ(\operatorname{z}\circ\operatorname{z}')$, so (ii) holds.\\

Lastly, if (ii) holds, then (i) is obviously satisfied using the definition of
$2$-morphism in $\CATC\left[\SETWinv\right]$ and axiom \bfTwo.
\end{proof}

\begin{rem}\label{rem-03}
Let us fix any pair of objects $A$, $B$, any pair of morphisms $f^1,f^2:A\rightarrow B$
and any $2$-morphism $\rho:f^1\Rightarrow f^2$ in $\CATC$. As explained
in~\cite[\S~2.4]{Pr}, the universal pseudofunctor $\mathcal{U}_{\SETW}:\CATC\rightarrow
\CATC\left[\SETWinv\right]$ is such that $\mathcal{U}_{\SETW}(f^m)$ is the triple
$(A,\id_A,f^m):A\rightarrow B$ for each $m=1,2$, and $\mathcal{U}_{\SETW}(\rho)$ is the $2$-morphism
represented by the following diagram:

\begin{equation}\label{eq-124}
\begin{tikzpicture}[xscale=1.4,yscale=-0.6]
    \node (A0_2) at (2, 0) {$A$};
    \node (A2_2) at (2, 2) {$A$};
    \node (A2_0) at (0, 2) {$A$};
    \node (A2_4) at (4, 2) {$B$.};
    \node (A4_2) at (2, 4) {$A$};

    \node (A2_3) at (2.8, 2) {$\Downarrow\,\rho$};
    \node (A2_1) at (1.2, 2) {$\Downarrow\,i_{\id_A}$};

    \path (A4_2) edge [->]node [auto,swap] {$\scriptstyle{f^2}$} (A2_4);
    \path (A0_2) edge [->]node [auto] {$\scriptstyle{f^1}$} (A2_4);
    \path (A2_2) edge [->]node [auto,swap] {$\scriptstyle{\id_A}$} (A0_2);
    \path (A2_2) edge [->]node [auto] {$\scriptstyle{\id_A}$} (A4_2);
    \path (A4_2) edge [->]node [auto] {$\scriptstyle{\id_A}$} (A2_0);
    \path (A0_2) edge [->]node [auto,swap] {$\scriptstyle{\id_A}$} (A2_0);
\end{tikzpicture}
\end{equation}

Now let us fix any $2$-morphism $\Phi:\mathcal{U}_{\SETW}(f^1)\Rightarrow\mathcal{U}_{\SETW}(f^2)$.
Using Corollary~\ref{cor-04} with $(A^1,A^2,E,\operatorname{w}^1,\operatorname{w}^2,\operatorname{p},
\operatorname{q},\varsigma):=(A,A,A,\id_A,\id_A,\id_A,\id_A,i_{\id_A})$, there are an
object $A'$, a morphism $\operatorname{t}:A'\rightarrow A$ in $\SETW$ and a $2$-morphism $\varphi:f^1
\circ\operatorname{t}\Rightarrow f^2\circ\operatorname{t}$, such that $\Phi$ is represented
by the following diagram:

\begin{equation}\label{eq-95}
\begin{tikzpicture}[xscale=1.4,yscale=-0.6]
    \node (A0_2) at (2, 0) {$A$};
    \node (A2_2) at (2, 2) {$A'$};
    \node (A2_0) at (0, 2) {$A$};
    \node (A2_4) at (4, 2) {$B$.};
    \node (A4_2) at (2, 4) {$A$};

    \node (A2_3) at (2.8, 2) {$\Downarrow\,\varphi$};
    \node (A2_1) at (1.2, 2) {$\Downarrow\,i_{\operatorname{t}}$};

    \path (A4_2) edge [->]node [auto,swap] {$\scriptstyle{f^2}$} (A2_4);
    \path (A0_2) edge [->]node [auto] {$\scriptstyle{f^1}$} (A2_4);
    \path (A2_2) edge [->]node [auto,swap] {$\scriptstyle{\operatorname{t}}$} (A0_2);
    \path (A2_2) edge [->]node [auto] {$\scriptstyle{\operatorname{t}}$} (A4_2);
    \path (A4_2) edge [->]node [auto] {$\scriptstyle{\id_A}$} (A2_0);
    \path (A0_2) edge [->]node [auto,swap] {$\scriptstyle{\id_A}$} (A2_0);
\end{tikzpicture}
\end{equation}

Therefore the following facts are equivalent:

\begin{enumerate}[(a)]
 \item $\Phi$ is in the image of $\mathcal{U}_{\SETW}$;
 \item there is a $2$-morphism $\rho:f^1\Rightarrow f^2$ in $\CATC$, such that \eqref{eq-95}
  is equivalent to \eqref{eq-124}.
\end{enumerate}

Since the class of \eqref{eq-124} is equal to $\Big[A',\operatorname{t},\operatorname{t},
i_{\operatorname{t}},\rho\circ\operatorname{t}\Big]$, by Lemma~\ref{lem-01} we get that (b)
is equivalent to:

\begin{enumerate}[(a)]
\setcounter{enumi}{2}
 \item there are a $2$-morphism $\rho:f^1\Rightarrow f^2$ in $\CATC$,
  an object $A''$ and a morphism $\operatorname{z}:A''\rightarrow A'$ in $\SETW$,
  such that $\varphi\circ
  \operatorname{z}=\rho\circ\operatorname{t}\circ\operatorname{z}$.
\end{enumerate}

Therefore in general \emph{the universal pseudofunctor $\mathcal{U}_{\SETW}$ is not
$2$-full}, but only ``$2$-full modulo morphisms of $\SETW$''.
It is $2$-full if and only if the pair $(\CATC,\SETW)$ satisfies the following
condition (in addition to the usual set of axioms
\bfOne{} -- \bfFive):

\begin{description}
 \descitem{$(\operatorname{BF6})$:}{BF6}
  for every pair of morphisms $f^1,f^2:A\rightarrow B$, for every morphism
  $\operatorname{t}:A'\rightarrow A$ in $\SETW$ and for every $2$-morphism $\varphi:f^1\circ
  \operatorname{t}\Rightarrow f^2\circ\operatorname{t}$, there are an object $A''$, a morphism
  $\operatorname{z}:A''\rightarrow A'$ in $\SETW$ and a $2$-morphism $\rho:f^1\Rightarrow f^2$,
  such that $\varphi\circ\operatorname{z}=\rho\circ\operatorname{t}\circ\operatorname{z}$.
\end{description}

Now let us fix another $2$-morphism $\rho\,':f^1\Rightarrow f^2$. Using
\eqref{eq-124} and Lemma~\ref{lem-01}, the following facts are equivalent:

\begin{itemize}
 \item $\mathcal{U}_{\SETW}(\rho)=\mathcal{U}_{\SETW}(\rho\,')$;
 \item there are an object $\widetilde{A}$ and a morphism $\operatorname{u}:\widetilde{A}
  \rightarrow A$ in $\SETW$, such that $\rho\circ\operatorname{u}=\rho\,'\circ\operatorname{u}$.
\end{itemize}

This shows that in general \emph{the universal pseudofunctor $\mathcal{U}_{\SETW}$ is not
$2$-faithful}, but only ``$2$-faithful modulo morphisms of $\SETW$''.\\
\end{rem}

Now we are able to give the following proof:

\begin{proof}[of Proposition~\ref{prop-05}]
Let us fix any representatives \eqref{eq-10}
for $\Gamma$ and $\Gamma'$ as in Lemma~\ref{lem-02},
any set of data (1) -- (5) as in Lemma~\ref{lem-04}.  Using Lemma~\ref{lem-02},
there are representatives for $\Gamma$ and $\Gamma'$ as in the first part of
Proposition~\ref{prop-05}:
it suffices to set $\alpha:=\varsigma\,\circ\operatorname{t}^1\circ\operatorname{t}^3$,
$\gamma:=\psi\circ\operatorname{t}^1\circ\operatorname{t}^3$, $\gamma':=\varphi$,
$\operatorname{v}^1:=\operatorname{p}\circ\operatorname{t}^1\circ\operatorname{t}^3$
and $\operatorname{v}^2:=\operatorname{q}\circ\operatorname{t}^1\circ\operatorname{t}^3$
in diagram \eqref{eq-111}.\\

Using the data (2) and (4), we get that $\operatorname{v}^1$ belongs to $\SETW$. Moreover,
$\operatorname{w}^1$ belongs to $\SETW$ because $\underline{f}^1$
is a morphism of $\CATC\left[\SETWinv\right]$. Hence $\operatorname{w}^1\circ
\operatorname{v}^1$ also belongs to $\SETW$, so we can apply
Lemma~\ref{lem-01} in order to get the second part of
Proposition~\ref{prop-05}.
\end{proof}

\begin{rem}
By induction and using the same ideas mentioned in this section, one can also prove that given
finitely many
$2$-morphisms $\Gamma^1,\cdots,\Gamma^n$ in $\CATC\left[\SETWinv\right]$, all defined between the
same pair of morphisms, there are data $A^3,\operatorname{v}^1,\operatorname{v}^2,\alpha,\gamma^1,
\cdots,\gamma^n$, such that $\operatorname{v}^1$ belongs to $\SETW$,
$\alpha$ is invertible and $\Gamma^m=[A^3,\operatorname{v}^1,\operatorname{v}^2,
\alpha,\gamma^m]$ for each $m=1,\cdots,n$. In other words, given finitely many $2$-morphisms
with the same source and target, each of them admits a representative that differs from the
other ones only in the last variable at most. However, in general the common data
$(A^3,\operatorname{v}^1,\operatorname{v}^2,\alpha)$ depend on the $2$-morphisms chosen:
if we add another $2$-morphism $\Gamma^{n+1}$ to the collection above, then the new common data
for $\Gamma^1,\cdots,\Gamma^{n+1}$
in general is of the form $(A^4,\operatorname{v}^1\circ\operatorname{z},
\operatorname{v}^2\circ\operatorname{z},\alpha\circ\operatorname{z})$ for some object $A^4$
and some morphism $\operatorname{z}:A^4\rightarrow A^3$ in $\SETW$.
\end{rem}

\section{The associators of a bicategory of fractions}\label{sec-01}
As we mentioned in the Introduction, in order to prove Theorem~\ref{thm-01} we need to
show that the constructions of associators and vertical/horizontal compositions do
not depend on the set of fixed choices $\DW$, but only on
the choices $\CW$ (at most). We start with the description
of associators in any bicategory of fractions:

\begin{prop}\label{prop-01}\emph{\textbf{(associators of $\CATC\left[\SETWinv\right]$)}}
Let us fix any triple of $1$-morphisms in $\CATC\left[\SETWinv\right]$ as follows:

\begin{equation}\label{eq-65}

\end{equation}

Let us suppose that the set of fixed choices $\CW$
gives data as in the upper part
of the following diagrams \emphatic{(}starting from the ones on the left\emphatic{)},
with $\operatorname{u}^1$, $\operatorname{u}^2$, $\operatorname{v}^1$ and
$\operatorname{u}^3$ in $\SETW$ and $\delta$, $\sigma$, $\xi$ and $\eta$ invertible:

\begin{gather}
\nonumber %
\end{gather}
so that by~\cite[\S~2.2]{Pr} one has

\begin{gather}
\label{eq-60}
%
\end{gather}

Then let us fix \emph{any set of data} in $\CATC$ as follows:

\begin{description}
 \descitem{$(\operatorname{F1})$:}{F1}
  an object $A^4$, a morphism $\operatorname{u}^4:A^4
  \rightarrow A^2$ such that $\operatorname{u}\circ\operatorname{u}^1\circ\operatorname{u}^2
  \circ\operatorname{u}^4$ belongs to $\SETW$ \emphatic{(}for example, using
  \emphatic{\bfTwo} this is the case if $\operatorname{u}^4$
  belongs to $\SETW$\emphatic{)}, a morphism $\operatorname{u}^5:A^4\rightarrow
  A^3$ and an invertible $2$-morphism $\gamma:\operatorname{u}^1\circ
  \operatorname{u}^2\circ\operatorname{u}^4\Rightarrow\operatorname{u}^3\circ\operatorname{u}^5$;

 \descitem{$(\operatorname{F2})$:}{F2}
  an invertible $2$-morphism $\omega:f^1\circ\operatorname{u}^2\circ\operatorname{u}^4
  \Rightarrow\operatorname{v}^1\circ f^2\circ\operatorname{u}^5$,
  such that $\operatorname{v}\circ\,\omega$ is equal to the following composition:
  
  \begin{equation}\label{eq-61}

  \end{equation}

 \descitem{$(\operatorname{F3})$:}{F3}
  an invertible $2$-morphism $\rho:\,l\circ
  \operatorname{u}^4\Rightarrow g^1\circ f^2\circ
  \operatorname{u}^5$, such that $\operatorname{w}\circ\,\rho$ is equal to
  the following composition:

  \begin{equation}\label{eq-62}
  %
  \end{equation}
\end{description}

Then the associator $\Theta^{\CATC,\SETW}_{\underline{h},\underline{g},\underline{f}}$ from
\eqref{eq-60} to \eqref{eq-39} is represented by the following diagram:

\begin{equation}\label{eq-47}
%
\end{equation}

Given any other data as in \emphatic{\F{1}} --
\emphatic{\F{3}}, the diagram \eqref{eq-47} induced by the new data
is equivalent to \eqref{eq-47}.
\end{prop}

In particular:

\begin{itemize}
 \item by definition of composition of $1$-morphisms in a bicategory of fractions,
  the construction of $\underline{h}\circ(\underline{g}\circ\underline{f})$ and
  of $(\underline{h}\circ\underline{g})\circ\underline{f}$ depends on the four choices
  in the set $\CW$ giving the four triangles in \eqref{eq-70}, hence inducing the
  compositions in \eqref{eq-60} and \eqref{eq-39} (different sets of
  choices $\CW$ give different, but equivalent, bicategories of fractions);
 \item apart from the four choices in $\CW$
  needed above, the associator in $\CATC\left[\SETWinv\right]$ from
  $\underline{h}\circ(\underline{g}\circ\underline{f})$ to
  $(\underline{h}\circ\underline{g})\circ\underline{f}$ does not depend on any additional
  choice of type $\CW$, nor on any choice of type $\DW$;
 \item hence the construction of associators in any bicategory of fractions does not depend
  on the set of choices $\DW$. As we will mention below, each set of choices
  $\DW$ induces a specific
  set of data \F{1} -- \F{3}, hence a specific diagram
  \eqref{eq-47}, but any other set of choices $\DW$
  gives a diagram that is equivalent to \eqref{eq-47}, i.e., it induces 
  the same $2$-morphism in $\CATC\left[\SETWinv\right]$.
\end{itemize}

\begin{rem}
Before proving Proposition~\ref{prop-01}, we remark that even if such a statement is long,
the construction described there is
simpler than the original construction in~\cite{Pr}.
As a consequence of the axioms, a set of data as in 
\F{1} -- \F{3} is easy to construct,
as shown in Remark~\ref{rem-01} below.
\end{rem}

\begin{proof}
\textbf{Step 1.} Following~\cite[Appendix~A.2]{Pr}, one gets
immediately a set of data satisfying conditions
\F{1} -- \F{3} and inducing
the desired associator as in \eqref{eq-47}.
In~\cite{Pr} these data are induced by the additional fixed choices 
$\DW$, so they satisfy also some additional properties that we
don't need to recall here in full details (for example, $\operatorname{u}^4$ belongs to
$\SETW$; using \bfTwo{} this condition is slightly stronger than condition
\F{1}).
The aim of this proof is to show that all these additional properties
are actually not necessary, because \emph{for any data satisfying}
\F{1} -- \F{3} (even if not induced by the choices
$\DW$), \emph{the induced diagram \eqref{eq-47} 
is a representative for the associator}
$\Theta^{\CATC,\SETW}_{\underline{h},\underline{g},\underline{f}}$.\\

Since in~\cite{Pr} the associator is induced by a particular set of data
\F{1} -- \F{3},
in order to prove the claim it is sufficient to show that any $2$ different sets of
data as in \F{1} -- \F{3} induce diagrams
of the form \eqref{eq-47} that are equivalent. So let us fix any other set of data satisfying
\F{1} -- \F{3}, as follows:

\begin{description}
 \descitem{$(\operatorname{F1})'$:}{F1prime}
  an object $\widetilde{A}^4$, a morphism $\widetilde{\operatorname{u}}^4:
  \widetilde{A}^4\rightarrow A^2$ such that $\operatorname{u}\circ\operatorname{u}^1\circ
  \operatorname{u}^2\circ\,\widetilde{\operatorname{u}}^4$ belongs to $\SETW$, a morphism
  $\widetilde{\operatorname{u}}^5:\widetilde{A}^4\rightarrow A^3$ and an
  invertible $2$-morphism $\widetilde{\gamma}:\operatorname{u}^1\circ\operatorname{u}^2\circ\,
  \widetilde{\operatorname{u}}^4\Rightarrow\operatorname{u}^3\circ\,
  \widetilde{\operatorname{u}}^5$;
 \descitem{$(\operatorname{F2})'$:}{F2prime}
  an invertible $2$-morphism $\widetilde{\omega}:
  f^1\circ\operatorname{u}^2\circ\,
  \widetilde{\operatorname{u}}^4\Rightarrow\operatorname{v}^1\circ f^2\circ
  \widetilde{\operatorname{u}}^5$, such that $\operatorname{v}\circ\,\widetilde{\omega}$
  is equal to the following composition:
  
  \begin{equation}\label{eq-74}  
  
  \end{equation}
 \descitem{$(\operatorname{F3})'$:}{F3prime} an invertible $2$-morphism $\widetilde{\rho}:l\circ
  \widetilde{\operatorname{u}}^4\Rightarrow g^1\circ f^2\circ
  \widetilde{\operatorname{u}}^5$, such that $\operatorname{w}\circ\,\widetilde{\rho}$
  is equal to the following composition:
  
  \begin{equation}\label{eq-14}
  %
  \end{equation}
\end{description}

Then proving the claim is equivalent to proving that \eqref{eq-47} is equivalent to:

\begin{equation}\label{eq-64}
%
\end{equation}

\textbf{Step 2.} We want to construct a pair of diagrams that are equivalent to
\eqref{eq-47}, respectively to \eqref{eq-64}, and that share
a common ``left hand side''. For this, we follow the procedure
explained in Lemmas~\ref{lem-04} and~\ref{lem-02}.
Using Lemma~\ref{lem-04}, there are a pair of objects $F^1$ and $F^2$, a triple
of morphisms $\operatorname{t}^1$, $\operatorname{t}^2$ and $\operatorname{t}^3$,
such that both $\operatorname{u}^4\circ\operatorname{t}^1$ and $\operatorname{t}^3$
belong to $\SETW$, and a pair of invertible $2$-morphisms
$\varepsilon$ and $\kappa$ in $\CATC$ as follows

\[

\]
such that $\operatorname{u}\circ\,\gamma\circ\operatorname{t}^1
\circ\operatorname{t}^3$ is equal to the following composition:

\begin{equation}\label{eq-03}
%
\end{equation}

Using Lemma~\ref{lem-02} and the previous identity, diagram \eqref{eq-47} is equivalent to

\begin{equation}\label{eq-103}
%
\end{equation}
and diagram \eqref{eq-64} is equivalent to

\begin{equation}\label{eq-102}
%
\end{equation}
where $\zeta$ is the following composition:

\begin{equation}\label{eq-05}
%
\end{equation}

\textbf{Step 3.} Now the claim is equivalent to proving that \eqref{eq-103}
and \eqref{eq-102} are equivalent.
Since they have already a common ``left hand side'', it suffices to prove that $h\circ\rho\circ
\operatorname{t}^1\circ\operatorname{t}^3$ and $h\circ\zeta$ are equal if pre-composed
with a suitable morphism of $\SETW$; then we will conclude using Lemma~\ref{lem-01}(ii).\\

\textbf{Step 4.} Since $\operatorname{u}\circ\,
\gamma\circ\operatorname{t}^1\circ\operatorname{t}^3$ coincides with \eqref{eq-03},
by Lemma~\ref{lem-03} there are an object $F^3$ and
a morphism $\operatorname{t}^4:F^3\rightarrow F^2$ in $\SETW$, such that
$\gamma\circ\operatorname{t}^1\circ\operatorname{t}^3\circ\operatorname{t}^4$
is equal to

\begin{equation}\label{eq-06}

\end{equation}

Let us denote by $\phi$ the composition below:

\begin{equation}\label{eq-07}
%
\end{equation}

Using \eqref{eq-74}, we get that $\operatorname{v}\circ\,\phi$ is equal to the
following composition:

\[
%
\]

Then using \eqref{eq-06}, we get that $\operatorname{v}\circ\,\phi$
coincides also with the following composition:

\[
%
\]

Using \eqref{eq-61}, we conclude that $\operatorname{v}\circ\,\phi$
coincides also with $\operatorname{v}\circ\,\omega\circ\operatorname{t}^1\circ\operatorname{t}^3
\circ\operatorname{t}^4$.
Since $\underline{g}=(B',\operatorname{v},g)$ is a morphism in $\CATC\left[\SETWinv\right]$,
$\operatorname{v}$ belongs to $\SETW$. So by Lemma~\ref{lem-03}
there are an object $F^4$ and a  morphism $\operatorname{t}^5:F^4\rightarrow F^3$ in $\SETW$, such
that $\phi\circ\operatorname{t}^5=\omega\circ\operatorname{t}^1\circ
\operatorname{t}^3\circ\operatorname{t}^4\circ\operatorname{t}^5$.
Using \eqref{eq-07}, this implies that $\omega\circ\operatorname{t}^1\circ
\operatorname{t}^3\circ\operatorname{t}^4\circ\operatorname{t}^5$ is equal to the following
composition:

\begin{equation}\label{eq-08}

\end{equation}

\textbf{Step 5.} Using \eqref{eq-05} and \eqref{eq-14},
$\operatorname{w}\circ\,\zeta\circ\operatorname{t}^4\circ\operatorname{t}^5$
is equal to the following composition:

\[
%
\]

Using \eqref{eq-08} we conclude that $\operatorname{w}\circ\,
\zeta\circ\,\operatorname{t}^4\circ\operatorname{t}^5$ 
coincides also with the composition

\[
%
\]

Using the previous identity and \eqref{eq-62}, we conclude that
$\operatorname{w}\circ\,\rho\circ\operatorname{t}^1\circ\operatorname{t}^3\circ\operatorname{t}^4
\circ\operatorname{t}^5=\operatorname{w}\circ\,\zeta\circ\operatorname{t}^4\circ\operatorname{t}^5$.
Using Lemma~\ref{lem-03}, this implies that there are an object $F^5$ and
a morphism $\operatorname{t}^6:F^5\rightarrow F^4$ in $\SETW$, such that
$\rho\circ\operatorname{t}^1\circ\operatorname{t}^3\circ\operatorname{t}^4
\circ\operatorname{t}^5\circ\operatorname{t}^6
=\zeta\circ\operatorname{t}^4\circ\operatorname{t}^5\circ\operatorname{t}^6$.
This implies that:

\begin{equation}\label{eq-17}
(h\circ\rho\circ\operatorname{t}^1\circ\operatorname{t}^3)\circ(\operatorname{t}^4
\circ\operatorname{t}^5\circ\operatorname{t}^6)=
(h\circ\zeta)\circ(\operatorname{t}^4\circ\operatorname{t}^5\circ\operatorname{t}^6).
\end{equation}

\textbf{Step 6.} By construction the morphisms $\operatorname{t}^4$, $\operatorname{t}^5$
and $\operatorname{t}^6$ all belong to $\SETW$, so
their composition also belongs to $\SETW$ by \bfTwo.
So \eqref{eq-17} and Lemma~\ref{lem-01}(ii) imply that \eqref{eq-103}
and \eqref{eq-102} are equivalent, so we conclude by Step 3.
\end{proof}

\begin{rem}\label{rem-01}
In order to find a set of data $(A^4,\operatorname{u}^4,\operatorname{u}^5,\gamma,
\omega,\rho)$ as in \F{1} -- \F{3} you can
follow the construction in~\cite[Appendix~A.2]{Pr}. However,
such a procedure is long: a shorter construction is the following
one. First of all, we use axiom \bfThree{} in order to get data as in the upper part
of the following diagram, with $\overline{\operatorname{u}}^4$ in $\SETW$ and $\overline{\gamma}$
invertible:

\[

\]

Then we use \bfFourA{} and \bfFourB{} in order to get an object $F$, a
morphism $\operatorname{z}:F\rightarrow E$ in $\SETW$ and an invertible $2$-morphism
$\overline{\omega}:f^1\circ\operatorname{u}^2\circ\,\overline{\operatorname{u}}^4\circ
\operatorname{z}\Rightarrow\operatorname{v}^1\circ f^2\circ\overline{\operatorname{u}}^5\circ
\operatorname{z}$, such that $\operatorname{v}\circ\,\overline{\omega}$ is equal to the
following composition:

\[
%
\]

Then we use again \bfFourA{} and \bfFourB{} in order to get an object
$A^4$, a morphism $\operatorname{r}:A^4\rightarrow F$ in $\SETW$ and an invertible $2$-morphism
$\rho:l\circ\overline{\operatorname{u}}^4\circ\operatorname{z}\circ\operatorname{r}
\Rightarrow g^1\circ f^2\circ\overline{\operatorname{u}}^5\circ\operatorname{z}\circ
\operatorname{r}$, such that $\operatorname{w}\circ\,\rho$ is equal to the following composition:

\[
%
\]

Then it suffices to define $\operatorname{u}^4:=
\overline{\operatorname{u}}^4\circ\operatorname{z}\circ\operatorname{r}$,
$\operatorname{u}^5:=\overline{\operatorname{u}}^5\circ\operatorname{z}
\circ\operatorname{r}$, $\gamma:=\overline{\gamma}\circ\operatorname{z}
\circ\operatorname{r}$ and $\omega:=\overline{\omega}\circ\operatorname{r}$.
Note that $\operatorname{u}^4$ belongs to $\SETW$ by construction and 
\bfTwo. Since $\operatorname{u}$,
$\operatorname{u}^1$ and $\operatorname{u}^2$ belong to $\SETW$ by hypothesis or
construction, \F{1} holds by \bfTwo. Also
\F{2} and \F{3} are easily verified. Above
we used axioms \bfThree{} and \bfFour, hence
the data that we found are in general non-unique.
\end{rem}

In the following Corollary, differently from the previous statements,
we need to explicitly use the associators and the right and left unitors for $\CATC$.
We denote these
$2$-morphisms by $\theta_{a,b,c}:a\circ(b\circ c)\Rightarrow(a\circ b)\circ c$
(for any triple of composable morphisms $a,b,c$), respectively by $\pi_a:a\circ\id_A
\Rightarrow a$, respectively by $\upsilon_a:\id_B\circ\,a\Rightarrow a$
(for any morphism $a:A\rightarrow B$).

\begin{cor}\label{cor-01}
Let us fix any pair $(\CATC,\SETW)$ satisfying conditions \emphatic{\bf}, any
triple of morphisms $\underline{f},\underline{g},\underline{h}$ as in 
\eqref{eq-65} and let us suppose that $B=B'$, $C=C'$, $\operatorname{v}=\id_B$ and
$\operatorname{w}=\id_C$. Moreover, let us suppose that the set of fixed
choices $\CW$
gives data as in the upper part of the following diagram, with $\operatorname{u}^3$ in $\SETW$
and $\eta$ invertible:

\begin{equation}\label{eq-76}

\end{gather}
and the associator $\Thetaa{\underline{h}}{\underline{g}}{\underline{f}}^{\CATC,\SETW}$ from
\eqref{eq-66} to \eqref{eq-67} is represented by the following diagram:

\begin{equation}\label{eq-125}
%
\end{equation}
where $\alpha$ is the following composition

\[
%
\]
and $\beta$ is the following one:

\[
%
\]
\end{cor}

\begin{rem}
If $\CATC$ is a $2$-category, then the associators $\theta_{\bullet}$ and the unitors
$\pi_{\bullet}$ and $\upsilon_{\bullet}$ are all trivial; moreover, $\id_B\circ\id_B=\id_B$.
We recall that the fixed choices $\CW$ imposed by Pronk on any pair
$(f,\operatorname{v})$ assume a very simple form in the case when either $f$ or $\operatorname{v}$
are an identity (see~\cite[p.~256]{Pr}), so the quadruple $(A^3,\operatorname{u}^3,f^2,
\eta)$ of \eqref{eq-76}
coincides with $(A',\id_{A'},f,i_f)$. So if $\CATC$ is a $2$-category, then the morphisms
\eqref{eq-66} and \eqref{eq-67} coincide, and the associator \eqref{eq-125}
is the $2$-identity of this morphism.
\end{rem}

\begin{proof}[of Corollary~\ref{cor-01}.]
This is the first proof where we explicitly need to use associators and unitors
of $\CATC$ because we
cannot prove the statement only in the special 
case of a $2$-category and then appeal to coherence results for the general case.
First, anyway, we prove this special case.\\

By the already mentioned~\cite[p.~256]{Pr}, since both $\operatorname{v}$ and
$\operatorname{w}$ are identities, one gets that the $4$ diagrams of \eqref{eq-70} (chosen from
left to right) assume this simple form:

\begin{gather}\label{eq-137}

\end{gather}

Then identities \eqref{eq-66} and \eqref{eq-67} follow at once from \eqref{eq-60} and \eqref{eq-39}.
In order to compute the associator, according to Proposition~\ref{prop-01} we have to fix any set
of data as in \F{1} -- \F{3}. For that, we choose:

\begin{itemize}
 \item $A^4:=A'$, $\operatorname{u}^4:=\id_{A'}$, $\operatorname{u}^5:=\id_{A'}$ and
  $\gamma:=i_{\id_{A'}}$;
 \item $\omega:=i_f$;
 \item $\rho:=i_{g\circ f}$.
\end{itemize}

Then the claim follows from Proposition~\ref{prop-01}.
In the general case when $\CATC$ is a bicategory, by~\cite[p.~256]{Pr}
the first diagram in \eqref{eq-137} is given by

\[
\begin{tikzpicture}[xscale=2.8,yscale=-0.8]
    \node (A3_2) at (2, 3.3) {$A^1:=A'$};
    \node (A6_1) at (1, 6) {$A'$};
    \node (A6_2) at (2, 6) {$B$};
    \node (A6_3) at (3, 6) {$B'=B$,};
    
    \node (A5_2) at (2, 5.4) {$\Rightarrow$};
    \node (A4_2) at (2, 4.8) {$\delta:=\upsilon_f^{-1}\circ\pi_f$};

    \path (A3_2) edge [->]node [auto] {$\scriptstyle{f^1:=f}$} (A6_3);
    \path (A6_1) edge [->]node [auto,swap] {$\scriptstyle{f}$} (A6_2);
    \path (A3_2) edge [->]node [auto,swap] {$\scriptstyle{\operatorname{u}^1:=\id_{A'}}$} (A6_1);
    \path (A6_3) edge [->]node [auto] {$\scriptstyle{\operatorname{v}=\id_B}$} (A6_2);
\end{tikzpicture}
\]
and analogously for the second diagram and for the third one
in \eqref{eq-137}; the fourth diagram must be replaced by
\eqref{eq-76}. Indeed, if $\CATC$ is simply a bicategory,
in general we cannot write $\operatorname{v}\circ\operatorname{v}^1=\id_B\circ\id_B$
in a simpler form, hence we cannot say anything more precise about the data in diagram \eqref{eq-76}
(it is completely determined as one 
of the fixed choices in the set $\CW$, so we have no control over it).
The data of \F{1} -- \F{3} above have to be changed according to these choices;
then the statement follows again from Proposition~\ref{prop-01}.
\end{proof}

\section{Vertical composition}
Following the plan explained in the Introduction, in this section we will prove that the
vertical composition in $\CATC\left[\SETWinv\right]$ does not depend on the set of choices
$\DW$ (actually, it does not depend on $\CW$
either).
In doing that, we will also provide a simple way of computing a representative for the vertical
composition of any pair of $2$-morphisms $\Gamma^1,\Gamma^2$ in any bicategory of fractions,
having fixed representatives for $\Gamma^1$ and $\Gamma^2$.

\begin{prop}\label{prop-02}\emph{\textbf{(vertical composition)}}
Let us fix any pair of objects $A,B$, any triple of morphisms $\underline{f}^m:=(A^m,
\operatorname{w}^m,f^m):A\rightarrow B$ for $m=1,2,3$ and any pair of $2$-morphisms
$\Gamma^1:\underline{f}^1\Rightarrow\underline{f}^2$ and $\Gamma^2:\underline{f}^2
\Rightarrow\underline{f}^3$ in $\CATC\left[\SETWinv\right]$;
let us fix any representative for $\Gamma^1$, respectively
for $\Gamma^2$, as follows:

\begin{equation}\label{eq-138}

\end{equation}

Then let us fix \emph{any set of data} in $\CATC$ as follows:

\begin{description}
 \descitem{$(\operatorname{F4})$:}{F4}
  an object $C$, a morphism $\operatorname{t}^1$ in $\SETW$, a morphism
  $\operatorname{t}^2$ and an invertible $2$-morphism $\rho$ in $\CATC$ as below:

  \[
  %
  \]
\end{description}

Then the vertical composition $\Gamma^2\circ\Gamma^1$ is represented by the following diagram

\begin{equation}\label{eq-02}
%
\end{equation}
where $\xi$ and $\gamma$ are the following compositions:

\begin{gather}
\label{eq-126}
%
\end{gather}

Given any other pair of representatives \eqref{eq-138} for $\Gamma^1$ and $\Gamma^2$, and any other
set of data as in \emphatic{\F{4}}, the diagram \eqref{eq-02} induced by the new data
represents the same $2$-morphism as \eqref{eq-02}.
\end{prop}

In particular, \emph{the vertical composition in $\CATC\left[\SETWinv\right]$ does not depend on the
set of choices} $\CW$ \emph{nor on the set of choices} $\DW$,
hence it is the same in any bicategory of fractions constructed from the pair $(\CATC,\SETW)$.

\begin{rem}
Since $\Gamma^2$ is a $2$-morphism in $\CATC\left[\SETWinv\right]$, both $\operatorname{w}^2$
and $\operatorname{w}^2\circ\operatorname{u}^3$ belong to $\SETW$. Therefore by Lemma~\ref{lem-05}
a set of choices as in \F{4} always exists, but in general it is not unique.
\end{rem}

The proof of Proposition~\ref{prop-02} mostly relies on the following:

\begin{lem}\label{lem-08}
Let us assume the same notations as Proposition~\ref{prop-02}. Moreover, let us choose:

\begin{description}
 \descitem{$(\operatorname{F4})'$:}{F4prime}
  an object $D$, a morphism $\operatorname{p}^1$ in $\SETW$, a morphism
  $\operatorname{p}^2$ and an invertible $2$-morphism $\mu$ as follows: 

  \[

\]
\end{description}

Let $\varphi$ and $\delta$ be the following compositions:

\begin{gather}
\label{eq-127}
%
\end{gather}

Then the following diagram 

\begin{equation}\label{eq-11}
%
  \end{equation}
represents the same $2$-cell of $\CATC\left[\SETWinv\right]$ as diagram \eqref{eq-02}.
\end{lem}

\begin{proof}
\textbf{Step 1.}
Let us apply Lemma~\ref{lem-04} on the following pair of diagrams:

\[
%
\]

Then there are a pair of objects $E^1$ and $E^2$,
a triple of morphisms $\operatorname{q}^1$, $\operatorname{q}^2$ and
$\operatorname{q}^3$, such that both $\operatorname{t}^1\circ
\operatorname{q}^1$ and $\operatorname{q}^3$ belong to $\SETW$,
and a pair of invertible $2$-morphisms $\varepsilon$ and $\kappa$ in $\CATC$ as follows:

\[
%
\]
such that $\rho\circ\operatorname{q}^1\circ\operatorname{q}^3$ is equal to the
composition of the following diagram:

\begin{equation}\label{eq-52}
%
\end{equation}

\textbf{Step 2.}
Using the definition of $\xi$ in \eqref{eq-126} together with \eqref{eq-52}, we get
that $\xi\,\circ\,\operatorname{q}^1\circ\operatorname{q}^3$ is equal to the following
composition:

\begin{equation}\label{eq-31}
%
\end{equation}

Using \eqref{eq-31} and the definition of $\varphi$ in \eqref{eq-127}, we conclude that
$\xi\circ\operatorname{q}^1\circ\operatorname{q}^3$ is also equal to the following
composition:

\[
%
\]

\textbf{Step 3.}
Using the definition of $\gamma$ in \eqref{eq-107} together with \eqref{eq-52}, we get
that $\gamma\,\circ\,\operatorname{q}^1\circ\operatorname{q}^3$ is equal to the following
composition:

\begin{equation}\label{eq-117}
%
\end{equation}

Using \eqref{eq-117} and the definition of $\delta$ in \eqref{eq-114}, we conclude that
$\gamma\circ\operatorname{q}^1\circ\operatorname{q}^3$ is also equal to the following
composition:

\[
%
\]

\textbf{Step 4.}
Using Steps 2 and 3 together with
Lemma~\ref{lem-02} on the (classes of the) diagrams \eqref{eq-02} and \eqref{eq-11},
with the data (1) -- (5) given by $(E^1,E^2,\operatorname{q}^1,\operatorname{q}^2,
\operatorname{q}^3,\operatorname{u}^1\circ\,\varepsilon,\operatorname{u}^4\circ\,
\kappa)$, we conclude that the classes of the diagrams \eqref{eq-02} and \eqref{eq-11}
have representatives that are equal. In other terms, \eqref{eq-02} and \eqref{eq-11}
represent the same $2$-cell of $\CATC\left[\SETWinv\right]$, as we wanted to prove.
\end{proof}

\begin{proof}[of Proposition~\ref{prop-02}.]
Following~\cite[p.~258]{Pr}, the vertical composition $\Gamma^2\circ\Gamma^1$ can
be computed by choosing any pair of representatives for $\Gamma^1$ and $\Gamma^2$
(as we did in \eqref{eq-138}). Having fixed such
representatives, the construction of the vertical composition in~\cite{Pr} is a special case
of the construction of Proposition~\ref{prop-02}. In particular, the
construction in~\cite{Pr} a priori holds only for a specific set of data
as in \F{4}, induced by the fixed
set of choices $\CW$ and $\DW$.
Lemma~\ref{lem-08} proves that any other set of
data as in \F{4} induces the same $2$-cell of $\CATC\left[\SETWinv\right]$.
Since the construction in~\cite{Pr}, does not depend on the representatives
for $\Gamma^1$ and $\Gamma^2$, the construction in Proposition~\ref{prop-02}
also does not depend on the
chosen representatives for $\Gamma^1$ and $\Gamma^2$, so we conclude.
\end{proof}

\section{Horizontal composition with \texorpdfstring{$1$}{1}-morphisms on the left}
\begin{prop}\label{prop-03}\emph{\textbf{(horizontal composition with $1$-morphisms on the left)}}
Let us fix any morphism $\underline{f}:=(A',\operatorname{w},f):A\rightarrow B$, any pair of
morphisms $\underline{g}^m:=(B^m,\operatorname{v}^m,g^m):B\rightarrow C$ for $m=1,2$, and any
$2$-morphism $\Delta:\underline{g}^1\Rightarrow\underline{g}^2$ in $\CATC\left[\SETWinv\right]$.
Let us fix also any representative for $\Delta$ as below:

\begin{equation}\label{eq-100}

\end{equation}

Let us suppose that the fixed set of choices $\CW$ gives
data as in the upper part of the following diagrams, with $\operatorname{w}^1$ and
$\operatorname{w}^2$ in $\SETW$, and $\rho^1$ and $\rho^2$ invertible:

\begin{equation}\label{eq-72}
%
\end{equation}
\emphatic{(}so that by~\cite[\S~2.2]{Pr} we have $\underline{g}^m\circ\underline{f}=(D^m,
\operatorname{w}\circ\operatorname{w}^m,g^m\circ f^m)$ for $m=1,2$\emphatic{)}. Then let us fix
\emph{any set of data} in $\CATC$ as follows:

\begin{description}
 \descitem{$(\operatorname{F5})$:}{F5} we choose data as in the upper part of the following diagrams,
  with $\operatorname{t}^1$ and $\operatorname{t}^2$ in $\SETW$,
  and $\sigma^1$ and $\sigma^2$ invertible:
  
  \begin{equation}\label{eq-105}
  %
  \end{equation}

 \descitem{$(\operatorname{F6})$:}{F6}
  we choose any set of data as in the upper part of the following diagram, with
  $\operatorname{t}^3$ in $\SETW$ and $\varphi$ invertible:
  
  \[
  %
  \]

 \descitem{$(\operatorname{F7})$:}{F7}
  we choose any invertible $2$-morphism $\delta:h^1\circ\operatorname{t}^3
  \Rightarrow h^2\circ\operatorname{t}^4$, such that $\operatorname{v}^1\circ
  \operatorname{u}^1\circ\,\delta$ is equal to the following composition:

  \begin{equation}\label{eq-36}
  %
  \end{equation}
\end{description}

Then $\Delta\circ\underline{f}$ is represented by the following diagram

\begin{equation}\label{eq-27}
%
\end{equation}
where $\xi$ is the following composition:

\begin{equation}\label{eq-34}
%
\end{equation}

Given any other representative \eqref{eq-100} for $\Delta$, and any other set of data
as in \emphatic{\F{5}} -- \emphatic{\F{7}}, the diagram
\eqref{eq-27} induced by the new data represents the same $2$-morphism as \eqref{eq-27}.
\end{prop}

Therefore, \emph{each composition of the form} $\Delta\circ\underline{f}$ 
\emph{does not depend on the set of choices} $\DW$.

\begin{rem}
Since \eqref{eq-100} represents a $2$-morphism in a 
bicategory of fractions, both $\operatorname{v}^1$ and $\operatorname{v}^1\circ
\operatorname{u}^1$ belong to $\SETW$, so by Lemma~\ref{lem-05} there are data as on
the left hand side of \eqref{eq-105}. Since $\alpha$ is an invertible $2$-morphism of $\CATC$,
by \bfFive{} also $\operatorname{v}^2\circ
\operatorname{u}^2$ belongs to $\SETW$; since $\underline{g}^2$ is a morphism in $\CATC
\left[\SETWinv\right]$, $\operatorname{v}^2$ belongs to $\SETW$. So using again
Lemma~\ref{lem-05} there are data as on the right hand side of \eqref{eq-105}, hence we
can always find data satisfying condition \F{5}. By axiom \bfTwo{}
the composition $\operatorname{w}^2\circ\operatorname{t}^2$ belongs to $\SETW$, hence
a set of data as in \F{6} exists by \bfThree. Since 
$\operatorname{v}^1\circ\operatorname{u}^1$
belongs to $\SETW$, a set of data as in \F{7} exists by
\bfFour, up to replacing $\varphi$ in \F{6} with the composition
of $\varphi$ with a suitable morphism in $\SETW$ with target in $D^5$ (this is analogous to the
procedure explicitly described in Remark~\ref{rem-01}). So a set of data as in
\F{5} -- \F{7} exists, but in general it is not unique.
\end{rem}

The proof of Proposition~\ref{prop-03} relies on the following:

\begin{lem}\label{lem-09}
Let us assume the same notations as Proposition~\ref{prop-03}. Moreover, let us fix
\emph{any set of data} in $\CATC$ as follows:

\begin{description}
 \descitem{$(\operatorname{F5})'$:}{F5prime}
  we choose data as in the upper part of the following diagrams,
  with $\operatorname{z}^1$ and $\operatorname{z}^2$ in $\SETW$,
  and $\gamma^1$ and $\gamma^2$ invertible:
  
  \begin{equation}\label{eq-55}

  \end{equation}

 \descitem{$(\operatorname{F6})'$:}{F6prime}
  we choose any set of data as in the upper part of the following diagram, with
  $\operatorname{z}^3$ in $\SETW$ and $\nu$ invertible:
  
  \[
  %
  \]
  
 \descitem{$(\operatorname{F7})'$:}{F7prime}
  we choose any invertible $2$-morphism $\eta:l^1\circ\operatorname{z}^3
  \Rightarrow l^2\circ\operatorname{z}^4$, such that $\operatorname{v}^1\circ
  \operatorname{u}^1\circ\,\eta$ is equal to the following composition:

  \begin{equation}\label{eq-54}
  %
  \end{equation}
\end{description}

Let $\omega$ be the following composition:

\begin{equation}\label{eq-128}
%
\end{equation}

Then the following diagram

\begin{equation}\label{eq-09}
%
\end{equation}
represents the same $2$-cell of $\CATC\left[\SETWinv\right]$ as \eqref{eq-27}.
\end{lem}

\begin{proof}
\textbf{Step 1.}
Let us apply Lemma~\ref{lem-04} on the following pair of diagrams:

\[
%
\]

Then there are a pair of objects $F^1$ and $F^2$, a triple of morphisms
$\operatorname{p}^1$, $\operatorname{p}^2$ and $\operatorname{p}^3$,
such that both $\operatorname{t}^1\circ\operatorname{t}^3\circ
\operatorname{p}^1$ and $\operatorname{p}^3$ belong to $\SETW$,
and a pair of invertible $2$-morphisms $\varepsilon$ and $\kappa$
in $\CATC$ as follows:

\[
%
\]
such that $\varphi\circ\operatorname{p}^1\circ\operatorname{p}^3$ 
is equal to the following composition:

\begin{equation}\label{eq-35}
%
\end{equation}

\textbf{Step 2.}
Using \eqref{eq-35} we get that
$(\operatorname{w}\circ\,\varphi)\circ\operatorname{p}^1\circ\operatorname{p}^3$ 
is equal to the following composition:

\begin{equation}\label{eq-129}
%
\end{equation}

Therefore, using Lemma~\ref{lem-02} and \eqref{eq-129}, diagram
\eqref{eq-27} is equivalent to

\begin{equation}\label{eq-37}
%
\end{equation}
and diagram \eqref{eq-09} is equivalent to

\begin{equation}\label{eq-40}
%
\end{equation}
where $\zeta$ is the following composition:

\begin{equation}\label{eq-41}
%
\end{equation}
(i.e., the vertical composition of the $2$-morphism $\omega$, defined in \eqref{eq-128}, with
$\varepsilon$ and $\kappa$).\\

\textbf{Step 3.}
The claim is now equivalent to proving that \eqref{eq-37} and \eqref{eq-40} 
represent the same $2$-cell of $\CATC\left[\SETWinv\right]$.
For this, it will be sufficient to prove that $\xi\,\circ\,\operatorname{p}^1\circ
\operatorname{p}^3$ and $\zeta$ are equal if pre-composed with a suitable morphism
of $\SETW$; then we will conclude by applying Lemma~\ref{lem-01}(ii).\\

\textbf{Step 4.}
By construction $\varphi\circ\operatorname{p}^1\circ\operatorname{p}^3$
coincides with \eqref{eq-35}. Therefore the composition

\begin{equation}\label{eq-118}

\end{equation}
is equal to the following one:

\begin{equation}\label{eq-119}
%
\end{equation}

\textbf{Step 5.}
Now let us denote by $\varsigma$ the following composition:

\begin{equation}\label{eq-48}
%
\end{equation}

Since $\Delta$ is a $2$-morphism in $\CATC\left[\SETWinv\right]$ with representative
\eqref{eq-100}, $\operatorname{v}^1\circ\operatorname{u}^1$ belongs to $\SETW$. Using
\bfFive{} on $\alpha$, we get that
$\operatorname{v}^2\circ\operatorname{u}^2$ also belongs to $\SETW$. Moreover,
$\operatorname{v}^2$ also belongs to $\SETW$ because the triple
$\underline{g}^2=(B^2,\operatorname{v}^2,g^2)$ is a morphism
in $\CATC\left[\SETWinv\right]$. Therefore we can apply Lemma~\ref{lem-06} to $\varsigma$, so
there are an object $F^3$, a morphism $\operatorname{p}^4:F^3\rightarrow F^2$ in $\SETW$
and a $2$-morphism

\begin{equation}\label{eq-51}
\widetilde{\varsigma}:\,h^2\circ\operatorname{t}^4\circ\operatorname{p}^1\circ\operatorname{p}^3
\circ\operatorname{p}^4\Longrightarrow l^2\circ\operatorname{z}^4
\circ\operatorname{p}^2\circ\operatorname{p}^3\circ\operatorname{p}^4
\end{equation}
in $\CATC$, such that $\varsigma\circ\operatorname{p}^4=\operatorname{u}^2
\circ\,\widetilde{\varsigma}$. 
$\varsigma$ is invertible since each $2$-morphism in \eqref{eq-48} is invertible in $\CATC$;
then using again Lemma~\ref{lem-06}, we get that also $\widetilde{\varsigma}$ is invertible.
Using the definition of $\varsigma$ in \eqref{eq-48} and the fact that $\varsigma
\circ\operatorname{p}^4=\operatorname{u}^2\circ\,\widetilde{\varsigma}$, we get that
the composition

\begin{equation}\label{eq-46}

\end{equation}
is equal to the following one:

\begin{equation}\label{eq-44}
%
\end{equation}

\textbf{Step 6.} Let $\mu$ be the following composition:

\begin{equation}\label{eq-42}
%
\end{equation}

Let us consider the $2$-morphism $\operatorname{v}^1\circ\,\mu$: we
use \eqref{eq-54} and we simplify $\gamma^1$ (coming from \eqref{eq-42})
with its inverse (coming from \eqref{eq-54}), so we get that
$\operatorname{v}^1\circ\,\mu$ is equal to the following composition:

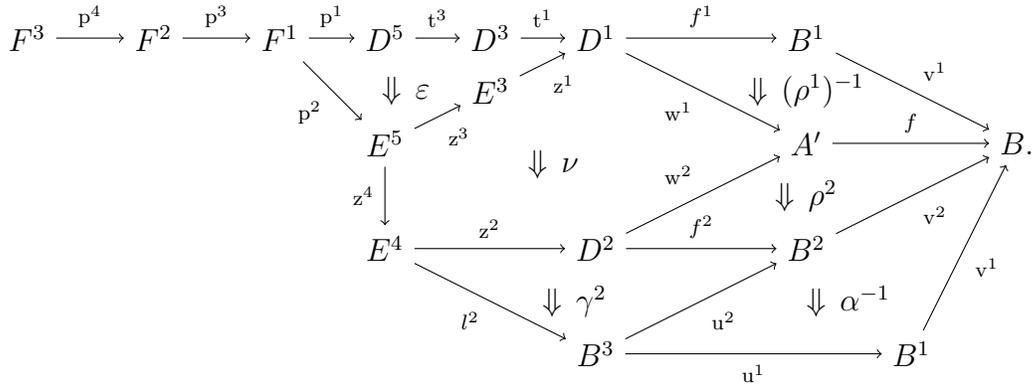
\[
%
\]

In the previous diagram we can replace \eqref{eq-118} with \eqref{eq-119}, so we get that 
$\operatorname{v}^1\circ\,\mu$ coincides also with the following composition:

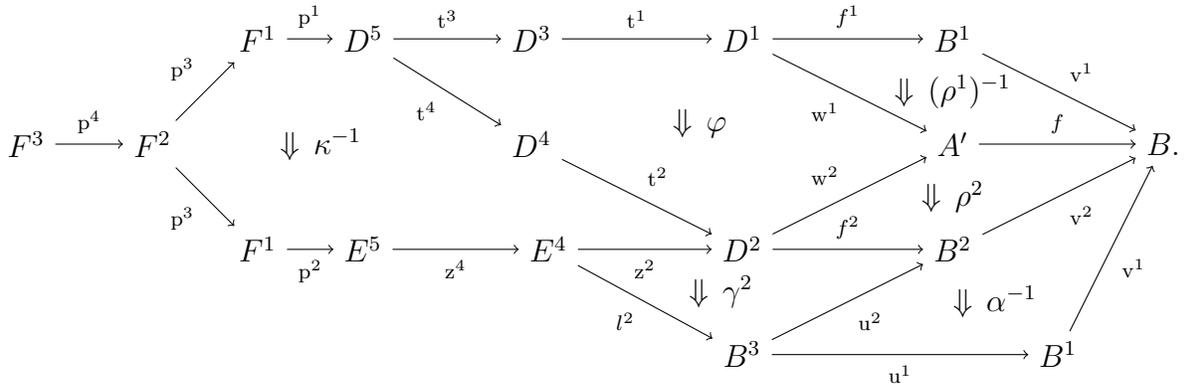
\[
%
\]

In the previous diagram we can replace \eqref{eq-46} with \eqref{eq-44}. So we conclude that
$\operatorname{v}^1\circ\,\mu$ coincides also with the following composition:

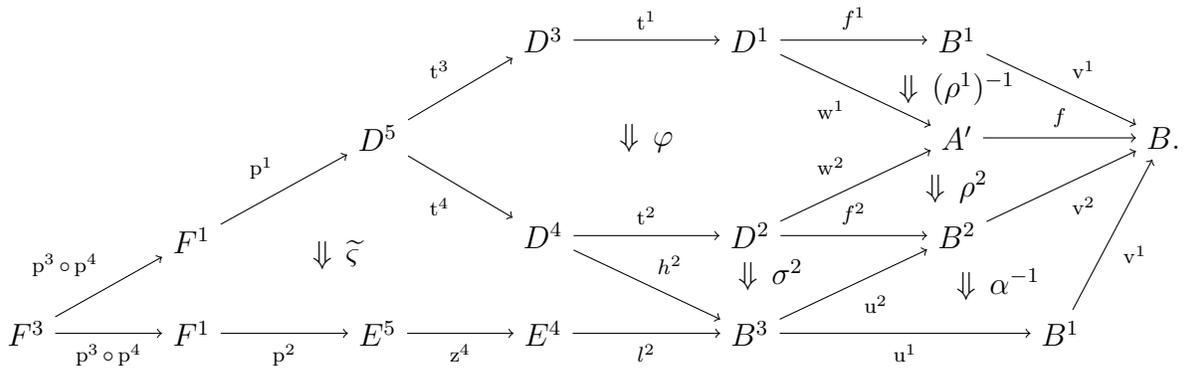
\[
%
\]

If we compare this diagram with \eqref{eq-36}, we conclude that
$\operatorname{v}^1\circ\,\mu$ is also equal to the following composition:

\begin{equation}\label{eq-120}
%
\end{equation}

Since the triple $\underline{g}^1=(B^1,\operatorname{v}^1,g^1)$ is a morphism in $\CATC\left[
\SETWinv\right]$, $\operatorname{v}^1$ belongs to $\SETW$. Since $\operatorname{v}^1\circ\,\mu$
coincides with \eqref{eq-120}, by Lemma~\ref{lem-03} there are an object $F^4$ and a
morphism $\operatorname{p}^5:F^4\rightarrow F^3$ in $\SETW$, such that $\mu\circ
\operatorname{p}^5$ is equal to the following composition:

\begin{equation}\label{eq-121}
%
\end{equation}

Using the definition of $\mu$ in \eqref{eq-42}, we conclude that \eqref{eq-121} is equal to
the following composition:

\begin{equation}\label{eq-122}
%
\end{equation}

\textbf{Step 7.}
Let us consider the $2$-morphism $\zeta\circ\operatorname{p}^4\circ\operatorname{p}^5$,
where $\zeta$ is defined in \eqref{eq-41}. If we replace
\eqref{eq-122} with \eqref{eq-121}, we conclude that $\zeta\circ\operatorname{p}^4
\circ\operatorname{p}^5$ is equal to the following composition:

\begin{equation}\label{eq-130}
%
\end{equation}

In the diagram above we replace $\operatorname{u}^2\circ\,\widetilde{\varsigma}$
with $\varsigma\circ
\operatorname{p}^4$ (see \eqref{eq-51}), then we replace $\varsigma$ with diagram
\eqref{eq-48}. Lastly, we simplify $\gamma^2$ and $\kappa^{-1}$ (coming from \eqref{eq-48})
with their inverses (coming from \eqref{eq-130}). So we conclude that
$\zeta\circ\operatorname{p}^4\circ\operatorname{p}^5$ coincides 
also with the following composition:

\[
%
\]

If we compare this with the definition of $\xi$ in \eqref{eq-34}, we conclude that:

\[\zeta\circ(\operatorname{p}^4\circ\operatorname{p}^5)=
(\xi\circ\operatorname{p}^1\circ\operatorname{p}^3)\circ(\operatorname{p}^4
\circ\operatorname{p}^5).\]

\textbf{Step 8.}
Since both $\operatorname{p}^4$ and $\operatorname{p}^5$ belong to
$\SETW$, so does their composition by \bfTwo. So the previous identity together
with Lemma~\ref{lem-01}(ii) proves that \eqref{eq-37} and \eqref{eq-40} are equivalent
diagrams, so we have proved the claim (see Step 3).
\end{proof}

\begin{proof}[of Proposition~\ref{prop-03}.]

Following~\cite[pp.~259--261]{Pr}, in order to compute $\Delta\circ\underline{f}$ 
one can choose freely any representative for $\Delta$ as in \eqref{eq-100} (and the result does not
depend on this choice). Then after a very long construction, one gets a set of data
satisfying conditions \F{5} -- \F{7} and some additional
properties, that we don't need to recall here. Such data are uniquely determined
by the choices of type $\CW$ and $\DW$. Then the
construction of $\Delta\circ\underline{f}$ in~\cite{Pr} is simply 
a special case of the construction described in
Proposition~\ref{prop-03}, with a special set of data \F{5} -- \F{7},
Using Lemma~\ref{lem-09}, any set of data \F{5} -- \F{7} gives the
same $2$-cell of $\CATC\left[\SETWinv\right]$, so we conclude. 

\end{proof}

\section{Horizontal composition with \texorpdfstring{$1$}{1}-morphisms on the right}
\begin{prop}\label{prop-04}\emph{\textbf{(horizontal composition with $1$-morphisms on the right)}}
Let us fix any morphism $\underline{g}=(B',\operatorname{u},g):B\rightarrow C$, any pair of
morphisms $\underline{f}^m:=(A^m,\operatorname{w}^m,f^m):A\rightarrow B$ for $m=1,2$ and any
$2$-morphism $\Gamma:\underline{f}^1\Rightarrow\underline{f}^2$ in $\CATC\left[\SETWinv\right]$.
Let us fix any representative for $\Gamma$ as in \eqref{eq-16} and let us suppose that the set 
of choices $\CW$ gives data as in the upper part of the following
diagrams, with $\operatorname{u}^1$ and $\operatorname{u}^2$ in $\SETW$,
and $\rho^1$ and $\rho^2$ invertible

\begin{equation}\label{eq-73}

\end{equation}
\emphatic{(}so that by~\cite[\S~2.2]{Pr} we have $\underline{g}\circ\underline{f}^m=(D^m,
\operatorname{w}^m\circ\operatorname{u}^m,g\circ h^m)$ for $m=1,2$\emphatic{)}.
Then let us fix \emph{any set of data} in $\CATC$ as follows:

\begin{description}
 \descitem{$(\operatorname{F8})$:}{F8}
  we choose data as in the upper part of the following diagrams, with
  $\operatorname{u}^3$ and $\operatorname{u}^5$ in $\SETW$, and $\eta^1$ and $\eta^2$ invertible:
  
  \[
  %
  \]

 \descitem{$(\operatorname{F9})$:}{F9}
  we choose data as in the upper part of the following
  diagram, with $\operatorname{u}^7$ in $\SETW$ and $\eta^3$ invertible:

  \[
  %
  \]

 \descitem{$(\operatorname{F10})$:}{F10}
  we choose any $2$-morphism $\lambda:h^1\circ\operatorname{u}^4
  \circ\operatorname{u}^7\Rightarrow h^2\circ\operatorname{u}^6\circ
  \operatorname{u}^8$, such that $\operatorname{u}\circ\,\lambda$ is equal to the following
  composition:

  \begin{equation}\label{eq-49}
  %
  \end{equation}
\end{description}

Then $\underline{g}\circ\Gamma$ is represented by the following diagram:

\begin{equation}\label{eq-28}
%
\end{equation}
where $\mu$ is the following composition:

\begin{equation}\label{eq-24}
%
\end{equation}

Given any other representative \eqref{eq-16} for $\Gamma$, and any other set of data as in
\emphatic{\F{8}} -- \emphatic{\F{10}}, the diagram \eqref{eq-28}
induced by the new data is equivalent to diagram \eqref{eq-28}.
\end{prop}

So the horizontal composition of the form $\underline{g}\circ\Gamma$
\emph{does not depend on the set of choices} $\DW$.\\

\begin{rem}
By hypothesis $\operatorname{u}^1$ and $\operatorname{u}^2$ belong to $\SETW$, hence a set of
data as in \F{8} exists by \bfThree; analogously also
a set of data as in \F{9} exists. Since $\underline{g}=(B',\operatorname{u},g)$
is a morphism in $\CATC\left[\SETWinv\right]$, $\operatorname{u}$ belongs
to $\SETW$; hence a set of of data as in \F{10} exists by \bfFour,
up to replacing $\eta^3$ with the composition of $\eta^3$ with a suitable morphism in $\SETW$
with target in $D^5$ (analogously to the construction explained in Remark~\ref{rem-01}).
\end{rem}

The proof of Proposition~\ref{prop-04} relies on the following:

\begin{lem}\label{lem-10}
Let us assume the same notations as Proposition~\ref{prop-04}. Moreover,
let us fix \emph{any set of data} in $\CATC$ as follows:

\begin{description}
 \descitem{$(\operatorname{F8})'$:}{F8prime}
  we choose data as in the upper part of the following diagrams, with
  $\operatorname{t}^3$ and $\operatorname{t}^5$ in $\SETW$, and $\xi^1$ and $\xi^2$ invertible:
  
  \[

  \]

 \descitem{$(\operatorname{F9})'$:}{F9prime}
  we choose data as in the upper part of the following
  diagram, with $\operatorname{t}^7$ in $\SETW$ and $\xi^3$ invertible:

  \[
  %
  \]

 \descitem{$(\operatorname{F10})'$:}{F10prime}
  we choose any $2$-morphism $\varphi:h^1\circ\operatorname{t}^4
  \circ\operatorname{t}^7\Rightarrow h^2\circ\operatorname{t}^6\circ
  \operatorname{t}^8$, such that $\operatorname{u}\circ\,\varphi$ is equal to the following
  composition:

  \[
  %
  \]
\end{description}

Let $\delta$ be the following composition:

\begin{equation}\label{eq-33}
%
\end{equation}

Then the following diagram:

\begin{equation}\label{eq-115}
%
\end{equation}
represents the same $2$-cell of $\CATC\left[\SETWinv\right]$ as \eqref{eq-28}.
\end{lem}

\begin{proof}
\textbf{Step 1.}
We want to apply Lemma~\ref{lem-04} for the following pair of diagrams:

\begin{equation}\label{eq-83}
%
\end{equation}

In order to do this, let us prove that all the relevant morphisms belong to $\SETW$:
\begin{itemize}
 \item $\operatorname{w}^1$ belongs to $\SETW$ since \eqref{eq-16} represents the $2$-cell
  $\Gamma$ of $\CATC\left[\SETWinv\right]$;
 \item $\operatorname{u}^1$ belongs to $\SETW$ by hypothesis, so by \bfTwo{}
  $\operatorname{w}^1\circ\operatorname{u}^1$ also belongs to $\SETW$;
 \item since \eqref{eq-16} represents a $2$-cell of $\CATC\left[\SETWinv\right]$,
  also $\operatorname{w}^1\circ\operatorname{v}^1$ belongs to $\SETW$; moreover
  $\operatorname{u}^3$ belongs to $\SETW$ by \F{8}; by \bfTwo{}
  this implies that $\operatorname{w}^1\circ\operatorname{v}^1\circ\operatorname{u}^3$
  belongs to $\SETW$. Using \bfFive{} on the invertible $2$-morphism
  $\operatorname{w}^1\circ\,\eta^1$, we conclude that also $\operatorname{w}^1\circ
  \operatorname{u}^1\circ\operatorname{u}^4$ belongs to $\SETW$. By \F{9}
  $\operatorname{u}^7$ belongs to $\SETW$, so by \bfTwo{}
  $\operatorname{w}^1\circ\operatorname{u}^1\circ\operatorname{u}^4\circ
  \operatorname{u}^7$ belongs to $\SETW$;
 \item analogously, using \F{8}$'$ and \F{9}$'$ we get that 
  $\operatorname{w}^1\circ\operatorname{u}^1\circ\operatorname{t}^4\circ
  \operatorname{t}^7$ also belongs to $\SETW$.
\end{itemize}

So we can apply Lemma~\ref{lem-04} on \eqref{eq-83}.
Then there are a pair of objects $F^1$, $F^2$, a triple of morphisms
$\operatorname{r}^1$, $\operatorname{r}^2$ and $\operatorname{r}^3$, such that
both $\operatorname{u}^4\circ\operatorname{u}^7\circ\operatorname{r}^1$
and $\operatorname{r}^3$ belong to $\SETW$, and a pair of invertible 
$2$-morphisms $\omega$ and $\tau$ in $\CATC$ as follows:

\[

\]
such that the composition

\begin{equation}\label{eq-13}
%
\end{equation}
coincides with the following one:

\begin{equation}\label{eq-18}
%
\end{equation}

\textbf{Step 2.}
Let us apply Lemma~\ref{lem-04} for the following pair of diagrams:

\begin{equation}\label{eq-56}
%
\end{equation}

Both $\operatorname{w}^2$ and $\operatorname{u}^2$ belong to $\SETW$
by hypothesis, hence also $\operatorname{w}^2\circ\operatorname{u}^2$
belongs to $\SETW$. So in order to apply the lemma, we need only
to verify that $\operatorname{w}^2\circ\operatorname{v}^2
\circ\operatorname{u}^5\circ\operatorname{u}^8\circ\operatorname{r}^1
\circ\operatorname{r}^3$ belongs to $\SETW$.\\

The various hypotheses and the construction of Step 1 imply that
$\operatorname{w}^1$, $\operatorname{u}^1$, $\operatorname{u}^4\circ
\operatorname{u}^7\circ\operatorname{r}^1$ and $\operatorname{r}^3$
all belong to $\SETW$; hence also their composition belongs to $\SETW$
by \bfTwo. Now we consider the invertible $2$-morphism
defined as the following composition:

\[

\]

The target of this $2$-morphism belongs to $\SETW$, so by \bfFive{}
also the source belongs to $\SETW$. So we can apply 
Lemma~\ref{lem-04} on \eqref{eq-56}. Therefore there are a pair of
objects $R^1$, $R^2$, a triple of morphisms $\operatorname{q}^1$,
$\operatorname{q}^2$ and $\operatorname{q}^3$, such that both
$\operatorname{q}^1$ and $\operatorname{q}^3$ belong to $\SETW$,
and a pair of invertible $2$-morphisms $\chi$ and $\zeta$ as follows

\[
%
\]
such that the composition

\begin{equation}\label{eq-20}
%
\end{equation}
coincides with the following one:

\begin{equation}\label{eq-21}
%
\end{equation}

\textbf{Step 3.} Let us consider the $2$-cell of $\CATC$

\begin{equation}\label{eq-23}
\operatorname{u}\circ\,\lambda\circ\operatorname{r}^1\circ\operatorname{r}^3
\circ\operatorname{q}^1\circ\operatorname{q}^3.
\end{equation}

We replace $\operatorname{u}\circ\,\lambda$ with \eqref{eq-49}, then
we replace \eqref{eq-13} with \eqref{eq-18} and
\eqref{eq-20} with \eqref{eq-21}; lastly we simplify $\tau$ with $\tau^{-1}$. So we get that
\eqref{eq-23} is equal to the composition

\[
%
\]

Comparing with \FtenPrime, we get that \eqref{eq-23} is also equal to
the following composition:

\begin{equation}\label{eq-53}
%
\end{equation}

Since the triple $\underline{g}=(B',\operatorname{u},g)$ is a morphism in
$\CATC\left[\SETWinv\right]$, $\operatorname{u}$ belongs to $\SETW$. So the equality
of \eqref{eq-23} with \eqref{eq-53} and Lemma~\ref{lem-03} imply that there are
an object $R^3$ and a morphism $\operatorname{q}^4:R^3\rightarrow R^2$ in $\SETW$, such that

\begin{equation}\label{eq-57}
g\circ\lambda\circ\operatorname{r}^1\circ\operatorname{r}^3\circ\operatorname{q}^1
\circ\operatorname{q}^3\circ\operatorname{q}^4
\end{equation}
is equal to the following composition:

\begin{equation}\label{eq-58}

\end{equation}

\textbf{Step 4.} Now let us consider the $2$-morphism of $\CATC$

\begin{equation}\label{eq-142}
\mu\circ\operatorname{r}^1\circ\operatorname{r}^3\circ\operatorname{q}^1\circ\operatorname{q}^3
\circ\operatorname{q}^4,
\end{equation}
where $\mu$ defined in \eqref{eq-24}. Replacing \eqref{eq-13} with \eqref{eq-18}, \eqref{eq-20}
with \eqref{eq-21}, and simplifying $\tau$ with $\tau^{-1}$, we get that \eqref{eq-142}
is equal to the following composition:

\[
%
\]

Comparing with the definition of $\delta$ in \eqref{eq-33}, we conclude that \eqref{eq-142}
is also equal to the following composition:

\begin{equation}\label{eq-59}
%
\end{equation}

\textbf{Step 5.} Now let $\varepsilon$ be the composition

\[
%
\]
and let $\kappa$ be the composition

\[
%
\]

Then let us apply Lemma~\ref{lem-02} on the pair of $2$-cells of $\CATC\left[\SETWinv\right]$
given by (the classes of) \eqref{eq-28} and \eqref{eq-115}, with the choice of
$\varepsilon$ and $\kappa$ as above. Then the equalities \eqref{eq-57}$=$\eqref{eq-58}
and \eqref{eq-142}$=$\eqref{eq-59} prove that
the classes of \eqref{eq-28} and \eqref{eq-115} admit representatives that are equal,
hence \eqref{eq-28} and \eqref{eq-115} are equivalent, as we wanted to prove.
\end{proof}

\begin{proof}[of Proposition~\ref{prop-04}.]

The construction of $\underline{g}\circ\Gamma$ in~\cite[p.~259]{Pr} begins by choosing
any representative \eqref{eq-16} for $\Gamma$; then it is simply a particular case
of the construction described in the statement of Proposition~\ref{prop-04}. To be more
precise, the construction in~\cite{Pr} is done by selecting a specific set
of data \F{8} -- \F{10}, uniquely determined by the choices
of type $\CW$ and $\DW$. Using Lemma~\ref{lem-10},
any set of data \F{8} -- \F{10} gives the same
$2$-cell in $\CATC\left[\SETWinv\right]$, so we conclude.
Note that the construction described in
Proposition~\ref{prop-04} does not depend on the representative \eqref{eq-16} chosen for $\Gamma$
because the construction of~\cite{Pr} does not depend on this choice
(see~\cite[p.~259]{Pr}).
\end{proof}

\section{Some useful applications}\label{sec-03}
As we stated in the Introduction, Theorem~\ref{thm-01} and Corollary~\ref{cor-03}
follow immediately from Propositions~\ref{prop-01}, \ref{prop-02}, \ref{prop-03}
and~\ref{prop-04}. As we remarked
in the Introduction, these four propositions are also interesting in their own,
since they allow us to easily calculate various compositions of $2$-cells
in a bicategory of fractions, only knowing the set of
fixed choices $\CW$ (that completely determines the bicategory because of
Theorem~\ref{thm-01}), and choosing freely any
representatives of the $2$-morphisms that we are composing (plus the additional data as in
\F{1} -- \F{10}). For example, having shown a simple procedure
for computing vertical compositions (i.e., Proposition~\ref{prop-02}),
we can prove the already mentioned Proposition~\ref{prop-06} about invertibility of $2$-morphisms
in $\CATC\left[\SETWinv\right]$:

\begin{proof}[of Proposition~\ref{prop-06}.]
Let us assume (iv); by hypothesis $\underline{f}^1$ is a morphism of
$\CATC\left[\SETWinv\right]$, hence $\operatorname{w}^1$ belongs to $\SETW$. 
Since $\operatorname{v}^1\circ\operatorname{u}$ belongs to $\SETW$, so does
$\operatorname{w}^1\circ\operatorname{v}^1\circ\operatorname{u}$.
Since \eqref{eq-16} represents $\Gamma$, then $\operatorname{w}^1\circ
\operatorname{v}^1$ also belongs to $\SETW$. So by Lemma~\ref{lem-11}
there are an object $A^{\prime 4}$ and a morphism $\operatorname{u}':A^{\prime 4}
\rightarrow A^4$ such that $\operatorname{u}\circ\operatorname{u}'$
belongs to $\SETW$. Since $\beta\circ\operatorname{u}$
is invertible, so is $\beta\circ(\operatorname{u}\circ\operatorname{u}')$,
hence (iii) holds.\\

Let us assume (iii), so let us choose any representative \eqref{eq-16} for $\Gamma$
and let us assume that there are an object $A^4$ and a morphism
$\operatorname{u}:A^4\rightarrow A^3$ in $\SETW$, such that $\beta\circ\operatorname{u}$
is invertible in $\CATC$. By the description of $2$-morphisms in $\CATC\left[\SETWinv\right]$,
we have that $\Gamma$ is also represented by the quintuple $(A^4,\operatorname{v}^1\circ
\operatorname{u},\operatorname{v}^2\circ\operatorname{u},\alpha\circ\operatorname{u},\beta\circ
\operatorname{u})$. Since $\beta\circ\operatorname{u}$ is invertible in $\CATC$,
(ii) holds with this new representative for $\Gamma$.\\

Now let us assume (ii), so let us assume that the data in \eqref{eq-16} are such that
$\beta$ is invertible in $\CATC$. Since $\alpha$ is invertible in $\CATC$ by definition
of $2$-morphism in a bicategory of fractions, it makes sense to consider the $2$-morphism
$[A^3,\operatorname{v}^2,\operatorname{v}^1,\alpha^{-1},\beta^{-1}]:\underline{f}^2\Rightarrow
\underline{f}^1$. Using Proposition~\ref{prop-02},
it is easy to see that this is a (vertical) inverse for $\Gamma$, so (i) holds.\\

Now let us prove that (i) implies (iv), so let us assume that $\Gamma$ is invertible and let
us fix any representative \eqref{eq-16} for it; by definition of $2$-morphism in $\CATC\left[
\SETWinv\right]$ we have that $\alpha$ is invertible, so we can apply Corollary~\ref{cor-04}
with $\underline{f}^1$ interchanged with $\underline{f}^2$ and $(E,\operatorname{p},
\operatorname{q},\varsigma,\Phi)$ replaced by $(A^3,\operatorname{v}^2,\operatorname{v}^1,
\alpha^{-1},\Gamma^{-1})$. So there are an object $E$, a morphism $\operatorname{r}:
E\rightarrow A^3$ such that $\operatorname{v}^2\circ\operatorname{r}$ belongs
to $\SETW$, and a $2$-morphism $\varphi:f^2\circ\operatorname{v}^2\circ
\operatorname{r}\Rightarrow f^1\circ\operatorname{v}^1\circ\operatorname{r}$, such that
$\Gamma^{-1}$ is represented by the following diagram:

\[
\begin{tikzpicture}[xscale=1.8,yscale=-0.7]
    \node (A0_2) at (2, 0) {$A^2$};
    \node (A2_2) at (2, 2) {$E$};
    \node (A2_0) at (0, 2) {$A$};
    \node (A2_4) at (4, 2) {$B$.};
    \node (A4_2) at (2, 4) {$A^1$};

    \node (A2_3) at (2.8, 2) {$\Downarrow\,\varphi$};
    \node (A2_1) at (1.2, 2) {$\Downarrow\,\alpha^{-1}\circ\operatorname{r}$};

    \path (A4_2) edge [->]node [auto,swap] {$\scriptstyle{f^1}$} (A2_4);
    \path (A0_2) edge [->]node [auto] {$\scriptstyle{f^2}$} (A2_4);
    \path (A2_2) edge [->]node [auto,swap] {$\scriptstyle{\operatorname{v}^2
      \circ\operatorname{r}}$} (A0_2);
    \path (A2_2) edge [->]node [auto] {$\scriptstyle{\operatorname{v}^1
      \circ\operatorname{r}}$} (A4_2);
    \path (A4_2) edge [->]node [auto] {$\scriptstyle{\operatorname{w}^1}$} (A2_0);
    \path (A0_2) edge [->]node [auto,swap] {$\scriptstyle{\operatorname{w}^2}$} (A2_0);
\end{tikzpicture}
\]

Then we have

\begin{gather*}
\Big[E,\operatorname{v}^2\circ\operatorname{r},\operatorname{v}^2\circ\operatorname{r},
 i_{\operatorname{w}^2\circ\operatorname{v}^2\circ\operatorname{r}},i_{f^2\circ\operatorname{v}^2
 \circ\operatorname{r}}\Big]=i_{\left(A^2,\operatorname{w}^2,f^2\right)}= \\
=\Gamma\circ\Gamma^{-1}\stackrel{(\ast)}{=}\Big[E,
 \operatorname{v}^2\circ\operatorname{r},\operatorname{v}^2\circ\operatorname{r},
 i_{\operatorname{w}^2\circ\operatorname{v}^2\circ\operatorname{r}},
 \Big(\beta\circ\operatorname{r}\Big)\circ\varphi\Big],
\end{gather*}
where $(\ast)$ is obtained by applying Proposition~\ref{prop-02}.
So by Lemma~\ref{lem-01}(ii) there are an object $F$ and a morphism
$\operatorname{s}:F\rightarrow E$ in $\SETW$, such that:

\begin{equation}\label{eq-68}
i_{f^2\circ\operatorname{v}^2\circ\operatorname{r}\circ\operatorname{s}}=
\Big(\beta\circ\operatorname{r}\circ\operatorname{s}\Big)\circ\Big(\varphi\circ
\operatorname{s}\Big).
\end{equation}

Analogously, we have:

\begin{gather*}
\Big[F,\operatorname{v}^1\circ\operatorname{r}\circ\operatorname{s},\operatorname{v}^1
 \circ\operatorname{r}\circ\operatorname{s},i_{\operatorname{w}^1\circ\operatorname{v}^1\circ
 \operatorname{r}\circ\operatorname{s}},i_{f^1\circ\operatorname{v}^1\circ\operatorname{r}
 \circ\operatorname{s}}\Big]=i_{\left(A^1,\operatorname{w}^1,f^1\right)}= \\
=\Gamma^{-1}\circ\Gamma\stackrel{(\ast)}{=}\Big[F,
 \operatorname{v}^1\circ\operatorname{r}\circ\operatorname{s},\operatorname{v}^1\circ
 \operatorname{r}\circ\operatorname{s},i_{\operatorname{w}^1\circ\operatorname{v}^1\circ
 \operatorname{r}\circ\operatorname{s}},\Big(\varphi\circ\operatorname{s}\Big)\circ\Big(\beta
 \circ\operatorname{r}\circ\operatorname{s}\Big)\Big],
\end{gather*}
where $(\ast)$ is obtained using again Proposition~\ref{prop-02}.
So by Lemma~\ref{lem-01}(iii) there are an object $A^4$ and a morphism $\operatorname{t}:
A^4\rightarrow F$, such that $\operatorname{v}^1\circ\operatorname{r}\circ\operatorname{s}
\circ\operatorname{t}$ belongs to $\SETW$ and such that

\begin{equation}\label{eq-69}
i_{f^1\circ\operatorname{v}^1\circ\operatorname{r}\circ\operatorname{s}\circ\operatorname{t}}=
\Big(\varphi\circ\operatorname{s}\circ\operatorname{t}\Big)\circ\Big(\beta\circ\operatorname{r}
\circ\operatorname{s}\circ\operatorname{t}\Big).
\end{equation}

Then \eqref{eq-68} and \eqref{eq-69} prove that $\beta\circ\operatorname{r}\circ\operatorname{s}
\circ\operatorname{t}$ has an inverse in $\CATC$, given by $\varphi\circ\operatorname{s}\circ
\operatorname{t}$. In order to conclude that (iv) holds, it suffices to define
$\operatorname{u}:=\operatorname{r}\circ\operatorname{s}\circ\operatorname{t}:A^4
\rightarrow A^3$.
\end{proof}

A second application of the results of the previous sections is the following corollary,
that will be useful in the next paper~\cite{T4}. In that paper, we will have to compare the
compositions of $3$ morphisms of the form \eqref{eq-65} and the compositions of $3$ morphisms
of the following form

\begin{equation}\label{eq-81}
\begin{tikzpicture}[xscale=2.8,yscale=-1.2]
    \def \y {0.8}
    \def \z {0}
    \node (A0_0) at (-0.2, \z*\y) {$\underline{f}':=\Big(A$};
    \node (A0_1) at (1, \z*\y) {$A'$};
    \node (A0_2) at (2, \z*\y) {$B\Big)$,};
    \path (A0_1) edge [->]node [auto,swap] {$\scriptstyle{\operatorname{u}\circ\id_{A'}}$} (A0_0);
    \path (A0_1) edge [->]node [auto] {$\scriptstyle{f\circ\id_{A'}}$} (A0_2);

    \def \z {1}
    \node (B0_0) at (-0.2, \z*\y) {$\underline{g}':=\Big(B$};
    \node (B0_1) at (1, \z*\y) {$B'$};
    \node (B0_2) at (2, \z*\y) {$C\Big)$,};
    \path (B0_1) edge [->]node [auto,swap]{$\scriptstyle{\operatorname{v}\circ\id_{B'}}$} (B0_0);
    \path (B0_1) edge [->]node [auto] {$\scriptstyle{g}\circ\id_{B'}$} (B0_2);

    \def \z {2}
    \node (C0_0) at (-0.2, \z*\y) {$\underline{h}':=\Big(C$};
    \node (C0_1) at (1, \z*\y) {$C'$};
    \node (C0_2) at (2, \z*\y) {$D\Big)$.};
    \path (C0_1) edge [->]node [auto,swap]  {$\scriptstyle{\operatorname{w}\circ\id_{C'}}$} (C0_0);
    \path (C0_1) edge [->]node [auto] {$\scriptstyle{h\circ\id_{C'}}$} (C0_2);
\end{tikzpicture}
\end{equation}

In the special case when $\CATC$ is a $2$-category we have that $\underline{f}=\underline{f}'$ and
so on, hence $\underline{h}\circ(\underline{g}\circ\underline{f})=\underline{h}'\circ
(\underline{g}'\circ\underline{f}')$ and $(\underline{h}\circ\underline{g})\circ\underline{f}=
(\underline{h}'\circ\underline{g}')\circ\underline{f}'$. However, when $\CATC$ is simply a
bicategory, in general the fixed choice $(A^1,\operatorname{u}^1,f^1,\delta)$ in
the set $\CW$ for the pair $(f,\operatorname{v})$ (see
the first diagram in \eqref{eq-70}) is different from the fixed choice for the pair
$(f\circ\id_{A'},\operatorname{v}\circ\id_{B'})$, and analogously for all the remaining choices
needed to compose the morphisms in \eqref{eq-81}. Therefore, we need a tool for comparing
$\underline{h}'\circ(\underline{g}'\circ\underline{f}')$ with $\underline{h}\circ(\underline{g}\circ
\underline{f})$ and analogously for the other pair of compositions. In order to do that,
first of all we compare separately $\underline{f}'$ with $\underline{f}$. For that, we recall that
we are denoting by $\pi_a:a\circ\id_A\Rightarrow a$ the right unitor of $\CATC$ for any morphism
$a:A\rightarrow B$. Then we define an invertible $2$-morphism $\chi(\underline{f}):\underline{f}'
\Rightarrow\underline{f}$ as the class of the following diagram

\[

\]
and analogously for $\chi(\underline{g})$ and $\chi(\underline{h})$.  Then we have the following
result:

\begin{cor}\label{cor-02}
Let us fix any pair $(\CATC,\SETW)$ satisfying conditions \emphatic{\bf} and any
set of $6$ morphisms as in \eqref{eq-65} and \eqref{eq-81}. Then
the associator $\Theta^{\CATC,\SETW}_{\underline{h},\underline{g},\underline{f}}:
\underline{h}\circ(\underline{g}\circ\underline{f})\Rightarrow
(\underline{h}\circ\underline{g})\circ\underline{f}$ is equal to the following composition:

\begin{equation}\label{eq-91}
%
\end{equation}
\end{cor}

\begin{rem}
As we mentioned above, in the case when $\CATC$ is a $2$-category there is nothing to prove
(since $\chi(\underline{f})$, $\chi(\underline{g})$ and $\chi(\underline{h})$ are all trivial).
In the case of a bicategory, things are a bit more complicated. One could be tempted to do the
following:

\begin{enumerate}[(a)]
 \item prove that $\chi(\underline{f})$ is a right or left unitor for $\underline{f}$ in the
  bicategory $\CATC\left[\SETWinv\right]$, and analogously for $\chi(\underline{g})$ and
  $\chi(\underline{h})$;
 \item appeal to coherence results in order to conclude that necessarily the composition in
  \eqref{eq-91} coincides with the associator
  $\Theta^{\CATC,\SETW}_{\underline{h},\underline{g},\underline{f}}$.
\end{enumerate}

However, since the identities for $A$ and $B$ in $\CATC\left[\SETWinv\right]$ are the triples
$(A,\id_A,\id_A)$ and $(B,\id_B,\id_B)$ respectively, we have:

\[
\begin{tikzpicture}[xscale=2.0,yscale=-1.2]
    \node (A0_0) at (-0.2, 0) {$\underline{f}\circ\id_A=\Big(A$};
    \node (A0_1) at (1.1, 0) {$A'$};
    \node (A0_2) at (2, 0) {$B\Big)$,};
    \path (A0_1) edge [->]node [auto,swap] {$\scriptstyle{\id_A\circ\operatorname{u}}$} (A0_0);
    \path (A0_1) edge [->]node [auto] {$\scriptstyle{f\circ\id_{A'}}$} (A0_2);
    
    \def \z {3.8}
    
    \node (B0_0) at (-0.2+\z, 0) {$\id_B\circ\underline{f}=\Big(A$};
    \node (B0_1) at (1.1+\z, 0) {$A'$};
    \node (B0_2) at (2+\z, 0) {$B\Big)$.};
    \path (B0_1) edge [->]node [auto,swap] {$\scriptstyle{\operatorname{u}\circ\id_{A'}}$} (B0_0);
    \path (B0_1) edge [->]node [auto] {$\scriptstyle{\id_B\circ f}$} (B0_2);
\end{tikzpicture}
\]

Since we are working in a bicategory, in general $\id_A\circ\operatorname{u}\neq
\operatorname{u}\circ\id_{A'}$ and $\id_B\circ f\neq f\circ\id_{A'}$, therefore in general
$\underline{f}'$ is different from $\underline{f}\circ\id_A$ and from $\id_B\circ\underline{f}$.
Therefore in general $\chi(\underline{f}):\underline{f}'\Rightarrow\underline{f}$ cannot be
a left or right unitor for $\underline{f}$ in $\CATC\left[\SETWinv\right]$. So in general
we cannot use (a) (then (b)) in order to conclude. Instead, we need to rely on
Propositions~\ref{prop-01}, \ref{prop-02} and~\ref{prop-04}, as we will show below.
\end{rem}

\begin{proof}[of Corollary~\ref{cor-02}.]
The main problem in the proof comes from the presence of unitors for $\CATC$ (with trivial
unitors, the statement is trivial). The presence of associators of $\CATC$ complicates only
the exposition of the proof, but it does not introduce additional problems.
So for simplicity of exposition we suppose that $\CATC$ has trivial associators.\\

\textbf{Step 1.}
Let us denote as in \eqref{eq-70} the fixed choices in the set $\CW$, needed
for composition of $\underline{f}$, $\underline{g}$ and $\underline{h}$,
so that we have identities \eqref{eq-60} and
\eqref{eq-39}. Moreover, let us suppose that the set of fixed choices $\CW$
gives data as in the upper part of the following diagrams (starting from the ones on the left),
with $\overline{\operatorname{u}}^1$, $\overline{\operatorname{u}}^2$,
$\overline{\operatorname{v}}^1$ and $\overline{\operatorname{u}}^3$ in
$\SETW$, and $\overline{\delta}$, $\overline{\sigma}$, $\overline{\xi}$ and $\overline{\eta}$
invertible:

\begin{gather}
\nonumber 

\end{gather}
so that by~\cite[\S~2.2]{Pr} one has

\begin{gather*}
%
\end{gather*}

\textbf{Step 2.} In this long step we are going to construct a set of data in
$\CATC$; such data will be used in order to provide representatives for the
various $2$-morphisms appearing in the claim of Corollary~\ref{cor-02}, in such a way that
it will be simple to compute
their composition in \eqref{eq-91}, and to prove that it coincides with
$\Theta^{\CATC,\SETW}_{\underline{h},\underline{g},\underline{f}}$.\\

\textbf{Step 2a.} Using \bfTwo{} we get that $\overline{\operatorname{u}}^1\circ
\overline{\operatorname{u}}^2$ belongs to $\SETW$, so by \bfThree{} there is
a set of data as in the upper part of the following diagram,
with $\operatorname{r}^1$ in $\SETW$ and $\zeta^1$ invertible:

\begin{equation}\label{eq-131}

\end{equation}

Using \bfFourA{} and \bfFourB, there are an object $E^2$,
a morphism $\operatorname{r}^3:E^2\rightarrow E^1$ in $\SETW$ and an
invertible $2$-morphism
$\varepsilon^1:f^1\circ\operatorname{u}^2\circ\operatorname{r}^1\circ\operatorname{r}^3
\Rightarrow\overline{f}^1\circ\overline{\operatorname{u}}^2\circ\operatorname{r}^2
\circ\operatorname{r}^3$ in $\CATC$, such that $\operatorname{v}\circ\,\varepsilon^1$
is equal to the following composition:

\[
%
\]

This implies that $\delta^{-1}\circ\operatorname{u}^2\circ\operatorname{r}^1\circ
\operatorname{r}^3$ is equal to the following composition:

\begin{equation}\label{eq-84}
%
\end{equation}

\textbf{Step 2b.}
Now we use again \bfFourA{} and \bfFourB{} in order to get an object
$E^3$, a morphism $\operatorname{r}^4:E^3\rightarrow E^2$ in $\SETW$ and an invertible $2$-morphism
$\varepsilon^2:l\circ\operatorname{r}^1\circ\operatorname{r}^3\circ\operatorname{r}^4
\Rightarrow\overline{l}\circ\operatorname{r}^2\circ\operatorname{r}^3\circ\operatorname{r}^4$,
such that $\operatorname{w}\circ\,\varepsilon^2$ is equal to the following composition:

\[
%
\]

Therefore, $\sigma^{-1}\circ\operatorname{r}^1\circ\operatorname{r}^3\circ\operatorname{r}^4$
is equal to the following composition:

\begin{equation}\label{eq-85}
%
\end{equation}

\textbf{Step 2c.}
We use \bfThree{} in order to get data as in the upper part of the following
diagram, with $\operatorname{r}^5$ in $\SETW$ and $\zeta^2$ invertible:

\[
%
\]

Using \bfFourA{} and \bfFourB, there are an object $E^5$,
a morphism $\operatorname{r}^7:E^5\rightarrow E^4$ in $\SETW$ and an
invertible $2$-morphism
$\varepsilon^3:\operatorname{v}^1\circ f^2\circ\operatorname{r}^5\circ\operatorname{r}^7
\Rightarrow\overline{\operatorname{v}}^1\circ\overline{f}^2\circ\operatorname{r}^6
\circ\operatorname{r}^7$ in $\CATC$, such that $\operatorname{v}\circ\,\varepsilon^3$
is equal to the following composition:

\[
%
\]

Therefore, $\eta\circ\operatorname{r}^5\circ\operatorname{r}^7$ is equal to the following
composition:

\begin{equation}\label{eq-87}
%
\end{equation}

\textbf{Step 2d.}
Using \bfTwo{} we get that $\operatorname{u}^3\circ\operatorname{r}^5\circ
\operatorname{r}^7$ belongs to $\SETW$; so by \bfThree{} there are
data as in the upper part of the following diagram, with $\operatorname{r}^8$
in $\SETW$ and $\zeta^3$ invertible:

\[
%
\]

Then we use \bfFourA{} and \bfFourB{} in order to get an object
$E^7$, a morphism $\operatorname{r}^{10}:E^7\rightarrow E^6$
in $\SETW$ and an invertible $2$-morphism
$\varepsilon^4:\overline{f}^1\circ\overline{\operatorname{u}}^2\circ
\operatorname{r}^2\circ\operatorname{r}^3
\circ\operatorname{r}^4\circ\operatorname{r}^8\circ\operatorname{r}^{10}\Rightarrow
\overline{\operatorname{v}}^1\circ\overline{f}^2\circ\operatorname{r}^6\circ\operatorname{r}^7
\circ\operatorname{r}^9\circ\operatorname{r}^{10}$
in $\CATC$, such that $\operatorname{v}\circ\,\varepsilon^4$ is equal to the
following composition:

\begin{equation}\label{eq-88}

\end{equation}

\textbf{Step 2e}.
Now we use \bfFourA{} and \bfFourB{} in order to get an object
$E^8$, a morphism $\operatorname{r}^{11}:E^8\rightarrow E^7$
in $\SETW$ and an invertible $2$-morphism
$\varepsilon^5:g^1\circ f^2\circ\operatorname{r}^5\circ\operatorname{r}^7\circ\operatorname{r}^9
\circ\operatorname{r}^{10}\circ\operatorname{r}^{11}\Rightarrow\overline{g}^1
\circ\overline{f}^2\circ\operatorname{r}^6\circ\operatorname{r}^7\circ\operatorname{r}^9
\circ\operatorname{r}^{10}\circ\operatorname{r}^{11}$,
such that $\operatorname{w}\circ\,\varepsilon^5$ is equal to the following composition:

\[

\]

Therefore $\xi\circ f^2\circ\operatorname{r}^5\circ\operatorname{r}^7\circ\operatorname{r}^9
\circ\operatorname{r}^{10}\circ\operatorname{r}^{11}$ is equal to the following composition:

\begin{equation}\label{eq-86}
%
\end{equation}

\textbf{Step 2f.}
We use again \bfFourA{} and \bfFourB{} in order to get an object
$A^4$, a morphism $\operatorname{r}^{12}:A^4\rightarrow
E^8$ in $\SETW$ and an invertible $2$-morphism
$\rho:l\circ\operatorname{r}^1\circ\operatorname{r}^3\circ\operatorname{r}^4\circ
\operatorname{r}^8\circ\operatorname{r}^{10}\circ\operatorname{r}^{11}\circ\operatorname{r}^{12}
\Rightarrow g^1\circ f^2\circ\operatorname{r}^5\circ\operatorname{r}^7\circ\operatorname{r}^9
\circ\operatorname{r}^{10}\circ\operatorname{r}^{11}\circ\operatorname{r}^{12}$,
such that $\operatorname{w}\circ\,\rho$ is equal to the following composition:

\begin{equation}\label{eq-89}

\end{equation}

\textbf{Step 3.} Now we define the following set of morphisms and $2$-morphisms:

\begin{gather}
\nonumber \operatorname{u}^4:=\operatorname{r}^1\circ\operatorname{r}^3\circ\operatorname{r}^4
 \circ\operatorname{r}^8\circ\operatorname{r}^{10}\circ\operatorname{r}^{11}\circ
 \operatorname{r}^{12}:\,A^4\longrightarrow A^2, \\
\nonumber \operatorname{u}^5:=\operatorname{r}^5\circ\operatorname{r}^7\circ\operatorname{r}^9
 \circ\operatorname{r}^{10}\circ\operatorname{r}^{11}\circ\operatorname{r}^{12}:
 \,A^4\longrightarrow A^3, \\
\label{eq-132} \gamma:=\zeta^3\circ\operatorname{r}^{10}\circ\operatorname{r}^{11}\circ
 \operatorname{r}^{12}:\,\operatorname{u}^1\circ\operatorname{u}^2\circ\operatorname{u}^4
 \Longrightarrow\operatorname{u}^3\circ\operatorname{u}^5.
\end{gather}

Moreover, we define $\omega:f^1\circ\operatorname{u}^2\circ\operatorname{u}^4
\Rightarrow\operatorname{v}^1\circ f^2\circ
\operatorname{u}^5$ as the following composition:

\begin{equation}\label{eq-15}
\begin{tikzpicture}[xscale=4.4,yscale=-1.5]
    \node (A0_2) at (2, 0) {$E^2$};
    \node (K1_1) at (3, 0) {$A^1$};
    \node (A0_3) at (2.5, 0) {$A^2$};
    \node (A1_0) at (0.5, 1) {$A^4$};
    \node (A1_1) at (1, 1) {$E^7$};
    \node (A1_3) at (3, 1) {$B'$.};
    \node (A2_2) at (2, 2) {$E^5$};
    \node (A2_3) at (3, 2) {$B^2$};
    \node (K0_0) at (2.5, 2) {$A^3$};
    
    \node (A0_1) at (2.75, 0.3) {$\Downarrow\,\varepsilon^1$};
    \node (A1_2) at (1.7, 1) {$\Downarrow\,\varepsilon^4$};
    \node (A2_1) at (2.8, 1.65) {$\Downarrow\,(\varepsilon^3)^{-1}$};

    \path (A1_1) edge [->]node [auto] {$\scriptstyle{\operatorname{r}^4
      \circ\operatorname{r}^8\circ\operatorname{r}^{10}}$} (A0_2);
    \path (A2_3) edge [->]node [auto,swap] {$\scriptstyle{\operatorname{v}^1}$} (A1_3);
    \path (A1_0) edge [->]node [auto] {$\scriptstyle{\operatorname{r}^{11}
      \circ\operatorname{r}^{12}}$} (A1_1);
    \path (K1_1) edge [->]node [auto] {$\scriptstyle{f^1}$} (A1_3);
    \path (A0_3) edge [->]node [auto] {$\scriptstyle{\operatorname{u}^2}$} (K1_1);
    \path (A2_2) edge [->]node [auto,swap] {$\scriptstyle{\operatorname{r}^5
      \circ\operatorname{r}^7}$} (K0_0);
    \path (K0_0) edge [->]node [auto,swap] {$\scriptstyle{f^2}$} (A2_3);
    \path (A2_2) edge [->]node [auto] {$\scriptstyle{\overline{\operatorname{v}}^1
      \circ\overline{f}^2\circ\operatorname{r}^6\circ\operatorname{r}^7}$} (A1_3);
    \path (A1_1) edge [->]node [auto,swap] {$\scriptstyle{\operatorname{r}^9
      \circ\operatorname{r}^{10}}$} (A2_2);
    \path (A0_2) edge [->]node [auto,swap] {$\scriptstyle{\overline{f}^1
     \circ\overline{\operatorname{u}}^2\circ\operatorname{r}^2\circ\operatorname{r}^3}$} (A1_3);
    \path (A0_2) edge [->]node [auto] {$\scriptstyle{\operatorname{r}^1
      \circ\operatorname{r}^3}$} (A0_3);
\end{tikzpicture}
\end{equation}

By construction of $\operatorname{u}^4$ and $\operatorname{u}^5$, we have that $\rho$
is defined from $l\circ\operatorname{u}^4$ to $g^1\circ f^2\circ\operatorname{u}^5$.
We claim that the set of data

\[A^4,\,\operatorname{u}^4,\,\operatorname{u}^5,\,\gamma,\,\omega,\,\rho\]
satisfies conditions \F{1} -- \F{3} for the
computation of the associator
$\Theta^{\CATC,\SETW}_{\underline{h},\underline{g},\underline{f}}$.
$\operatorname{u}^4$ belongs to $\SETW$ by construction and \bfTwo,
and $\gamma$ is invertible because $\zeta^3$ is so, hence condition
\F{1} holds.
In order to prove \F{2}, we replace in \eqref{eq-61} the
definition of $\operatorname{u}^4$, $\operatorname{u}^5$ and $\gamma$ (see \eqref{eq-132}),
the expression for $\delta^{-1}\circ
\operatorname{u}^2\circ\operatorname{r}^1\circ\operatorname{r}^3$ (obtained
in diagram \eqref{eq-84}) and the expression for
$\eta\,\circ\operatorname{r}^5\circ\operatorname{r}^7$ 
(obtained in diagram \eqref{eq-87}).
After simplifying the term $\pi_f$ (coming from \eqref{eq-84}) with its inverse
(coming from \eqref{eq-87}), we get that the composition in
\eqref{eq-61} is equal to the following one:

\begin{equation}\label{eq-90}

\end{equation}

Using \eqref{eq-88} and \eqref{eq-15}, we get that \eqref{eq-90} (hence \eqref{eq-61}) coincides
with $\operatorname{v}\circ\,\omega$. This implies that \F{2} holds.
In order to prove \F{3},
in \eqref{eq-62} we replace $\omega$ with its definition in \eqref{eq-15},
$\operatorname{u}^4$ and $\operatorname{u}^5$ with their definition in \eqref{eq-132}, $\sigma^{-1}
\circ\operatorname{r}^1\circ\operatorname{r}^3\circ\operatorname{r}^4$ with \eqref{eq-85},
and $\xi\circ f^2\circ\operatorname{r}^5\circ\operatorname{r}^7\circ\operatorname{r}^9\circ
\operatorname{r}^{10}\circ\operatorname{r}^{11}$ with \eqref{eq-86}.
After simplifying the terms $\pi_g$, $\varepsilon^1$ and $\varepsilon^3$ with their inverses,
we get that the composition
in \eqref{eq-62} coincides with \eqref{eq-89}, hence with $\operatorname{w}\circ\,\rho$.
Therefore also condition \F{3} holds. So using Proposition~\ref{prop-01} we conclude
that the associator $\Theta^{\CATC,\SETW}_{\underline{h},\underline{g},\underline{f}}$
is represented by the following diagram:

\begin{equation}\label{eq-97}
\begin{tikzpicture}[xscale=2.7,yscale=-0.9]
    \node (A0_2) at (2, 0) {$A^2$};
    \node (A2_2) at (2, 2) {$A^4$};
    \node (A2_4) at (4, 2) {$D$.};
    \node (A2_0) at (0, 2) {$A$};
    \node (A4_2) at (2, 4) {$A^3$};

    \node (A2_3) at (3.1, 2) {$\Downarrow\,h\circ\rho$};
    \node (A2_1) at (1.2, 2) {$\Downarrow\,\operatorname{u}\circ\,\gamma$};
    
    \node (B1_1) at (2.42, 0.95) {$\scriptstyle{\operatorname{r}^1\circ
      \operatorname{r}^3\circ\operatorname{r}^4\circ\operatorname{r}^8\circ}$};
    \node (B2_2) at (2.38, 1.3) {$\scriptstyle{\circ\operatorname{r}^{10}
      \circ\operatorname{r}^{11}\circ\operatorname{r}^{12}}$};
    \node (B3_3) at (2.32, 2.7) {$\scriptstyle{\operatorname{r}^5
      \circ\operatorname{r}^7\circ\operatorname{r}^9\circ}$};
    \node (B4_4) at (2.38, 3.05) {$\scriptstyle{\circ\operatorname{r}^{10}
      \circ\operatorname{r}^{11}\circ\operatorname{r}^{12}}$};
    
    \path (A4_2) edge [->]node [auto,swap] {$\scriptstyle{h\circ
      g^1\circ f^2}$} (A2_4);
    \path (A0_2) edge [->]node [auto] {$\scriptstyle{h\circ l}$} (A2_4);
    \path (A4_2) edge [->]node [auto]
      {$\scriptstyle{\operatorname{u}\circ\operatorname{u}^3}$} (A2_0);
    \path (A0_2) edge [->]node [auto,swap] {$\scriptstyle{\operatorname{u}\circ
      \operatorname{u}^1\circ\operatorname{u}^2}$} (A2_0);
    \path (A2_2) edge [->]node [auto,swap] {} (A0_2);
    \path (A2_2) edge [->]node [auto] {} (A4_2);
\end{tikzpicture}
\end{equation}

\textbf{Step 4.} Now we want to compute also the associator
$\Theta^{\CATC,\SETW}_{\underline{h}',\underline{g}',\underline{f}'}$.
For that, we define the following set of data:

\begin{gather}
\nonumber \overline{\operatorname{u}}^4:=\operatorname{r}^2\circ\operatorname{r}^3\circ
 \operatorname{r}^4\circ
 \operatorname{r}^8\circ\operatorname{r}^{10}\circ\operatorname{r}^{11}\circ
 \operatorname{r}^{12}:\,A^4\longrightarrow\overline{A}^2, \\
\nonumber \overline{\operatorname{u}}^5:=\operatorname{r}^6\circ
 \operatorname{r}^7\circ\operatorname{r}^9\circ\operatorname{r}^{10}
 \circ\operatorname{r}^{11}\circ\operatorname{r}^{12}:\,A^4\longrightarrow
 \overline{A}^3, \\
\label{eq-25} \overline{\omega}:=\varepsilon^4\circ \operatorname{r}^{11}
 \circ\operatorname{r}^{12}:\,
 \overline{f}^1\circ\overline{\operatorname{u}}^2\circ
 \overline{\operatorname{u}}^4\Longrightarrow\overline{\operatorname{v}}^1\circ
 \overline{f}^2\circ\overline{\operatorname{u}}^5.
\end{gather}

Moreover, we define
$\overline{\gamma}:\overline{\operatorname{u}}^1\circ\overline{\operatorname{u}}^2\circ
\overline{\operatorname{u}}^4\Rightarrow
\overline{\operatorname{u}}^3\circ\overline{\operatorname{u}}^5$ as the following composition

\begin{equation}\label{eq-22}

\end{equation}
and $\overline{\rho}:\overline{l}\circ\overline{\operatorname{u}}^4\Rightarrow\overline{g}^1\circ
\overline{f}^2\circ\overline{\operatorname{u}}^5$ as the following composition:

\begin{equation}\label{eq-38}
%
\end{equation}

We claim that the set of data

\begin{equation}\label{eq-94}
A^4,\,\overline{\operatorname{u}}^4,\,\overline{\operatorname{u}}^5,\,\overline{\gamma},\,
\overline{\omega},\,\overline{\rho}
\end{equation}
satisfies conditions \F{1} -- \F{3} for the computation of the
associator $\Theta^{\CATC,\SETW}_{\underline{h}',\underline{g}',\underline{f}'}$,
i.e., the same conditions stated in Proposition~\ref{prop-01}, but with \eqref{eq-65}
replaced by \eqref{eq-81}, and with the four diagrams
appearing in \eqref{eq-70} replaced by the four diagrams of \eqref{eq-123}.
By construction $\operatorname{u}^1$, $\operatorname{u}^2$ and $\operatorname{r}^1$
belong to $\SETW$. So using \eqref{eq-131} together with \bfTwo{}
and \bfFive, we get that $\overline{\operatorname{u}}^1\circ
\overline{\operatorname{u}}^2\circ\operatorname{r}^2$ also belongs to $\SETW$. Moreover,
by construction $\operatorname{r}^3$, $\operatorname{r}^4$, $\operatorname{r}^8$,
$\operatorname{r}^{10}$, $\operatorname{r}^{11}$ and $\operatorname{r}^{12}$ belong to
$\SETW$. So by \bfTwo{} also the morphism
$\overline{\operatorname{u}}^1\circ\overline{\operatorname{u}}^2\circ
\overline{\operatorname{u}}^4=\overline{\operatorname{u}}^1\circ\overline{\operatorname{u}}^2
\circ\operatorname{r}^2\circ\operatorname{r}^3\circ\operatorname{r}^4\circ\operatorname{r}^8
\circ\operatorname{r}^{10}\circ\operatorname{r}^{11}\circ\operatorname{r}^{12}$
belongs to $\SETW$; moreover $\overline{\gamma}$ is invertible because all the
$2$-morphisms in \eqref{eq-22} are invertible by construction, so \F{1} holds.
Using the definition of $\overline{\omega}$ in \eqref{eq-25}, we get
that $\operatorname{v}\circ\id_{B'}\circ\,\overline{\omega}$ is equal to the following
composition: 

\[

\]

In the diagram above we replace $\operatorname{v}\circ\,\varepsilon^4$ with diagram
\eqref{eq-88}, then we simplify all the terms of the form
$\pi_{\operatorname{v}}$ and $\pi_{\operatorname{v}}^{-1}$.
Using \eqref{eq-22}, we get that $\operatorname{v}\circ\id_{B'}
\circ\,\overline{\omega}$ is equal to the following composition:

\[
%
\]

This proves that condition \F{2} holds for the set of data in \eqref{eq-94}.
Now we want to prove also \F{3}. If we use the definition of
$\overline{\rho}$ in \eqref{eq-38}, then 
$\operatorname{w}\circ\id_{C'}\circ\,\overline{\rho}$ is equal to the following composition:

\[
%
\]

In the diagram above we replace the term $\operatorname{w}\circ\,\rho$ with diagram
\eqref{eq-89}, and we simplify $\varepsilon^2$, $\varepsilon^5$ and all the terms of
the form $\pi_{\operatorname{w}}$ with their inverses. Then we get that
$\operatorname{w}\circ\id_{C'}\circ\,\overline{\rho}$
coincides also with the following composition

\[
%
\]
so \F{3} holds. Therefore, using Proposition~\ref{prop-01} with the data of
\eqref{eq-81}, \eqref{eq-123} and \eqref{eq-94}, we get that the associator
$\Theta^{\CATC,\SETW}_{\underline{h}',\underline{g}',\underline{f}'}$ is represented
by the following diagram:

\begin{equation}\label{eq-96}
%
\end{equation}

\textbf{Step 5.}
Until now we have computed the two associators appearing in the claim of Corollary~\ref{cor-02}.
Using Propositions~\ref{prop-02} and~\ref{prop-04} it is not difficult to prove that the 
$2$-morphism

\[\chi(\underline{h})^{-1}\circ\Big(\chi(\underline{g})^{-1}\circ\chi(\underline{f})^{-1}\Big)
:\,\underline{h}\circ\Big(\underline{g}\circ\underline{f}\Big)\Longrightarrow
\underline{h}'\circ\Big(\underline{g}'\circ\underline{f}'\Big)\]
appearing in the upper part of \eqref{eq-91} is represented by the diagram

\begin{equation}\label{eq-98}

\end{equation}
where $\alpha^1$ and $\beta^1$ are the following compositions:

\begin{gather}
\nonumber %
\end{gather}

\textbf{Step 6.}
Using again Propositions~\ref{prop-02} and~\ref{prop-04}, the composition

\[\Big(\chi(\underline{h})\circ\chi(\underline{g})\Big)\circ\chi(\underline{f}):\,
\Big(\underline{h}'\circ\underline{g}'\Big)\circ\underline{f}'\Longrightarrow
\Big(\underline{h}\circ\underline{g}\Big)\circ\underline{f}\]
appearing in the lower part of \eqref{eq-91} is represented by the diagram

\begin{equation}\label{eq-99}
%
\end{equation}
where $\alpha^2$ and $\beta^2$ are the following compositions:

\begin{gather}
\nonumber %
\end{gather}

\textbf{Step 7.} Now we have to compute \eqref{eq-91}, i.e., the vertical composition of:

\begin{enumerate}[(a)]
 \item $\chi(\underline{h})^{-1}\circ(\chi(\underline{g})^{-1}\circ\chi(\underline{f})^{-1})$,
  represented by \eqref{eq-98};
 \item $\Theta^{\CATC,\SETW}_{\underline{h}',\underline{g}',\underline{f}'}$, represented by
  \eqref{eq-96};
 \item $(\chi(\underline{h})\circ\chi(\underline{g}))\circ\chi(\underline{f})$, represented
  by \eqref{eq-99}.
\end{enumerate}

We compose (a) and (b) using Proposition~\ref{prop-02}: since the map $A^4\rightarrow
\overline{A}^2$ in \eqref{eq-98} coincides with the map $A^4\rightarrow \overline{A}^2$
in \eqref{eq-96}, we can choose

\[\left(A^4,\id_{A^4},\id_{A^4},i_{\operatorname{r}^2\circ\operatorname{r}^3\circ
\operatorname{r}^4\circ\operatorname{r}^8\circ\operatorname{r}^{10}\circ
\operatorname{r}^{11}\circ\operatorname{r}^{12}}\right)\]
as the data $(C,\operatorname{t}^1,\operatorname{t}^2,\rho)$ needed in \F{4}.
So the composition of (a) and (b) is represented by the following diagram:

\begin{equation}\label{eq-133}

\end{equation}

Using the same ideas on the pair of diagrams \eqref{eq-133} and \eqref{eq-99}, the composition
of (a), (b) and (c) is represented by the following diagram:

\begin{equation}\label{eq-134}
%
\end{equation}

Above we replace $\overline{\gamma}$ with \eqref{eq-22} and $\overline{\rho}$ with
\eqref{eq-38}. In addition, we replace $\alpha^1$, $\beta^1$, $\alpha^2$ and $\beta^2$
with their definitions in \eqref{eq-146} and \eqref{eq-147}.
Then we simplify the terms of the form $\pi_{\operatorname{u}},\pi_h,\zeta^1,
\zeta^2,\varepsilon^2$ and $\varepsilon^5$ with their inverses. So on the left hand
side of \eqref{eq-134} we get $\operatorname{u}\circ\,\zeta^3\circ\operatorname{r}^{10}
\circ\operatorname{r}^{11}\circ\operatorname{r}^{12}=\operatorname{u}\circ\,\gamma$,
and on the right hand side we get $h\circ\,\rho$. In other
terms, \eqref{eq-134} coincides with \eqref{eq-97}. By construction \eqref{eq-134}
represents the composition of \eqref{eq-91}, and \eqref{eq-97} represents
$\Theta^{\CATC,\SETW}_{\underline{h},\underline{g},\underline{f}}$, so we conclude.
\end{proof}

\section{Appendix - An alternative description of \texorpdfstring{$2$}{2}-morphisms in a bicategory of fractions}
Proposition~\ref{prop-05} suggests an alternative construction of $2$-morphisms
in $\CATC\left[\SETWinv\right]$. This new construction is
equivalent to the original one of~\cite{Pr}, but much simpler. Let us suppose that we want
to define a $2$-morphism from the morphism $\underline{f}^1:=(A^1,\operatorname{w}^1,f^1)$
to the morphism $\underline{f}^2:=(A^2,\operatorname{w}^2,f^2)$ (both defined from $A$ to $B$).
Since the composition of $1$-morphisms in $\CATC\left[\SETWinv\right]$ depends on the set of choices
$\CW$, such choices must be \emph{fixed} before constructing the
bicategory of fractions (and such a bicategory depends on these choices). Since they
are fixed, there is no harm in using them in order to get data $(E,\operatorname{p},
\operatorname{q},\varsigma)$ in $\CATC$ as below, with $\operatorname{p}$ in $\SETW$
and $\varsigma$ invertible:

\begin{equation}\label{eq-101}
\begin{tikzpicture}[xscale=1.5,yscale=-0.8]
    \node (A0_1) at (1, 0) {$E$};
    \node (A1_0) at (0, 2) {$A^1$};
    \node (A1_2) at (2, 2) {$A^2$.};
    \node (A2_1) at (1, 2) {$A$};

    \node (A1_1) at (1, 1) {$\varsigma$};
    \node (B1_1) at (1, 1.4) {$\Rightarrow$};
    
    \path (A1_2) edge [->]node [auto] {$\scriptstyle{\operatorname{w}^2}$} (A2_1);
    \path (A0_1) edge [->]node [auto] {$\scriptstyle{\operatorname{q}}$} (A1_2);
    \path (A1_0) edge [->]node [auto,swap] {$\scriptstyle{\operatorname{w}^1}$} (A2_1);
    \path (A0_1) edge [->]node [auto,swap] {$\scriptstyle{\operatorname{p}}$} (A1_0);
\end{tikzpicture}
\end{equation}

Then we give the following definition:

\begin{definition}\label{def-01}
Let us fix any pair $(\CATC,\SETW)$ satisfying conditions \emphatic{\bf},
and any set of choices $\CW$; let $\CATC\left[\SETWinv\right]$
be the bicategory of fractions induced by such choices \emphatic{(}see
Theorem~\ref{thm-01}\emphatic{)}.
Then an \emph{almost canonical representative of a $2$-morphism in
$\CATC\left[\SETWinv\right]$} from $\underline{f}^1$ to $\underline{f}^2$
is any triple $(A^3,\operatorname{t},\varphi)$ where $\operatorname{t}:A^3\rightarrow E$
is a morphism in $\SETW$, and $\varphi$ is a $2$-morphism in $\CATC$ from
$f^1\circ\operatorname{p}\circ\operatorname{t}$
to $f^2\circ\operatorname{q}\circ\operatorname{t}$. Given another triple
$(A^{\prime 3},\operatorname{t}':A^{\prime 3}\rightarrow E,\varphi':f^1\circ\operatorname{p}
\circ\operatorname{t}'\Rightarrow f^2\circ\operatorname{q}\circ\operatorname{t}')$,
we say that it is \emph{equivalent} to the previous one, and we write
$(A^3,\operatorname{t},\varphi)\sim(A^{\prime 3},\operatorname{t}',\varphi')$
if and only if there are data $(A^4,\operatorname{u},\operatorname{u}',\sigma)$ in $\CATC$
as follows

\[

\]
such that:

\begin{itemize}
 \item $\operatorname{u}$ belongs to $\SETW$;
 \item $\sigma$ is invertible;
 \item the compositions of the following two diagrams are equal:
\end{itemize}

\begin{equation}\label{eq-139}
%
\end{equation}

A \emph{$2$-morphism in $\CATC\left[\SETWinv\right]$} is any class
of equivalence of data as above, denoted by

\[\Big[A^3,\operatorname{t},\varphi\Big]:\,\Big(A^1,\operatorname{w}^1,f^1\Big)
\Longrightarrow\Big(A^2,\operatorname{w}^2,f^2\Big).\]
\end{definition}

As you can see, this definition is much shorter than the original one (that we
recalled in the Introduction). 
%
%
Later we will prove in full details that:

\begin{enumerate}[(a)]
 \item $\sim$ is a actually an equivalence relation (so Definition~\ref{def-01}
  is well-posed);
 \item there is a natural bijection between $2$-morphisms according to Definition~\ref{def-01}
  and $2$-morphisms according to the original definition of~\cite{Pr}, induced by associating
  to any triple $(A^3,\operatorname{t},\varphi)$ the equivalence class of the following diagram:
\end{enumerate}

\begin{equation}\label{eq-141}
\begin{tikzpicture}[xscale=1.6,yscale=-0.6]
    \node (A0_2) at (2, 0) {$A^1$};
    \node (A2_2) at (2, 2) {$A^3$};
    \node (A2_0) at (0, 2) {$A$};
    \node (A2_4) at (4, 2) {$B$.};
    \node (A4_2) at (2, 4) {$A^2$};

    \node (A2_3) at (2.8, 2) {$\Downarrow\,\varphi$};
    \node (A2_1) at (1.2, 2) {$\Downarrow\,\varsigma\circ\operatorname{t}$};

    \path (A4_2) edge [->]node [auto,swap] {$\scriptstyle{f^2}$} (A2_4);
    \path (A0_2) edge [->]node [auto] {$\scriptstyle{f^1}$} (A2_4);
    \path (A2_2) edge [->]node [auto,swap] {$\scriptstyle{\operatorname{p}
      \circ\operatorname{t}}$} (A0_2);
    \path (A2_2) edge [->]node [auto] {$\scriptstyle{\operatorname{q}
      \circ\operatorname{t}}$} (A4_2);
    \path (A4_2) edge [->]node [auto] {$\scriptstyle{\operatorname{w}^2}$} (A2_0);
    \path (A0_2) edge [->]node [auto,swap] {$\scriptstyle{\operatorname{w}^1}$} (A2_0);
\end{tikzpicture}
\end{equation}


The result above can be rephrased by saying that each $2$-morphism in
$\CATC\left[\SETWinv\right]$ (defined according to~\cite{Pr})
has a representative given by a diagram of type \eqref{eq-141} for a triple
$(A^3,\operatorname{t},\varphi)$ that is ``almost canonical'' (having fixed the
set of choices $\CW$ that completely determine the structure of
$\CATC\left[\SETWinv\right]$, see Theorem~\ref{thm-01}).
``Almost canonical'' here refers to the
fact that the triple $(A^3,\operatorname{t},\varphi)$ 
inducing a given $2$-morphism is almost unique and varies in a much smaller
set if compared to the set of diagrams as \eqref{eq-16}. We cannot say ``canonical''
since sometimes there is more than one such triple inducing the same $2$-morphism (this is
why we need to use the equivalence relation given in Definition~\ref{def-01}).\\

The description of $2$-cells of $\CATC\left[\SETWinv\right]$ in terms of
``almost canonical'' triples is simpler than the original one in~\cite{Pr},
but does not interact well with vertical and horizontal composition.
%
%
Therefore we are just mentioning it here in the Appendix, and we are not using it
in the rest of the paper.\\

In order to prove (a), let us denote by $\mathcal{M}(A^3,\operatorname{t},\varphi)$
the diagram \eqref{eq-141}. Then we have:

\begin{lem}\label{lem-07}
Let us fix any pair of triples $(A^3,\operatorname{t},\varphi)$ and $(A^{\prime 3},
\operatorname{t}',\varphi')$. Then the following facts are equivalent:

\begin{enumerate}[\emphatic{(}i\emphatic{)}]
 \item $(A^3,\operatorname{t},\varphi)\sim(A^{\prime 3},\operatorname{t}',\varphi')$;
 \item $\mathcal{M}(A^3,\operatorname{t},\varphi)$ and $\mathcal{M}(A^{\prime 3},
  \operatorname{t}',\varphi')$ are equivalent diagrams, i.e., they represent
  the same $2$-morphism of $\CATC\left[\SETWinv\right]$ according to~\cite{Pr}.
\end{enumerate}

This implies that $\sim$ is an equivalence relation.
\end{lem}

\begin{proof}
Let us suppose that (i) holds. Then the following diagrams

\begin{equation}\label{eq-78}
\begin{tikzpicture}[xscale=1.6,yscale=-0.6]
    \node (A0_2) at (2, 0) {$A^1$};
    \node (A2_2) at (2, 2) {$A^3$};
    \node (A2_0) at (0, 2) {$A$};
    \node (A2_4) at (4, 2) {$B$,};
    \node (A4_2) at (2, 4) {$A^2$};

    \node (A2_3) at (2.8, 2) {$\Downarrow\,\varphi$};
    \node (A2_1) at (1.2, 2) {$\Downarrow\,\varsigma\circ\operatorname{t}$};

    \path (A4_2) edge [->]node [auto,swap] {$\scriptstyle{f^2}$} (A2_4);
    \path (A0_2) edge [->]node [auto] {$\scriptstyle{f^1}$} (A2_4);
    \path (A2_2) edge [->]node [auto,swap] {$\scriptstyle{\operatorname{p}
      \circ\operatorname{t}}$} (A0_2);
    \path (A2_2) edge [->]node [auto] {$\scriptstyle{\operatorname{q}
      \circ\operatorname{t}}$} (A4_2);
    \path (A4_2) edge [->]node [auto] {$\scriptstyle{\operatorname{w}^2}$} (A2_0);
    \path (A0_2) edge [->]node [auto,swap] {$\scriptstyle{\operatorname{w}^1}$} (A2_0);
    
    \def \z {5}
    
    \node (B0_2) at (2+\z, 0) {$A^1$};
    \node (B2_2) at (2+\z, 2) {$A^{\prime 3}$};
    \node (B2_0) at (0+\z, 2) {$A$};
    \node (B2_4) at (4+\z, 2) {$B$};
    \node (B4_2) at (2+\z, 4) {$A^2$};

    \node (B2_3) at (2.8+\z, 2) {$\Downarrow\,\varphi'$};short
    \node (B2_1) at (1.2+\z, 2) {$\Downarrow\,\varsigma\circ\operatorname{t}'$};

    \path (B4_2) edge [->]node [auto,swap] {$\scriptstyle{f^2}$} (B2_4);
    \path (B0_2) edge [->]node [auto] {$\scriptstyle{f^1}$} (B2_4);
    \path (B2_2) edge [->]node [auto,swap] {$\scriptstyle{\operatorname{p}
      \circ\operatorname{t}'}$} (B0_2);
    \path (B2_2) edge [->]node [auto] {$\scriptstyle{\operatorname{q}
      \circ\operatorname{t}'}$} (B4_2);
    \path (B4_2) edge [->]node [auto] {$\scriptstyle{\operatorname{w}^2}$} (B2_0);
    \path (B0_2) edge [->]node [auto,swap] {$\scriptstyle{\operatorname{w}^1}$} (B2_0);
\end{tikzpicture}
\end{equation}
represent the same $2$-morphism of $\CATC\left[\SETWinv\right]$ according to~\cite{Pr}: it 
suffices to consider $\sigma^1:=\operatorname{p}\circ\,\sigma$, $\sigma^2:=
\operatorname{q}\circ\,\sigma^{\,-1}$, and to verify that the conditions after
diagram \eqref{eq-109} are all satisfied: \eqref{eq-135} is a consequence of the
interchange law, and \eqref{eq-136} is simply \eqref{eq-139}; so (ii) holds.\\

Conversely, let us suppose that (ii) holds, i.e., let us suppose that the diagrams in
\eqref{eq-78} represent the same $2$-morphism according to~\cite{Pr}. Then there are
data $(C^1,\operatorname{r}^1,\operatorname{r}^2,\gamma,\xi)$ in $\CATC$ as follows:

\[

\]
such that:

\begin{enumerate}[(1)]
 \item $\operatorname{w}^1\circ\operatorname{p}\circ\operatorname{t}
   \circ\operatorname{r}^1$ belongs to $\SETW$;
 \item $\gamma$ and $\xi$ are both invertible;
 \item the compositions of the following two diagrams are equal:
 
 \begin{equation}\label{eq-75}
  %
  \end{equation}

 \item the compositions of the following two diagrams are equal:
 
 \begin{equation}\label{eq-80}
  %
  \end{equation}
\end{enumerate}

By hypothesis (see the beginning of the Appendix) $\underline{f}^1=(A^1,
\operatorname{w}^1,f^1)$ is a morphism in $\CATC\left[\SETWinv\right]$, hence
$\operatorname{w}^1$ belongs to $\SETW$. Moreover
$\operatorname{p}$ and $\operatorname{t}$ also belong to $\SETW$
(see \eqref{eq-101} and Definition~\ref{def-01}). Using \bfTwo,
this implies that $\operatorname{w}^1\circ\operatorname{p}\circ
\operatorname{t}$ belongs also to $\SETW$. Then using (1) and Lemma~\ref{lem-11}
there are an object $C^2$ and a morphism $\operatorname{r}^3:C^2\rightarrow C^1$,
such that $\operatorname{r}^1\circ\operatorname{r}^3$ belongs to $\SETW$.\\

Since $\operatorname{p}$ belongs to $\SETW$ by construction, by \bfFourA{}
and \bfFourB{} there are an object $C^3$, a morphism $\operatorname{r}^4:
C^2\rightarrow C^1$ in $\SETW$ and an invertible $2$-morphism $\widetilde{\gamma}:
\operatorname{t}'\circ\operatorname{r}^2\circ\operatorname{r}^3\circ
\operatorname{r}^4\Rightarrow
\operatorname{t}\circ\operatorname{r}^1\circ\operatorname{r}^3
\circ\operatorname{r}^4$, such that:

\begin{equation}\label{eq-93}
\gamma\circ\operatorname{r}^3\circ\operatorname{r}^4=
\operatorname{p}\circ\,\widetilde{\gamma}. 
\end{equation}

Now in diagram \eqref{eq-101} the morphisms $\operatorname{w}^1$, $\operatorname{w}^2$ 
and $\operatorname{p}$ all belong to $\SETW$. Using \bfTwo{} and
\bfFive{} we get that also $\operatorname{w}^2\circ\operatorname{q}$ belongs
to $\SETW$. So by Lemma~\ref{lem-06} applied to $\xi\circ\operatorname{r}^3\circ
\operatorname{r}^4$, there are an object $C^4$, a morphism
$\operatorname{r}^5:C^4\rightarrow C^3$ in $\SETW$ and an invertible $2$-morphism
$\widetilde{\xi}:\operatorname{t}\circ\operatorname{r}^1\circ\operatorname{r}^3\circ
\operatorname{r}^4\circ\operatorname{r}^5\Rightarrow\operatorname{t}'\circ\operatorname{r}^2
\circ\operatorname{r}^3\circ\operatorname{r}^4\circ\operatorname{r}^5$, such that:

\begin{equation}\label{eq-12}
\xi\circ\operatorname{r}^3\circ\operatorname{r}^4\circ\operatorname{r}^5
=\operatorname{q}\circ\,\widetilde{\xi}. 
\end{equation}

Then using identities \eqref{eq-75}, \eqref{eq-93} and
\eqref{eq-12}, and the fact that $\varsigma$ is invertible, we conclude
that $(\operatorname{w}^1\circ\operatorname{p})\circ\,\widetilde{\xi}=
(\operatorname{w}^1\circ\operatorname{p})\circ\,\widetilde{\gamma}^{-1}
\circ\operatorname{r}^5$. Using Lemma~\ref{lem-03}
we conclude that there are an object $A^4$ and a morphism $\operatorname{r}^6:
A^4\rightarrow C^4$ in $\SETW$, such that $\widetilde{\xi}\circ\operatorname{r}^6
=\widetilde{\gamma}^{-1}\circ\operatorname{r}^5\circ\operatorname{r}^6$. So we have:

\[\gamma\circ\operatorname{r}^3\circ\operatorname{r}^4\circ\operatorname{r}^5
\circ\operatorname{r}^6\stackrel{\eqref{eq-93}}{=}
\operatorname{p}\circ\,\widetilde{\gamma}\circ\operatorname{r}^5
\circ\operatorname{r}^6\]
and

\[\xi\circ\operatorname{r}^3\circ\operatorname{r}^4\circ\operatorname{r}^5
\circ\operatorname{r}^6\stackrel{\eqref{eq-12}}{=}\operatorname{q}\circ\,
\widetilde{\xi}\circ\operatorname{r}^6=\operatorname{q}\circ\,
\widetilde{\gamma}^{-1}\circ\operatorname{r}^5\circ\operatorname{r}^6.\]

Using \eqref{eq-80} and the previous two identities, we conclude that the following composition

\begin{equation}\label{eq-143}

\end{equation}
is equal to

\begin{equation}\label{eq-144}
\varphi'\circ\operatorname{r}^2\circ\operatorname{r}^3\circ\operatorname{r}^4
\circ\operatorname{r}^5\circ\operatorname{r}^6.
\end{equation}

Then we define:

\[\operatorname{u}:=\operatorname{r}^1\circ\operatorname{r}^3\circ\operatorname{r}^4
\circ\operatorname{r}^5\circ\operatorname{r}^6:\,A^4\longrightarrow A^3,\quad
\operatorname{u}':=\operatorname{r}^2\circ\operatorname{r}^3
\circ\operatorname{r}^4\circ\operatorname{r}^5\circ\operatorname{r}^6
:\,A^4\longrightarrow A^{\prime 3}\]
and $\sigma:=\widetilde{\gamma}\circ\operatorname{r}^5\circ\operatorname{r}^6:
\operatorname{t}'\circ\operatorname{u}'
\Rightarrow\operatorname{t}\circ\operatorname{u}$.
By construction $\operatorname{r}^1\circ\operatorname{r}^3$,
$\operatorname{r}^4$, $\operatorname{r}^5$ and $\operatorname{r}^6$ all belong
to $\SETW$; so by \bfTwo{} $\operatorname{u}$ also belongs to $\SETW$.
Then the identity of \eqref{eq-143} with \eqref{eq-144}
implies that $(A^3,\operatorname{t},\varphi)\sim(A^{\prime 3},
\operatorname{t}',\varphi')$, i.e., (i) holds.
\end{proof}

Since $\sim$ is an equivalence relation, Definition~\ref{def-01} makes sense. Moreover,
we have:

\begin{prop}\label{prop-07}
Let us fix any pair $(\CATC,\SETW)$ satisfying conditions \emphatic{\bf},
and any set of choices $\CW$; let $\CATC\left[\SETWinv\right]$
be the bicategory of fractions induced by such choices \emphatic{(}see
Theorem~\ref{thm-01}\emphatic{)}.
Let us fix any pair of morphisms $\underline{f}^m=(A^m,\operatorname{w}^m,f^m)$ for $m=1,2$ in
$\CATC\left[\SETWinv\right]$, both defined from $A$ to $B$. Let us denote by 
$(E,\operatorname{p},\operatorname{q},\varsigma)$
the fixed choice $\CW$
for the pair $(\operatorname{w}^1,\operatorname{w}^2)$
as in \eqref{eq-101}. Then there is a natural bijection between the set of
$2$-morphisms from $\underline{f}^1$ to $\underline{f}^2$ according to Definition~\ref{def-01}
and the set of such $2$-morphisms according to~\cite{Pr}, given by

\[\mathcal{N}:\,\Big[A^3,\operatorname{t},\varphi\Big]\mapsto\Big[A^3,\operatorname{p}
\circ\operatorname{t},\operatorname{q}\circ\operatorname{t},\varsigma\circ\operatorname{t},
\varphi\Big].\]
\end{prop}

\begin{proof}
First of all, the set map $\mathcal{N}$ is well-defined. Indeed, if we choose another
representative $[A^{\prime 3},\operatorname{t}',\varphi']$ for the class
$[A^3,\operatorname{t},\varphi]$, then its image via $\mathcal{N}$ is the same because
of Lemma~\ref{lem-07}. Using again such lemma, we have that $\mathcal{N}$ is injective.\\

In order to prove that $\mathcal{N}$ is surjective, let us fix any $2$-cell $\Phi:
\underline{f}^1\Rightarrow\underline{f}^2$ in $\CATC\left[\SETWinv\right]$. Using
Corollary~\ref{cor-04} there are an object $A^3$ and a morphism $\operatorname{t}:
A^3\rightarrow E$, such that $\operatorname{p}\circ\operatorname{t}$ belongs
to $\SETW$, and a $2$-morphism $\varphi:f^1\circ\operatorname{p}
\circ\operatorname{t}\Rightarrow f^2\circ\operatorname{p}
\circ\operatorname{t}$, such that $\Phi$ is represented by diagram \eqref{eq-145}.
Since both $\operatorname{p}$ and $\operatorname{p}\circ\operatorname{t}$
belong to $\SETW$, then by Lemma~\ref{lem-11} there are an object $A^4$
and a morphism $\operatorname{r}:A^4\rightarrow A^3$, such that
$\operatorname{t}\circ\operatorname{r}$ belongs to $\SETW$. 
Since $\Phi$ is represented by diagram \eqref{eq-145}, then we have:

\[\Phi=\Big[A^3,\operatorname{p}
\circ\operatorname{t},\operatorname{q}\circ\operatorname{t},\varsigma\circ\operatorname{t},
\varphi\Big]=
\Big[A^4,\operatorname{p}\circ\operatorname{t}\circ\operatorname{r},
\operatorname{q}\circ\operatorname{t}\circ\operatorname{r},
\varsigma\circ\operatorname{t}\circ\operatorname{r},
\varphi\circ\operatorname{r}\Big]=
\mathcal{N}\Big(\Big[A^4,\operatorname{t}\circ\operatorname{r},
\varphi\circ\operatorname{r}\Big]\Big),\]
hence $\mathcal{N}$ is surjective.
\end{proof}


\end{document}